\documentclass[11pt, a4paper]{amsart}
\usepackage{amssymb,latexsym,amsmath, graphicx} 
\usepackage{enumerate,enumitem,color}
\usepackage{mathrsfs} 
\usepackage{calligra}
\usepackage[T1]{fontenc}
\usepackage{hyperref}
\usepackage{cleveref}

\usepackage{tikz}
\usepackage{tikz-cd}
\hypersetup{
    colorlinks=true,   
    linkcolor=blue,    
    citecolor=red, 
    urlcolor=magenta   
}

\input{xypic}
\xyoption{all}

\title[Arithmetic of intertwining operators]{An arithmetic property of intertwining operators for real reductive groups}
\author[Raghuram \ \& \  Zou]{A.\,Raghuram \ \ \& \ \  Xin Zou}

\date{\today}      

\address{A.~Raghuram, Department of Mathematics, Fordham University at Lincoln Center,\newline New~York, NY 10023, USA.} 

\email{araghuram@fordham.edu}

\address{Xin Zou, Department of Mathematics, Johns Hopkins University, Baltimore, MD 21218, USA.} 

\email{zxin3@jhu.edu}

\numberwithin{equation}{section}   
\newtheorem{lemma}[equation]{Lemma}
\newtheorem{thm}[equation]{Theorem}
\newtheorem{prop}[equation]{Proposition}
\newtheorem{cor}[equation]{Corollary}

\theoremstyle{definition}
\newtheorem{exam}[equation]{Example} 
\newtheorem{rem}[equation]{Remark}
\newtheorem{defn}[equation]{Definition}

\usepackage{bm}
\usepackage{upgreek}
\makeatletter
\newcommand{\bfgreek}[1]{\bm{\@nameuse{up#1}}}
\makeatother

\let\oldtocsection=\tocsection
\let\oldtocsubsection=\tocsubsection
\let\oldtocsubsubsection=\tocsubsubsection
\renewcommand{\tocsection}[2]{\hspace{0em}\oldtocsection{#1}{#2}}
\renewcommand{\tocsubsection}[2]{\hspace{1em}\oldtocsubsection{#1}{#2}}
\renewcommand{\tocsubsubsection}[2]{\hspace{2em}\oldtocsubsubsection{#1}{#2}}

\begin{document}

\begin{abstract}
Let $G$ be the group of complex points of a connected split reductive group defined over ${\mathbb Q}$. 
We study the standard intertwining operator between representations parabolically induced from essentially tempered cohomological representations of Levi subgroups of maximal parabolic subgroups. Under certain algebraicity, purity, and criticality hypotheses the operator is holomorphic and invertible at a particular point of evaluation. The relative Lie algebra cohomology of these induced representations is one-dimensional in the lowest nonzero degree. Under prescribed rational generators of these one-dimensional spaces, we show that the intertwining operator is a nonzero rational multiple of a certain product of ratios of local $L$-factors. This is a cohomological version of Langlands's calculation of the Gindikin-Karplevic formula. The main theorem provides a local input for the arithmetic study of special values of Langlands--Shahidi $L$-functions.
\end{abstract} 

\maketitle

 \def\R{\mathbb{R}}
\def\C{\mathbb{C}}
\def\Z{\mathbb{Z}}
\def\Q{\mathbb{Q}}
\def\A{\mathbb{A}}
\def\F{\mathbb{F}}
 \def\J{\mathbb{J}}
\newcommand\D{\mathbb{ D}}
 \def\L{\mathbb{L}}
\def\G{\mathbb{G}}
\def\N{\mathbb{N}}
\def\BH{\mathbb{H}}
\newcommand\Qi{\mathbb{Q}(i)}

\newcommand{\vless}{\rotatebox[origin=c]{-90}{$<$}}
\newcommand{\vgreat}{\rotatebox[origin=c]{90}{$<$}}
\newcommand{\vgreater}{\rotatebox[origin=c]{90}{$\leq$}}

\newcommand\Dm{\D_\lambda}
\newcommand\Dmp{\D_{\lambda^\prime}}
 \newcommand\Dim{\D_{\io\lambda}}
 \newcommand\Dimp{\D_{\io\lambda^\prime}}
\newcommand\Dum{\D_{\ul{\lambda}}}
\newcommand\Dium{\D_{\io\ul{\lambda}}} 
\newcommand\DumN{\D_{\ul{\lambda}-N\gamma_P}} 
\newcommand\DiumN{\D_{\io\ul{\lambda}-N\gamma_P}} 
\newcommand\cal{\mathcal}
\newcommand\SMK{{\cal S}^M_{K^M_f}}
\newcommand\tMZl{\tM_{\lambda,\Z}} 
\newcommand\Gm{{\mathbb G}_m}
\newcommand\cA{\cal A}
\newcommand\cC{\cal C}
\newcommand\calL{\cal L}
\newcommand\cO{\cal O}
\newcommand\cU{\cal U}
\newcommand\cK{\cal K}   
\newcommand\cW{\cal W}     
\newcommand\HH{{\cal H}}
\newcommand\cF{\mathcal{F}} 
\newcommand\cG{\mathcal{G}}
\newcommand\cB{\mathcal{B}}
\newcommand\cT{\mathcal{T}}
\newcommand\cS{\mathcal{S}}
\newcommand\cP{\mathcal{P}}
\newcommand\cZ{\mathcal{Z}}
\newcommand\cM{\mathcal{M}}
\newcommand\cV{\mathcal{V}}
\newcommand\caL{\mathcal{L}}
\newcommand\cSL{\mathcal{SL}}
\newcommand\cGL{\mathcal{GL}}

\newcommand\GL{{ \rm  GL}}
\newcommand\U{{ \rm  U}}
\def\SU{{\rm SU}}
\def\S{{\bf S}}
\newcommand\GSp{{\rm GSp}}
\newcommand\GSpin{{\rm GSpin}}
\newcommand\Lie {{ \rm Lie}} 
\newcommand\Spin{{ \rm  Spin}}
\newcommand\SL{{ \rm  SL}}
\newcommand\SO {{ \rm  SO}}
\newcommand{\Sp}{\mathrm{Sp}}
\newcommand\Ad{\text{Ad}}
\newcommand\AI{\text{AI}}
\newcommand\Sym{\text{Sym}}
\newcommand{\dd}{\mathrm{d}}
\def\ringO{\mathcal{O}}
\def\idealP{\mathfrak{P}} 
 \def\g{\mathfrak{g}}
\def\k{\mathfrak{k}}
\def\t{\mathfrak{t}}
\def\z{\mathfrak{z}}
\def\a{\mathfrak{a}}
\def\c{\mathfrak{c}}
\def\n{\mathfrak{n}}
\def\t{\mathfrak{t}}
\def\p{\mathfrak{p}}
\def\l{\mathfrak{l}}
\def\gl{\mathfrak{gl}}
\def\sl{\mathfrak{sl}}
\def\u{\mathfrak{u}} 
\def\m{\mathfrak{m}}
\def\fp{\mathfrak{p}} 
\def\p{\mathfrak{p}}   
\def\r{\mathfrak{r}}
\def\fd{\mathfrak{d}}
\def\fR{\mathfrak{R}}
\def\fI{\mathfrak{I}}
\def\fJ{\mathfrak{J}}
\def\i{\mathfrak{i}}
\def\perm{\mathfrak{S}}
\newcommand\fg{\mathfrak g}
\newcommand\fk{\mathfrak k}
\newcommand\fgK{(\mathfrak{g},K_\infty^0)}
\newcommand\gK{ \mathfrak{g},K_\infty^0 }

\newcommand\ul{\underline} 
\newcommand\tp{{  {\pi}_f}}
\newcommand\tv{ {\pi}_v}
\newcommand\ts{{ {\sigma}_f}}
\newcommand\pts{ {\sigma}^\prime_f}
\newcommand\usf{\ul{\sigma}_f}
\newcommand\pusf{\ul{\sigma}^\prime_f}
\newcommand\usvp{\ul{\sigma}^\prime_v}
\newcommand\usv{\ul{\sigma}_v}
\newcommand\io{{}^\iota}
\newcommand\uls{{\underline{\sigma}}}
\newcommand\ulb{\underline{\lambda}}
\newcommand\ual{\underline{\alpha}}
\newcommand\ulw{\underline{w}}
\newcommand\ulm{\underline{\mu}} 
\newcommand\uln{\underline{\nu}}
\newcommand\ubs{\underline{\s}}
\newcommand\ule{\underline{\eta}}
\newcommand{\trivial}{1\!\!1}
\newcommand\Spec{\hbox{\rm Spec}} 
\newcommand\SGK{\mathcal{S}^G_{K_f}}
\newcommand\SMP{\mathcal{S}^{M_P}}
 \newcommand\SGn{\mathcal{S}^{G_n}}
 \newcommand\SGp{\mathcal{S}^{G_{n^\prime}}}
\newcommand\SMPK{\mathcal{S}^{M_P}_{K_f^{M_P}}}
\newcommand\SMQ{\mathcal{S}^{M_Q}}
\newcommand\SMp{\mathcal{S}^{M }_{K_f^M}}
\newcommand\SMq{\mathcal{S}^{M^\prime}_{K_f^{M^\prime}}}
\newcommand\uSMP{\ul{\mathcal{S}}^{M_P}}
\newcommand\SGnK{\mathcal{S}^{G_n}_{K_f}}
\newcommand\SG{\mathcal{S}^G}
\newcommand\SGKp{\mathcal{S}^G_{K^\prime_f}}
\newcommand\piKK{{ \pi_{K_f^\prime,K_f}}}
\newcommand\piKKpkt{\pi^{\pkt}_{K_f^\prime,K_f}}
\newcommand\BSC{ \bar{\mathcal{S}}^G_{K_f}}
\newcommand\PBSC{\partial\SGK}
\newcommand\pBSC{\partial\SG}
\newcommand\PPBSC{\partial_P\SGK}
\newcommand\PQBSC{\partial_Q\SGK}
\newcommand\ppBSC{\partial_P\mathcal{S}^G}
\newcommand\pqBSC{\partial_Q\mathcal{S}^G}
\newcommand\prBSC{\partial_R\mathcal{S}^G}
\newcommand\bs{\backslash}
 \newcommand\tr{\hbox{\rm tr}}
 \newcommand\ord{\text{ord}}
\newcommand\Tr{\hbox{\rm Tr}}
\newcommand\HK{\mathcal{H}^G_{K_f}}
\newcommand\HKS{\mathcal{H}^G_{K_f,\place}}
\newcommand\HKv{\mathcal{H}^G_{K_v}}
\newcommand\HGS{\mathcal{H}^{G,\place}}
\newcommand\HKp{\mathcal{H}^G_{K_p}}
\newcommand\HKpo{\mathcal{H}^G_{K_p^0}}
\newcommand\ch{{\bf ch}}

\newcommand\M{\mathcal{M}}
\newcommand\Ml{\M_\lambda}
\newcommand\tMl{\tilde{\Ml}}
\newcommand\tM{\widetilde{\mathcal{M}}}
\newcommand\tMZ{\tM_\Z}
\newcommand\tsigma{\ul{\sigma}}
\newcommand\pkt{\bullet}
\newcommand\tH{\widetilde{\mathcal{H}}}
\newcommand\Mot{{\bf M}} 
\newcommand\eff{{\rm eff}}
\newcommand\Aql{A_{\q}(\lambda)}
\newcommand\wl{w\cdot\lambda}
\newcommand\wlp{w^\prime\cdot\lambda} 

\def\w{{\sf w}} 
\def\s{{\sf s}}
\def\e{{\bf e}} 
\def\x{{\tt x}}
\def\y{{\tt y}}
\def\v{{\sf v}}
\def\q{{\sf q}} 
\def\ff{{\bf f}}
\def\bk{{\bf k}}
 
\def\Ext{{\rm Ext}}
\def\Aut{{\rm Aut}}
\def\Hom{{\rm Hom}}
\newcommand\Ind{\mathrm{Ind}}
\def\aInd{{}^{\rm a}\Ind}
\def\aIndPG{\aInd_{\pi_0(P(\R)) \times P(\A_f)}^{\pi_0(G(\R)) \times G(\A_f)}}
\def\aIndQG{\aInd_{\pi_0(Q(\R)) \times Q(\A_f)}^{\pi_0(G(\R)) \times G(\A_f)}}
\def\Gal{{\rm Gal}}
\def\End{{\rm End}} 
\def\cm{{\rm cm}} 
\newcommand\Coh{{\rm Coh}}  
\newcommand\Eis{{\rm Eis}}
\newcommand\Res{\mathrm{Res}}
\newcommand\place{\mathsf{S}}
\newcommand\emb{\mathcal{I}} 
\newcommand\LB{\mathcal{L}}  
\def\Hod{{\mathcal{H}od}}
\def\Crit{{\rm Crit}}
  
 \newcommand\ip{\pi_f\circ \iota} 
\newcommand\Wp{W_{\pi_\infty\times \ip}}
\newcommand\Wpc{W^{\text{cusp}}_{\pi_\infty\times \pi_f\circ\iota}}
\newcommand\Lcusp{  L^2_{\text{cusp}}(G(\Q)\bs G(\A_f)/K_f)}
\newcommand\MiC{\tM_{\iota\circ\lambda,\C} }
\newcommand\miC{\M_{\iota\circ\lambda,\C} }
\newcommand\tpl{{}^\iota}
\newcommand\Id{\rm Id}
\newcommand\Lr{L^{\text{rat}}}
\newcommand\iso{ \buildrel \sim \over\longrightarrow}
\newcommand\us{\ul\sigma}
\newcommand\qvs{q_v^{-z}}
\newcommand\into{\hookrightarrow}
\newcommand\ppfeil[1]{\buildrel #1\over \longrightarrow}
\newcommand\eb{{}^\iota}

\newcommand\ulr{\underline{\rho}}
\def\bfpi{\mathbf{\Pi}}
\def\bfdelta{\mathbf{\Delta}}

\def\sI{\mathscr{I}}
\def\sU{\mathscr{U}}
\def\sJ{\mathscr{J}}
\newcommand\cocha[1]{\mathrm{X}_*(#1)}
\newcommand\cha[1]{\mathrm{X}^*(#1)}

\def\bfG{\mathbf{G}}
\def\bfB{\mathbf{B}}
\def\bfT{\mathbf{T}}
\def\bfU{\mathbf{U}}
\def\bfP{\mathbf{P}}
\def\bfM{\mathbf{M}}
\def\bfN{\mathbf{N}}
\def\bfA{\mathbf{A}}

{\Small
\tableofcontents
}

\section{\bf Introduction}
\label{sec:intro}

A fundamental problem in the arithmetic theory of automorphic forms is to understand the algebraicity of special values of automorphic 
$L$-functions. Eisenstein cohomology provides a powerful method for approaching this problem, as in 
\cite{harder-inventiones}, \cite{harder-raghuram}, \cite{raghuram-imaginary}, \cite{bhagwat-raghuram}, 
\cite{farid}, and \cite{muthu-raghuram}. 
In this method, a key local input at the archimedean places is an arithmeticity statement for the map in cohomology induced by the 
normalized standard intertwining operator between induced representations. This article establishes that local archimedean input in the 
general framework of the Langlands–Shahidi machinery \cite{shahidi-book}.
The results proved here will be used in the forthcoming work \cite{raghuram-book}. While the motivating problem is global, the present paper is purely 
local and archimedean. The corresponding nonarchimedean local question is treated in \cite{raghuram-cjm}. 
We work with complex reductive groups, which are the relevant archimedean groups, for example, when the base field in the global setting is a 
CM number field. Henceforth, all considerations are local.

\medskip

Let $\mathrm G_0$ denote a connected reductive algebraic group over $\C$, and 
put $\mathrm G=\operatorname{Res}_{\C/\R}\mathrm G_0$. Write $G:=\mathrm{G}(\R) = \mathrm{G}_0(\C)$, this is a complex reductive 
group thought of as a real reductive group. We follow a similar notational scheme for subgroups of $\mathrm G_0$. 
Fix a choice of Borel subgroup $\mathrm{B}_0$ of $\mathrm{G}_0$; write $\mathrm{B}_0 = \mathrm{T}_0\mathrm{U}_0,$ where $\mathrm{T}_0$ 
is a maximal torus, and $\mathrm{U}_0$ the unipotent radical of $\mathrm{B}_0$. Suppose $\mathrm{P}_0$ is a maximal parabolic subgroup 
of $\mathrm{G}_0$, assumed to be standard, i.e., containing $\mathrm{B}_0$, and with Levi decomposition 
$\mathrm{P}_0 = \mathrm{L}_0\mathrm{N}_0.$ Let $\mathrm{A}_0$ denote the maximal central torus of $\mathrm{L}_0$. 
The $\C$-points of $\mathrm{G}_0, \mathrm{B}_0, \mathrm{T}_0, \mathrm{U}_0, \mathrm{P}_0, \mathrm{L}_0,$ 
$\mathrm{N}_0,$ and $\mathrm{A}_0$ are denoted by $G, B, T, U, P, L, N$, and $A$, respectively. 
To emphasize the dependence on $P$, we also denote $L = L_P$, $N = N_P,$ etc. Let $\tau$ be an irreducible admissible representation 
of $L_P$. (In the applications we have in mind, $\tau$ will be a local archimedean component of a global cuspidal automorphic 
representation of cohomological type.) Let $I_P^G(\tau, s)$ be the induced representation as in the Langlands--Shahidi theory 
\cite{shahidi-book}; the precise definition is recalled in \eqref{eq:parabolic-central-deformation}; for the introduction suffice it to say that it is 
obtained by normalised parabolic induction from $P$ to $G$ of $\tau$ with the complex variable $s$ introduced in a delicate manner. 
Let $\Delta_0$ be the set of simple roots of $\mathrm{G}_0$ with respect to $\mathrm{B}_0$ and $\alpha_{\mathrm{P}_0}$ the unique 
simple root corresponding to $\mathrm{P}_0$, and $w_P$ the unique element in the Weyl group of $G_0$ such that 
$w_P(\Delta_0 \setminus \{\alpha_{\mathrm{P}_0}\}) \subset {\Delta_0}$ and $w_0(\alpha_{\mathrm{P}_0}) < 0$; 
we also denote by $w_P$ a representative element in $\mathrm{G}_0$. Let $\mathrm{Q}_0$ be the parabolic subgroup of 
$\mathrm{G}_0$ associate to $\mathrm{P}_0$, and ${}^{w_P}\pi$ the corresponding representation of $M_{\mathrm{Q}}$. We denote 
$$
T_{\rm st}(s) : I_P^G(\tau, s) \to I_Q^G({}^{w_P}\tau, -s)
$$ 
for the standard intertwining operator (see \eqref{eq:inter_op} and \eqref{eqn:standard_inter_op}). 
There is a choice of measure implicit in the integral that defines the intertwining operator. 
{\it The main purpose of this article is to study the map induced in cohomology by the standard intertwining operator.}

\medskip

To put ourselves in an arithmetically interesting situation, assume that $\tau$ is tempered and cohomological with respect to some finite-dimensional 
irreducible representation $V_{\underline\mu}$ of $L_P$ with highest weight $\underline{\mu}$. 
\Cref{lem:purity} says that this highest weight satisfies a purity condition. Also, $\tau$ has nontrivial cohomology in a certain degree $q_0(L)$; 
see \eqref{eqn:q-0-L}. Denote by $s_0$ the {\it point of evaluation} (see \Cref{def:point_of_evaluation}), which is the point such that 
$$I_P^G(\tau, s)|_{s=s_0} \ = \ \aInd_P^G(\tau),$$ 
where the right hand side is the algebraic (un-normalized) parabolic induction of $\tau$ to a representation of $G$. 
In \Cref{prop:principal-series} one sees that $\aInd_P^G(\tau)$ is a principal series representation, i.e., induced from the Borel subgroup. 
The combinatorial lemma (see \Cref{lem:combinatorial_lemma}) says that if $s_0$ satisfies certain a combinatorial condition which we 
call \textit{hyper-critical}, then $\aInd_P^G(\tau)$ has nontrivial cohomology in degree $q_0(G) = q_0(L_P)+\tfrac12\dim_{\R}N_P$ 
(see \eqref{eq:complex-bottom-degree}) with respect to the finite-dimensional irreducible representation $V_{\underline\lambda}$ of $G$ with 
highest weight $\underline{\lambda}$ prescribed in the lemma. The cohomology group in degree $q_0(G)$ is one-dimensional. Furthermore, 
$I_Q^G({}^{w_0}\tau, -s_0)$ also has nontrivial cohomology in degree $q_0(G)$ with respect to the same $V_{\underline\lambda}$ and the 
cohomology group is one-dimensional. Under the conditions of \Cref{lem:combinatorial_lemma},  it is shown in \Cref{sec:irreducibility-principal-series} 
using Speh--Vogan \cite{speh-vogan} that $\aInd_P^G(\tau)$ is irreducible. Hence,  
$$
T_{\rm st}(s_0) : I_P^G(\tau, s_0) \to I_Q^G({}^{w_0}\tau, -s_0)
$$ 
is an isomorphism between two irreducibly induced representations and these representations have one-dimensional relative Lie algebra cohomology group in degree $q_0(G)$ with respect to $V_{\underline\lambda}$. We fix a generator $[J_\tau]$ for the one-dimensional relative Lie algebra cohomology in degree $q_0(G)$ of $I_P^G(\tau, s_0) \otimes V_{\underline\lambda}$ by the construction in \Cref{subsec:main-first-Delorme}. Similarly, fix $[J_{\tilde{\tau}}]$ for $I_Q^G({}^{w_0}\tau, -s_0) \otimes V_{\underline\lambda}.$ Under the principal-series identifications of \eqref{eq:mt3-endpoints}, these are the generators $[J_\chi]$ and $[J_{\tilde\chi}]$ of Theorem~\ref{thm:mainthm}, respectively. The isomorphism $T_{\rm st}(s_0)$ induces an isomorphism $T_{\rm st}^\bullet(s_0)$ in cohomology. Hence, for cohomology in degree $q_0(G)$, one has
$$
T_{\rm st}^\bullet(s_0)([J_\tau]) \ = \ c \cdot [J_{\tilde{\tau}}], \quad \mbox{for some $c \in \C^*$}.
$$
The main theorem of this article shows that this scalar $c$ is related to critical values of relevant Langlands--Shahidi $L$-functions. 

\medskip

To describe this more precisely, consider the Langlands dual groups. Suppose $\widehat{\mathrm G}_0$ is the complex reductive group 
dual to $G_0$. Let $\widehat P_0=\widehat{\mathrm L}_{P_0}\widehat{\mathrm N}_{P_0}$ be the standard parabolic of 
$\widehat{\mathrm G}_0$ corresponding to $P_0$. The Lie algebra $\widehat{\mathfrak n}_{P_0}$ of unipotent radical 
$\widehat{\mathrm N}_{P_0}$ of $\widehat P_0$ decomposes, under the adjoint action of the Levi subgroup 
$\widehat{\mathrm L}_{P_0}$ of $\widehat P_0$, into a multiplicity free direct sum $\oplus_{j=1}^m r_j$. Given $\tau$ as above, i.e.,
 an irreducible admissible essentially tempered cohomological representation of $L_P$, and an irreducible finite-dimensional representation 
 $r_j$ of $\widehat{\mathrm L}_{P_0}$, the local aspects of the Langlands--Shahidi machinery (see \Cref{subsec:LS_L-functions}) 
 attach a local $L$-factor $L_\C(s, \tau, r_j)$. The main result of this article, see \Cref{thm:mainthm}, says that 
\begin{equation}
\label{eqn:main-intro}
T_{\rm st}^\bullet(s_0)([J_\tau]) \ \approx_{\Q^\times} \  
\left(\prod_{j=1}^m
       \frac{L_{\mathbb C}(js_0,\tau,r_j)}
            {L_{\mathbb C}(1+js_0,\tau,r_j)}\right)
\cdot [J_{\tilde{\tau}}],
\end{equation}
where $\approx_{\Q^\times}$ means equality up to some nonzero rational number. 
The hyper-criticality mentioned in the hypothesis of the combinatorial lemma means that for each 
$1 \leq j \leq m$, the point $js_0$ has the right half-integrality property (determined by $L_P$ and $r_j$), and the $L$-values 
$L_{\mathbb C}(js_0,\tau,r_j)$ and $L_{\mathbb C}(1+js_0,\tau,r_j)$ appearing in the above product of ratios are critical in the 
sense of Definition~\ref{def:critical_set}. In particular, the displayed $L$-values are nonzero and finite.  

\medskip

The proof follows the general strategy in the Langlands--Shahidi machinery of appealing to Langlands's lemma 
(see, for example, \cite[Lemma 4.2.1]{shahidi-book}) to factorize the intertwining operator $T_{\rm st}(s_0)$ as a composition 
$T_{\alpha_1} \circ \cdots \circ T_{\alpha_d}$
of rank-one operators; see \Cref{sec:rank-one-int-op}; here the $\alpha_j$'s are simple roots and $d = \tfrac12\dim_{\R}N_P.$  
In cohomology, one has a corresponding composition $T_{\rm st}^\bullet(s_0) = T^\bullet_{\alpha_1} \circ \cdots \circ T^\bullet_{\alpha_d}$. 
Each rank-one operator $T_{\alpha_j}$ is the induction to $G$ of an intertwining operator $T_H(\psi_{\alpha_j})$ for a copy of $H = \SL_2(\C)$ 
determined by the simple root $\alpha_j$. The somewhat delicate \Cref{lem:main-first-Kostant-transitivity} shows that 
$T^\bullet_{\alpha_j}$ and $T_H^\bullet(\psi_{\alpha_j})$ are compatible with respect to the chosen rational generators (the appropriate $[J_\chi]$'s) 
up to some some unspecified 
nonzero rational number which explains the $\approx_{\Q^\times}$ appearing in \eqref{eqn:main-intro}. Piecing together all the $\SL_2(\C)$ contributions 
(see \Cref{prop:rk1-scalar} and specially \eqref{eq:main-first-root-line-formula}) gives the proof of \Cref{thm:mainthm}.

\medskip

Let us end this introduction by mentioning G\"unter Harder's hope, based on his \cite{harder-tifr} and especially Zagier's appendix \cite{zagier-tifr} of that paper, 
that the unspecified nonzero rational number appearing in the main theorem has to have a simple form.
This hope is motivated by the study of congruences, together with the construction of 
$p$-adic functions that interpolate products of ratios of critical values of global 
$L$-functions.

\bigskip

\section{\bf Preliminaries}
\subsection{Basic notations}
\label{sec:nota}
We now work locally at an archimedean place. Throughout, let $\mathrm G_0$ denote a split connected reductive algebraic group over $\C$, and put $\mathrm G=\operatorname{Res}_{\C/\R}\mathrm G_0$. Write $G:=\mathrm{G}(\R)$ and $G_\C:=\mathrm{G}(\C)$ for its real and complex points. Denote the derived subgroup of $G$ by $\mathrm G^1$, and put $G^1=\mathrm G^1(\R)$.

For example, if $\mathrm{G}_0=\GL_n$, then $\mathrm{G}_0^1=\SL_n$, $G=\GL_n(\C)$, and $G_{\C}\simeq\GL_n(\C)\times\GL_n(\C)$. We use analogous notation for other subgroups of $\mathrm{G}$.

We can and do fix a closed algebraic embedding $\mathrm G_0\hookrightarrow \mathrm{GL}_{m/\C}$ once and for all \cite[Theorem 8.6]{humphreys}. The group $G=\mathrm G_0(\C)$ is a real reductive group \cite[Chapter~VII, Section~2, Examples~2 and~4]{knapp3}. Let $\Gamma=\mathrm{Gal}(\C/\R)=\{1,\sigma\}$. In product coordinates, the inclusion of real points into complex points is
\begin{equation}\label{eq:real-complex-embedding}
 \iota_{\sigma}:
 (\Res_{\C/\R}\mathrm{GL}_{m/\C})(\R)
 \hookrightarrow(\Res_{\C/\R}\mathrm{GL}_{m/\C})(\C),
 \qquad g\longmapsto(g,\bar g).
\end{equation}
The nontrivial Galois element acts on the complex points by $\sigma(g_1,g_2)=(\bar g_2,\bar g_1)$. Thus
\begin{equation}\label{eq:Galois}
 (\Res_{\C/\R}\mathrm{GL}_{m/\C})(\R)
 =\big((\Res_{\C/\R}\mathrm{GL}_{m/\C})(\C)\big)^\sigma
 =\{(g,\bar g):g\in\GL_m(\C)\}\simeq\GL_m(\C).
\end{equation}
Together with \eqref{eq:real-complex-embedding}, the chosen embedding of $\mathrm G_0$ gives the commutative diagram

\[\begin{tikzcd}
\mathrm{G}_0(\C) \arrow[r,"\sim "] \arrow[d,hook]& G\arrow[r, hook]\arrow[d, hook]
  & G_{\C} \arrow[d,hook] \\
\GL_m(\C) \arrow[r,"= "]&\GL_m(\C) \arrow[r,hook]
  & \GL_m(\C)\times \GL_m(\C) 
\end{tikzcd}
\]
where the composition of the maps in the first row is
\begin{equation}\label{eq:embedding}\iota_\sigma: g\mapsto (g, \bar{g}) \in \GL_m(\C)\times \GL_m(\C), \quad \text{for}\  g\in \mathrm{G}_0(\C)\subset \GL_m(\C).\end{equation}
We use the global Cartan involution $\Theta$ of $\GL_m(\C)$ \cite{knapp3} given by
\begin{equation}\label{eq:cartan_involution}
 g\longmapsto{}^t\bar g^{-1}.
\end{equation}
Its complexification acts on $\GL_m(\C)\times\GL_m(\C)$ by $\Theta(g_1,g_2)=({}^tg_2^{-1},{}^tg_1^{-1})$. We choose the embedding so that $G_\C$ is stable under this involution. Then $\Theta$ restricts to a Cartan involution of $G$, which is a real reductive group in the sense of \cite[2.1.1]{wallach}.

Let $\mathrm Z$ be the algebraic center of $\mathrm G$, and write $C=G^\Theta$ for the maximal compact subgroup fixed by the chosen Cartan involution. Put $K=CZ$, where $Z=\mathrm{Z}(\R)$.

For $\mathrm G_0=\GL_n$, we have
\[ C_{\mathbb{C}}=\{(g_1, g_2)| \ g_1={}^{t}g_2^{-1}\}\cong \GL_n(\mathbb{C}),\quad C=\{(g, \bar{g})|\  g={}^{t}\bar{g}^{-1}\}\cong U(n)),\]
where $U(n)$ is the usual compact unitary subgroup of $GL_n(\C)$. And
\[
 \begin{aligned}
 K_{\C}&=\{(g_1,g_2)\in\GL_n(\C)\times \GL_n(\C) :
                g_1{}^tg_2=cI_n,\ c\in\C^\times\},\\
 K&=\{(g,\bar g) \in \GL_n(\C)\times \GL_n(\C):
                g{}^t\bar g=cI_n,\ c\in\R_{>0}\}.
 \end{aligned}
\]

For a maximal $\Theta$-stable torus $\mathrm{T}_0$ of $\mathrm{G}_0$, its restriction of scalars $\mathrm{T}=\Res_{\C/\R}\mathrm{T}_0$ is $\Theta$-stable, maximally split, and maximally compact in $\mathrm{G}$.
\vspace{1em}

\noindent\textbf{\large Characters and cocharacters.}
Write $\cha{T_0}$, $\cha{T}$, $\cha{T_{\C}}$ for the algebraic character groups over their base fields, and $X_*( T_0)$, $X_*(T)$, $X_*(T_\C)$ for the corresponding cocharacter groups. Restriction of scalars gives
\begin{equation}
  \label{eq:rational_characters}
\cha{T}\cong \cha{T_0}, \quad \cha{T_\C}\cong \cha{T_0}\oplus \cha{T_0}.
\end{equation}
Set
\[
 \Lambda=X^*(\mathrm T_0),\qquad
 \Lambda^\vee=X_*(\mathrm T_0),\qquad
 \Lambda_{\mathbb F}=\Lambda\otimes_{\mathbb Z}\mathbb F
 \quad(\mathbb F=\mathbb Q,\mathbb R,\mathbb C).
\]
We use Greek letters $\alpha, \lambda, \mu$ (resp., $\alpha^{\vee}, \lambda^{\vee}, \mu^{\vee}$) for elements of $\Lambda_\C$ (resp., $\Lambda_{\C}^\vee$). Underlined notation, such as $\ulb=(\lambda, \lambda')$, denotes an element of $\Lambda_\C \oplus \Lambda_\C$. The intrinsic perfect pairing is determined by
\[
 \lambda\circ\mu(z)=z^{\langle\lambda,\mu\rangle},
 \qquad \lambda\in\Lambda,\quad\mu\in\Lambda^\vee,
\]
and on $\Lambda\oplus\Lambda$ and $\Lambda^\vee\oplus\Lambda^\vee$ it is
\begin{equation}\label{eq:intrinsic-paired-pairing}
\langle \ulb, \ulm\rangle=  \langle(\lambda,\lambda'),(\mu,\mu')\rangle
 =\langle\lambda,\mu\rangle+\langle\lambda',\mu'\rangle.
\end{equation}

Through \eqref{eq:rational_characters}, $\lambda$ and $\ulb$ also denote elements of $\cha{T}_{\C}$ and $\cha{T_\C}_{\C}$, respectively. By \eqref{eq:embedding}, the real torus $T\subset T_\C$ is \(\left\{ (t, \bar{t})\in T_\C: t\in T_0\right\}.\) A character $\underline\lambda=(\lambda,\lambda')$ of $\mathrm T_{\C}$ therefore restricts by
\begin{equation}\label{eq:geometric-character-restriction}
 (\underline\lambda|_T)(t)
   =\underline\lambda(t,\bar t)
   =\lambda(t)\overline{\lambda'(t)}.
\end{equation}
This gives a complex-valued character of the real Lie group $T$, which need not be defined over $\mathbb R$. The Galois element $\sigma$ acts on $\cha{T_\C}$ by\(({}^\sigma \underline\lambda)(x) =\overline{\underline\lambda(\sigma^{-1}x)}\). Therefore \({}^{\sigma} \ulb={}^\sigma (\lambda,\lambda')=(\lambda',\lambda)\), and $\cha{T}_\C$ is identified with
\begin{equation}\label{eq:diag_embedding}
\left( \Lambda_\C\oplus \Lambda_\C\right)^{\sigma}
  =\{(\lambda,\lambda):\lambda\in\Lambda_\C\}.
\end{equation}
\begin{exam}[Characters of a complex torus]
\label{ex:torus-characters}
If $\mathrm T_0=\mathbb G_{m/\C}$ and $\mathrm T=\operatorname{Res}_{\C/\R}\mathrm T_0$, then $T=\C^\times$, $T_\C\simeq(\C^\times)^2$, and $\Lambda=\Z$. An algebraic character $(p,q)\in\Z^2$ of $\mathrm T_{\C}$ restricts to $z\mapsto z^p\bar z^q$. A real rational character has $p=q=k\in\Z$ and evaluates as $|z|_\C^k$, where $|z|_\C=z\bar z$, in agreement with \eqref{eq:geometric-character-restriction}.

More generally, write $\operatorname{Hom}_{\mathrm{cont}}(T,\C^\times)$ for the continuous quasi-characters of the Lie group $T$. For $T=\C^\times$, every such character is uniquely of the form
\[
 \chi_{c,m}(z)=|z|^c\left(\frac{z}{|z|}\right)^m,
 \qquad c\in\C,\quad m\in\Z,
\]
where $|z|$ is the ordinary modulus. The corresponding exponents are $p=(c+m)/2$ and $q=(c-m)/2$, so $p-q\in\Z$. It is algebraic exactly when $p,q\in\Z$.

For a torus of complex rank $n$, choose coordinates $T\simeq(\C^\times)^n$ and $T_{\C}\simeq(\C^\times)^n\times(\C^\times)^n$. In the latter coordinates real points have the form $(t,\bar t)$. For $k\in\Z$, $ke_i$ denotes the character of $\mathrm T_0$ given by $(z_1,\ldots,z_n)\mapsto z_i^k$.
\end{exam}

\subsection{Root systems and Weyl groups}
Let $\mathrm A$ be the maximal $\R$-split subtorus of $\mathrm T$, put $A=\mathrm A(\R)$ and $A^\circ=\mathrm A(\R)^\circ$, and let $\mathrm T_c$ be the compact real subtorus with $T_c=\mathrm T_c(\R)=T\cap C$. Then
\[
 T=T_cA^\circ,\qquad T_c\cap A^\circ=\{1\}.
\]
Write $\mathfrak g_0=\operatorname{Lie}_{\R}G$ and $\mathfrak g=\mathfrak g_0\otimes_{\R}\C$, and use analogous notation for $T$, $T_c$, and $A^\circ$.

The coordinates of \Cref{ex:torus-characters} identify $X_*(\mathrm A)$ with $\Lambda^\vee$. For an integral cocharacter $\mu^\vee\in\Lambda^\vee$, the curve $t\mapsto\mu^\vee(e^t)$ lies in $A^\circ$. Its tangent vector at the identity is
\[
 d\mu^\vee(1)=\left.\frac{d}{dt}\right|_{t=0}\mu^\vee(e^t).
\]
A cocharacter basis supplies split coordinates and a tangent basis of $\mathfrak a_0$. Differentiation therefore gives
\begin{equation}\label{eq:split-Lie-identification}
 \Lambda^\vee_{\R}\xrightarrow{\ \sim\ }\mathfrak a_0,
 \qquad \mu^\vee\otimes r\longmapsto r\,d\mu^\vee(1).
\end{equation}
Dually, the intrinsic pairing identifies $\mathfrak a_0^*$ with $\Lambda_{\R}$ by character differentiation. For $\lambda\in\Lambda$, restrict the character of $\mathrm T_0$ to $A^\circ$ and put
\begin{equation}\label{eq:differential}
 d\lambda(H)=\left.\frac{d}{dt}\right|_{t=0}
                 \lambda(\exp(tH)),\qquad H\in\mathfrak a_0.
\end{equation}
Indeed, differentiating $\lambda(\mu^\vee(e^t))=e^{t\langle\lambda,\mu^\vee\rangle}$ gives
\begin{equation}\label{eq:differentiated-intrinsic-pairing}
 d\lambda\bigl(d\mu^\vee(1)\bigr)
      =\langle\lambda,\mu^\vee\rangle.
\end{equation}
Thus $\lambda\mapsto d\lambda$ identifies $\Lambda_{\R}$ with $\mathfrak a_0^*$. Extending this identification, \eqref{eq:split-Lie-identification}, and the pairing complex-linearly gives
\[
 \mathfrak a\simeq\Lambda^\vee_{\C},\qquad
 \mathfrak a^*\simeq\Lambda_{\C}.
\]
We henceforth identify weights with their differentials and cocharacters with their tangent vectors, extended linearly.

The tangent vectors in \eqref{eq:split-Lie-identification} also give $\operatorname{Lie}_{\C}(\mathrm T_0)\simeq\Lambda^\vee_{\C}$. The underlying real Lie algebra is $\mathfrak t_0$. Let $\mathcal J$ denote multiplication by $i$ on this complex Lie algebra. The compact and split tangent spaces are then related by
\[
 \mathfrak t_{c,0}=\mathcal J\mathfrak a_0,\qquad
 \mathfrak t_0=\mathfrak a_0\oplus\mathcal J\mathfrak a_0.
\]
Differentiating \eqref{eq:embedding} on $T$ and extending complex-linearly gives
\begin{equation}\label{eq:torus-Lie-complexification}
 \begin{aligned}
 \mathfrak t=\mathfrak t_0\otimes_{\R}\C
   &\xrightarrow{\ \sim\ }
          \Lambda^\vee_{\C}\oplus\Lambda^\vee_{\C},\\
 H+\mathcal JH'&\longmapsto(H+iH',H-iH'),
       \qquad H,H'\in\mathfrak a_0.
 \end{aligned}
\end{equation}
Here $i$ on the right is scalar multiplication in the complexified Lie algebra. The subalgebras $\mathfrak a$ and $\mathfrak t_c$ are the diagonal and antidiagonal subspaces, respectively. Dualizing gives
\[
 \mathfrak t^*\simeq\Lambda_{\C}\oplus\Lambda_{\C},
\]
with the pairing given by \eqref{eq:intrinsic-paired-pairing}.

For a smooth character $\chi$ with differential $d\chi=\underline\nu=(\nu,\nu')$, \eqref{eq:torus-Lie-complexification} gives
\begin{equation}\label{eq:paired-character-differential}
 d\chi(H+\mathcal JH')
   =\langle\nu+\nu',H\rangle+i\langle\nu-\nu',H'\rangle,
 \qquad H,H'\in\mathfrak a_0.
\end{equation}
At the group level, \eqref{eq:geometric-character-restriction} gives, for an algebraic character $\underline\lambda=(\lambda,\lambda')$ of $\mathrm T_{\C}$,
\begin{equation}\label{eq:split-compact-evaluation}
 (\underline\lambda|_T)(au)
    =(\lambda+\lambda')(a)(\lambda-\lambda')(u),
 \qquad a\in A^\circ,\quad u\in T_c.
\end{equation}
Thus the character restriction maps are
\begin{equation}\label{eq:pre-resA}
 \operatorname{res}_A(\lambda,\lambda')=\lambda+\lambda',
 \qquad
 \operatorname{res}_c(\lambda,\lambda')=\lambda-\lambda',
\end{equation}
with targets $X^*(\mathrm A)$ and $X^*((\mathrm T_c)_{\C})$, respectively, both identified with $\Lambda$. Their dual cocharacter maps are $\mu\mapsto(\mu,\mu)$ and $\mu\mapsto(\mu,-\mu)$, so
\begin{equation}\label{eq:restriction-cocharacter-duality}
 \begin{aligned}
 \langle\operatorname{res}_A\underline\lambda,\mu\rangle
   &=\langle\underline\lambda,(\mu,\mu)\rangle,\\
 \langle\operatorname{res}_c\underline\lambda,\mu\rangle
   &=\langle\underline\lambda,(\mu,-\mu)\rangle.
 \end{aligned}
\end{equation}
By \eqref{eq:torus-Lie-complexification}, the compact tangent $\mathcal J\mu$ has coordinates $(i\mu,-i\mu)$, explaining the factor $i$ in \eqref{eq:paired-character-differential}. As in \cref{ex:torus-characters}, a single-valued smooth character satisfies $\nu-\nu'\in\Lambda$, whereas $\nu+\nu'$ may be arbitrary in $\Lambda_{\C}$. For an $\R$-rational character, the identification \eqref{eq:diag_embedding} gives
\begin{equation}\label{eq:norm-split-restriction}
 \operatorname{res}_A(\lambda,\lambda)=2\lambda,\qquad
 \operatorname{res}_c(\lambda,\lambda)=0.
\end{equation}
Thus the differential on $\mathfrak a_0$ of the $\R$-rational character corresponding to $\lambda$ is $2\lambda$.
\vspace{1em}

\noindent\textbf{\large Restricted root system of $G$.}
Let $\Phi_0=\Phi(\mathrm G_0,\mathrm T_0)\subset\Lambda$ denote the set of roots of $G_0$ with respect to $T_0$. Differentiating the adjoint action gives the restricted root spaces
\[
 \mathfrak g_{0,\alpha}
   =\{X\in\mathfrak g_0:
       [H,X]=\langle\alpha,H\rangle X
       \text{ for every }H\in\mathfrak a_0\}.
\]
The nonzero restricted weights are $\Phi_{\mathrm{res}}(G)=\Phi_0$, and each has multiplicity $\dim_{\R}\mathfrak g_{0,\alpha}=2$. Transport the coroot $\alpha^\vee\in\Lambda^\vee=X_*(\mathrm A)$ from the root datum of $(\mathrm G_0,\mathrm T_0)$. The root--coroot normalization and the corresponding Weyl reflection are
\begin{equation}\label{eq:root-datum-reflection}
 \langle\alpha,\alpha^\vee\rangle=2,\qquad
 s_\alpha(\lambda)
   =\lambda-\langle\lambda,\alpha^\vee\rangle\alpha.
\end{equation}
The restricted Weyl group is
\begin{equation}\label{eq:res_weyl}
 W_{\mathrm{res}}(G)=N_C(A)/Z_C(A)
   \simeq W(\mathrm G_0,\mathrm T_0)=:W_0.
\end{equation}
The root lattice is $\Z\Phi_0\subset\Lambda$. Its span omits central weight directions when the center has positive dimension.
\vspace{1em}

\noindent\textbf{\large Absolute root system of $G$.}
The absolute roots, namely the nonzero weights of $\mathrm T_{\C}$ on $\mathfrak g$, form two copies of $\Phi_0$:
\begin{equation}\label{eq:absolute_roots}
 \Phi_{\mathrm{abs}}(G)
   =\{\alpha^{(1)}=(\alpha,0):\alpha\in\Phi_0\}
      \sqcup
     \{\alpha^{(2)}=(0,\alpha):\alpha\in\Phi_0\}.
\end{equation}
Their coroots are $(\alpha^\vee,0)$ and $(0,\alpha^\vee)$. Restriction to $\mathrm A$, as in \eqref{eq:pre-resA}, gives the map between root systems
\begin{equation}\label{eq:Res}
 \operatorname{res}_{\mathrm{rt}}:
 \Phi_{\mathrm{abs}}(G)\longrightarrow\Phi_{\mathrm{res}}(G),
 \qquad \alpha^{(1)},\alpha^{(2)}\longmapsto\alpha.
\end{equation}
On split coroots, the dual map \eqref{eq:restriction-cocharacter-duality} is $\alpha^\vee\mapsto(\alpha^\vee,\alpha^\vee)$. Each restricted root $\alpha$ has the two absolute lifts $\alpha^{(1)}$ and $\alpha^{(2)}$ in \eqref{eq:Res}. Their sum $(\alpha,\alpha)$ is Galois invariant and restricts to $2\alpha$ by \eqref{eq:diag_embedding} and \eqref{eq:norm-split-restriction}. It is not itself an absolute root, because it belongs to neither copy of $\Phi_0$ in \eqref{eq:absolute_roots}.

The Weyl groups are related by
\begin{equation}\label{eq:absolute-weyl-diagonal}
 W_{\mathrm{abs}}(G)=W_0\times W_0,\qquad
 \jmath:W_{\mathrm{res}}(G)\hookrightarrow W_{\mathrm{abs}}(G),
 \quad w\longmapsto(w,w).
\end{equation}
The restriction map $\operatorname{res}_A$ in \eqref{eq:pre-resA} is equivariant for this diagonal embedding:
\[
 \operatorname{res}_A(\jmath(w)(\underline\nu))
   =w(\operatorname{res}_A(\underline\nu)).
\]
\vspace{1em}

\noindent\textbf{\large The Cartan involution.}
On the torus, the involution \eqref{eq:cartan_involution} acts as $\Theta(t_1,t_2)=(t_2^{-1},t_1^{-1})$. Differentiating gives
\begin{equation}\label{eq:theta-paired-weights}
 \theta(\nu,\nu')=(-\nu',-\nu).
\end{equation}
This action is complex-linear. The nontrivial Galois element $\sigma$ acts on the character lattice by $(\lambda,\lambda')\mapsto(\lambda',\lambda)$, extended complex-linearly in \eqref{eq:diag_embedding}. In particular, $\theta\alpha^{(1)}=-\alpha^{(2)}$. Thus every absolute root $\beta$ is complex in the sense that $\theta\beta\ne\pm\beta$. Extending the split and compact restrictions \eqref{eq:pre-resA} by zero gives
\begin{equation}\label{eq:Cartan-zero-extended-projections}
 \operatorname{pr}_A\underline\nu
   =\frac{\underline\nu-\theta\underline\nu}{2}
     =\left(\frac{\nu+\nu'}2,\frac{\nu+\nu'}2\right),\quad 
 \operatorname{pr}_c\underline\nu
   =\frac{\underline\nu+\theta\underline\nu}{2}
     =\left(\frac{\nu-\nu'}2,-\frac{\nu-\nu'}2\right).
\end{equation}
The factors $1/2$ come from extension by zero. The restrictions in \eqref{eq:pre-resA} remain the sum and difference.
\vspace{1em}

Fix a positive definite $W_0$-invariant symmetric form $(\, ,\, )_0$ on $\Lambda_{\R}$ and extend it complex-bilinearly. On $\Lambda_{\C}\oplus\Lambda_{\C}$, define
\begin{equation}\label{eq:absolute-symmetric-form}
 \big((\nu_1,\nu'_1),(\nu_2,\nu'_2)\big)_{\mathrm{abs}}
   =2\big((\nu_1,\nu_2)_0+(\nu'_1,\nu'_2)_0\big).
\end{equation}
To explain this normalization, let $b$ be the dual form on $\mathfrak a_0$ and put
\begin{equation}\label{eq:B0-blocks}
 B_0(H+\mathcal JK,H'+\mathcal JK')
    =b(H,H')-b(K,K').
\end{equation}
By \eqref{eq:paired-character-differential}, the split and compact components of $(\nu,\nu')$ are $\nu+\nu'$ and $i(\nu-\nu')$. The complexified dual of $B_0$ is therefore
\[
 (\nu_1+\nu'_1,\nu_2+\nu'_2)_0
 -(i(\nu_1-\nu'_1),i(\nu_2-\nu'_2))_0,
\]
which is exactly \eqref{eq:absolute-symmetric-form}. The form is positive definite on $\Lambda_{\R}\oplus\Lambda_{\R}$, whose covectors are real on $\mathfrak a_0$ and purely imaginary on $\mathfrak t_{c,0}$. We take real and imaginary parts of weights relative to this real form. The extension of covectors on $\mathfrak a_0$ by zero on $\mathfrak t_{c,0}$, given by $\nu\mapsto(\nu/2,\nu/2)$ in \eqref{eq:Cartan-zero-extended-projections}, is isometric.

We now relate these symmetric forms to the intrinsic coroot pairings, as needed for the Speh--Vogan criterion in \Cref{sec:irreducibility-principal-series}. Recall the Weyl reflection $s_\alpha$ from \eqref{eq:root-datum-reflection}. Since $\langle\alpha,\alpha^\vee\rangle=2$, it sends $\alpha$ to $-\alpha$. The $W_0$-invariance of $(\, ,\, )_0$ therefore gives
\[
 (\nu,\alpha)_0
   =(s_\alpha(\nu),s_\alpha(\alpha))_0=-(\nu,\alpha)_0
       +\langle\nu,\alpha^\vee\rangle(\alpha,\alpha)_0.
\]
Rearranging yields the first identity below. Applying the same argument to the absolute roots with the form \eqref{eq:absolute-symmetric-form} gives the second:
\begin{equation}\label{eq:root-coroot-pairing-comparison}
 \langle\nu,\alpha^\vee\rangle
   =\frac{2(\nu,\alpha)_0}{(\alpha,\alpha)_0},\qquad
 \langle\underline\nu,\ual^\vee\rangle
   =\frac{2(\underline\nu,\ual)_{\mathrm{abs}}}
          {(\ual,\ual)_{\mathrm{abs}}},
\end{equation}
for $\alpha\in\Phi_0$ and $\ual\in\Phi_{\mathrm{abs}}(G)$. These formulas recover the intrinsic coroot pairings. Throughout, $\langle\, ,\,\rangle$ denotes the weight--cocharacter pairing, extended by differentiation and linearity.
\vspace{1em}

\noindent\textbf{\large Positive systems and half-sums.}
Fix a Borel subgroup $\mathrm B_0\supset\mathrm T_0$, which then gives the positive and simple roots $\Phi_0^+$ and $\Delta_0$. Use
\[
 \Phi_{\mathrm{res}}^+(G)=\Phi_0^+,\qquad
 \Phi_{\mathrm{abs}}^+(G)
   =(\Phi_0^+)^{(1)}\sqcup(\Phi_0^+)^{(2)}.
\]
By \eqref{eq:absolute-weyl-diagonal}, $\ell_{\mathrm{abs}}(w_1,w_2) =\ell_{\mathrm{res}}(w_1)+\ell_{\mathrm{res}}(w_2)$. If a $\theta$-stable positive system is needed, we explicitly use $(\Phi_0^+)^{(1)}\sqcup(\Phi_0^-)^{(2)}$ instead. Define the half-sums of positive roots by
\[
 \rho_G=\frac12\sum_{\alpha\in\Phi_0^+}\alpha,\qquad
 \underline\rho_G
   =\frac12\sum_{\beta\in\Phi_{\mathrm{abs}}^+(G)}\beta
   =(\rho_G,\rho_G).
\]
Since each restricted root has real multiplicity $2$, the half-sum of restricted roots counted with their multiplicities is
\[
 \frac12\sum_{\alpha\in\Phi_{\mathrm{res}}^+(G)}
       \dim_{\R}(\mathfrak g_{0,\alpha})\,\alpha
   =2\rho_G=\operatorname{res}_A\underline\rho_G.
\]
The same convention applies to half-sums of positive roots of Levi subgroups and of roots in unipotent radicals of parabolic subgroups. These half-sums are rational weights and need not be integral.

\subsection{Pinnings and rank-one homomorphisms}
\label{subsec:root-homomorphisms} Put $\mathrm B=\operatorname{Res}_{\C/\R}\mathrm B_0$ and $B=\mathrm B(\R)$, so $\Delta_{\mathrm{res}}(G)=\Delta_0$. Choose a pinning \((\mathrm T_0,\mathrm B_0,\{x_\alpha\}_{\alpha\in\Delta_0})\) compatible with the compact form $C$, where $x_\alpha:\mathbb G_a\longrightarrow\mathrm G_0$ parametrizes the simple-root subgroup. Extend these maps to compatible Chevalley parametrizations $x_\alpha$ for all $\alpha\in\Phi_0$. For each simple root $\alpha$, the pinning determines a root homomorphism \(\varphi_\alpha:\mathrm{SL}_2(\C)\longrightarrow G.\) For
\[
 a(z)=\begin{pmatrix}z&0\\0&z^{-1}\end{pmatrix},\qquad
 n(u)=\begin{pmatrix}1&u\\0&1\end{pmatrix},\qquad
 w_2=\begin{pmatrix}0&-1\\1&0\end{pmatrix},
\]
one has
\begin{equation}\label{eq:rk1-root-map}
 \begin{aligned}
 \varphi_\alpha(a(z))&=\alpha^\vee(z),&
 \varphi_\alpha(n(u))&=x_\alpha(u),\\
 \dot s_\alpha&=\varphi_\alpha(w_2),&
 \varphi_\alpha(\mathrm{SU}(2))&\subset C.
 \end{aligned}
\end{equation}
Its kernel is contained in $\{\pm I\}$ and is determined by $\varphi_\alpha(-I)=\alpha^\vee(-1)$ \cite[Proposition~20.32]{milne}. Thus the image of $\varphi_\alpha$ is $\mathrm{SL}_2(\C)$ or $\mathrm{PGL}_2(\C)$. Its restriction to each root subgroup of $\mathrm{SL}_2(\C)$ is an isomorphism onto the corresponding root subgroup of $G$. Transport the Haar measure
\begin{equation}\label{eq:rk1-measure}
 d\mu_{\C}(u)=2\,dx\,dy,\qquad u=x+iy,
\end{equation}
through $x_\alpha$ to each root subgroup. These choices fix the normalizations used in the intertwining integrals and rational bases.

\section{\bf Finite-dimensional representations and parabolic inductions}
\label{sec:representations}

\subsection{Highest-weight representations}
\label{subsec:highest_weight} Base change of $\mathrm B=\operatorname{Res}_{\C/\R}\mathrm B_0$ gives the Borel subgroup \(\mathrm B_{\C}=\mathrm B\times_{\R}\C\subset\mathrm G_{\C}, \mathrm T_{\C}\subset\mathrm B_{\C}\). In the product coordinates of \eqref{eq:embedding}, this is the product of $\mathrm B_0$ with its conjugate Borel subgroup~\cite[2.61]{milne}. Hence, in the product coordinates, its simple roots are
\[
 \Delta_{\mathrm{abs}}(G)
   =\{(\alpha,0):\alpha\in\Delta_0\}
       \sqcup\{(0,\alpha):\alpha\in\Delta_0\}.
\]
This choice of $\mathrm B_0$ determines the positive system for highest weights of $\mathrm G_{\C}$.

An \textit{integral weight} is an element $\underline\lambda=(\lambda,\lambda')\in X^*(\mathrm T_{\C})=\Lambda\oplus\Lambda$. It is \textit{dominant} with respect to $\mathrm B_{\C}$ if
\[
 \langle\lambda,\alpha^\vee\rangle\geq0,\qquad
 \langle\lambda',\alpha^\vee\rangle\geq0
       \qquad(\alpha\in\Delta_0).
\]
Equivalently, both components are dominant for $\mathrm B_0$. See \cite[Theorem~22.2]{milne} for the following:
\begin{prop}[Highest-weight modules]\label{prop:highest_weight}
Suppose $G$ is semisimple. The dominant integral weights $\underline\lambda\in X^*(\mathrm T_{\C})$ parametrize the irreducible finite-dimensional algebraic representations of $G$. We write $V_{\underline\lambda}$ for the representation of $\mathrm G_{\C}$ with highest weight $\underline\lambda$ relative to $\mathrm B_{\C}$, and for its restriction to $G$ along $\iota_\sigma:G\hookrightarrow G_{\C}$ in \eqref{eq:embedding}
\end{prop}
\vspace{1em}

\noindent\textbf{\large The central character.}
For a general reductive $G$, put $\mathrm T^1=\mathrm T\cap\mathrm G^1$ and $T^1=T\cap G^1$, and let $\mathrm Z^\circ$ be the connected center of $\mathrm G$. The multiplication map $\mathrm Z^\circ\times\mathrm G^1\longrightarrow\mathrm G$ is a central isogeny \cite[Proposition~21.61]{milne}. Since the connected center acts by a character on an irreducible representation, the classification extends as follows.

\begin{cor}\label{cor:finite_dim_repn}
An irreducible finite-dimensional algebraic representation of $G$ is determined by a dominant weight $(\lambda,\lambda')\in X^*((\mathrm T^1)_{\C})$ and an algebraic central character $\ulw=(\mathsf w,\mathsf w')$, subject to agreement of the two actions on $(\mathrm Z^\circ\cap\mathrm G^1)_{\C}$. Equivalently, the derived and central weights form a pair in the image of
\[
 X^*(\mathrm T_{\C})\hookrightarrow
 X^*((\mathrm T^1)_{\C})\oplus
 X^*((\mathrm Z^\circ)_{\C}).
\]
Every such pair gives a unique irreducible representation.
\end{cor}

The inclusion in \Cref{cor:finite_dim_repn} has finite-index image and becomes an isomorphism over $\mathbb Q$. Extend the derived and central weights by zero in the complementary directions. Their sum gives the highest weight of the representation of $\mathrm G_{\C}$
\[
 \underline\lambda=(\lambda+\mathsf w,\lambda'+\mathsf w')
                 \in X^*(\mathrm T_{\C}).
\]
The summands may be rational, but their sum must be integral. We retain $V_{\underline\lambda}$ for the corresponding representation of $\mathrm G_{\C}$ and of $G$.

\subsection{Parabolic induction and the point of evaluation}
\label{subsec:parabolic_induction}
Let $\mathrm P_0=\mathrm L_{P,0}\mathrm N_{P,0}\supset\mathrm B_0$ be a standard parabolic subgroup of $\mathrm G_0$, with Levi subgroup $\mathrm L_{P,0}\supset\mathrm T_0$ and unipotent radical $\mathrm N_{P,0}$. By the convention of \Cref{sec:nota}, write $P=L_PN_P$ for the corresponding real parabolic subgroup. Let $\Phi(P)\subset\Lambda=X^*(\mathrm T_0)$ be the roots occurring in $\operatorname{Lie}_{\C}\mathrm N_{P,0}$, identified with $\Phi_{\rm res}^+(G)\setminus\Phi_{\rm res}^+(L_P)$. We shall denote the half-sum of these roots by
\begin{equation}\label{eq:parabolic-half-sum}
 \rho_P=\frac12\sum_{\alpha\in\Phi(P)}\alpha
       =\rho_G-\rho_{L_P}\in\Lambda_{\mathbb Q}.
\end{equation}
By restriction of scalars, each $\alpha\in\Phi(P)$ contributes the two absolute roots $(\alpha,0)$ and $(0,\alpha)$ in $\operatorname{Lie}_{\R}(N_P)\otimes_{\R}\C$, as in \eqref{eq:absolute_roots}. The half-sum of these absolute roots is
\[
 \underline\rho_P
   =\frac12\sum_{\alpha\in\Phi(P)}
       \big((\alpha,0)+(0,\alpha)\big)
   =(\rho_P,\rho_P).
\]
\vspace{1em}

\noindent\textbf{\large Langlands decomposition.}
Let $\mathrm A_P$ be the maximal $\R$-split subtorus of $Z(\mathrm L_P)^\circ$, and put $A_P=\mathrm A_P(\R)$, $\mathfrak a_{P,0}=\operatorname{Lie}_{\R}(A_P^\circ)$, and $\mathfrak a_P=\mathfrak a_{P,0}\otimes_{\R}\C$. Let $M_P$ be the common kernel of the absolute values of the $\R$-rational characters of $\mathrm L_P$. The Langlands decomposition is
\[
 L_P=M_P\times A_P^\circ,\qquad P=M_PA_P^\circ N_P.
\]
Thus the $M$ and $A$ in the usual notation $P=MAN$ are $M_P$ and $A_P^\circ$, respectively \cite[0.3.2--0.3.4]{borel-wallach}. The Harish--Chandra homomorphism $\mathcal H_P:L_P\to\mathfrak a_{P,0}$ is characterized by
\begin{equation}\label{eq:harish-chandra-map}
 e^{\langle d\eta,\mathcal H_P(l)\rangle}=|\eta(l)|,
 \qquad l\in L_P,
\end{equation}
for every $\R$-rational algebraic character $\eta:\mathrm L_P\to\mathbb G_{m/\R}$, where $d\eta$ is restricted to $\mathfrak a_{P,0}$ \cite[Section~3.3, p.~49]{shahidi-book}. Its kernel is $M_P$, and $\mathcal H_P(\exp X)=X$ for $X\in\mathfrak a_{P,0}$. When pairing weights with $\mathcal H_P(l)$, we likewise restrict them to $\mathfrak a_{P,0}$.
\vspace{1em}

\noindent\textbf{\large The modulus character.}
To fix the normalization of parabolic induction, the modulus character of $P$ is the Jacobian of $n\mapsto lnl^{-1}$, given by the positive character
\begin{equation}\label{eq:parabolic-modulus}
 \delta_P(l)=\left|\det_{\R}\left(
   \operatorname{Ad}(l)|_{\operatorname{Lie}_{\R}N_P}\right)\right|,
 \qquad l\in L_P,
\end{equation} and extended trivially across $N_P$. Restricting to $T$ and using \eqref{eq:parabolic-half-sum} and the restriction map in \eqref{eq:pre-resA} gives
\begin{equation}\label{eq:parabolic-modulus-differential}
 \delta_P^{1/2}(t)=\prod_{\alpha\in\Phi(P)}
                              |\alpha(t)|_{\C}^{1/2},\quad 
 d(\delta_P^{1/2}|_T)=\underline\rho_P,\qquad
 \operatorname{res}_A\underline\rho_P=2\rho_P.
\end{equation}
The weight $2\rho_P$ extends to the algebraic determinant character of $\mathrm L_{P,0}$ on $\operatorname{Lie}_{\C}\mathrm N_{P,0}$. Using $2\rho_P$ also for this character, we can write \(\delta_P(l)=\left|(2\rho_P)(l)\right|_{\C}\). Thus $\delta_P$ is shorthand for the positive character determined by $\rho_P$. In terms of the Harish--Chandra map,
\[
 \delta_P(l)^{1/2}
   =e^{\langle2\rho_P,\mathcal H_P(l)\rangle},
 \qquad l\in L_P.
\]
In particular, $\delta_P(ma)=\delta_P(a)$. The factor $2$ comes from applying the restriction map $\operatorname{res}_A$ in \eqref{eq:pre-resA} to $(\rho_P,\rho_P)$.
\vspace{1em}

\noindent\textbf{\large Normalized parabolic induction.}
Let $\tau$ be a smooth admissible Fr\'echet representation of  $L_P$ on $V_\tau$, and let $\xi:A_P^\circ\to\C^\times$ be a smooth character. We use the same notation for its extension to $P$, defined by $\xi(man)=\xi(a)$. Thus $\xi$ is trivial on $M_P$ and $N_P$. Its restriction to $L_P$ has the form
\[
 \xi(l)=e^{\langle v,\mathcal H_P(l)\rangle},
 \qquad l\in L_P,
\]
for a unique $v\in\mathfrak a_P^*$.

The normalized induced representation $\Ind_P^G(\tau,\xi)$ consists of the smooth functions $f:G\to V_\tau$ satisfying
\begin{equation}\label{eq:normalized-parabolic-induction}
 f(mang)=\delta_P(a)^{1/2}\xi(a)\tau(ma)f(g).
\end{equation}
Here $m\in M_P$, $a\in A_P^\circ$, $n\in N_P$, and $g\in G$. The group $G$ acts by right translation.

When $\tau$ is given as a representation of $M_P$, we extend it trivially to $A_P^\circ$. Then $\tau(ma)=\tau(m)$, giving the usual $MAN$ convention of \cite[III.3.2]{borel-wallach}. In our applications, $\tau$ is a fixed cohomological representation of $L_P$, including its prescribed central character. The character $\xi$ supplies an additional twist of its $A_P^\circ$-action.

We abbreviate $\Ind_P^G(\tau,\trivial)$ to $\Ind_P^G(\tau)$, so
\[
 \Ind_P^G(\tau,\xi)=\Ind_P^G(\tau\otimes\xi).
\]
Omitting $\delta_P^{1/2}$ defines unnormalized (algebraic) induction. The two conventions are related by
\[
 \begin{aligned}
 \Ind_P^G(\tau,\xi)
   &={}^a\Ind_P^G(\tau\otimes\xi\otimes\delta_P^{1/2}),\\
 {}^a\Ind_P^G(\tau)
   &=\Ind_P^G(\tau,\delta_P^{-1/2}).
 \end{aligned}
\]
When $P$ is the Borel subgroup $B$, we write $I(\chi)$ for the normalized induction of a character $\chi$ of $T$.
\vspace{1em}

\noindent\textbf{\large Maximal parabolics and the point of evaluation.}
Now suppose that $P=L_PN_P$ is maximal, with omitted simple root $\alpha_P$. To fix the direction and scale of the parameter, put
\begin{equation}\label{eq:parabolic-parameter-normalization}
 \gamma_P=\frac{\rho_P}{\langle\rho_P,\alpha_P^\vee\rangle}.
\end{equation}
This weight satisfies $\langle\gamma_P,\alpha_P^\vee\rangle=1$ and $\langle\gamma_P,\beta^\vee\rangle=0$ for every simple root $\beta$ of $L_P$. We keep this normalization even when $\gamma_P$ is not integral. For fixed $\tau$ and $s\in\C$, take the additional character $\xi(l)=e^{\langle2s\gamma_P,\mathcal H_P(l)\rangle}$ on $L_P$. This gives the family
\begin{equation}\label{eq:parabolic-central-deformation}
 I(\tau, s)=\Ind_P^G\bigl(
       \tau,e^{s\langle2\gamma_P,\mathcal{H}_P(\cdot)\rangle}\bigr).
\end{equation}
The twisting character depends holomorphically on $s$. It is trivial on the derived group of $L_P$ and on $Z(G)$, and varies along the one-dimensional split central quotient $A_P^\circ/A_G^\circ$. Multiplying this character by $\delta_P^{1/2}$ in \eqref{eq:normalized-parabolic-induction} gives
\begin{equation}\label{eq:parabolic-family-law}
 f_s(mang)
   =e^{\langle2(\rho_P+s\gamma_P),\mathcal H_P(ma)\rangle}
                 \tau(ma)f_s(g).
\end{equation}
Thus the family becomes unnormalized induction when $\rho_P+s\gamma_P=0$.

\begin{defn}[\bf Point of evaluation]\label{def:point_of_evaluation}
The point of evaluation for $P$ is
\[
 s_0=-\langle\rho_P,\alpha_P^\vee\rangle.
\]
Since $2\rho_P$ is a sum of roots, $2s_0\in\Z$.
\end{defn}

Equation~\eqref{eq:parabolic-parameter-normalization} gives $\rho_P=-s_0\gamma_P$. Sometimes we would use the simplified notations:
\begin{equation}\label{eq:parabolic-modulus-family}
 I(\tau, s):=\Ind_P^G\bigl(\tau,\delta_P^{-s/(2s_0)}\bigr).
\end{equation}
Here the complex power means $\exp\bigl(-s\log\delta_P(l)/(2s_0)\bigr)$, using the real logarithm of the positive character $\delta_P$. At $s=s_0$, the exponential factor in \eqref{eq:parabolic-family-law} is $1$, so $f_{s_0}(mang)=\tau(ma)f_{s_0}(g)$. We denote the resulting fixed representation by
\[
 \Pi:=I(\tau,s_0)
     =\Ind_P^G(\tau,\delta_P^{-1/2})
     ={}^a\Ind_P^G(\tau).
\]

For $\beta\in\Phi(P)$, put $j=\langle\gamma_P,\beta^\vee\rangle$. This is the positive integer coefficient of $\alpha_P^\vee$ in the simple-coroot expansion of $\beta^\vee$. The formula for $\delta_P^{1/2}$ in \eqref{eq:parabolic-modulus-differential} gives
\[
 (\delta_P^{1/2}\circ\beta^\vee)(z)
    =|z|_\C^{\langle\rho_P,\beta^\vee\rangle}
    =|z|_\C^{-js_0}.
\]
Raising this character to the power $-s/s_0$ yields
\begin{equation}\label{eq:parabolic-root-twist}
 \bigl(\delta_P^{-s/(2s_0)}\circ\beta^\vee\bigr)(z)
   =|z|_\C^{s\langle\gamma_P,\beta^\vee\rangle}
   =|z|_\C^{js}.
\end{equation}
This explains the scale of $s$: a root with coefficient $j$ acquires the twist $|\cdot|_\C^{js}$.
\vspace{1em}

\noindent\textbf{\large Induction in stages.}
For any standard parabolic $P=L_PN_P$, let $\chi$ be a smooth character of $T$ and $\xi$ a smooth character of $A_P^\circ$, extended to $L_P$ as above. Write $I_{L_P}(\chi)=\Ind_{B\cap L_P}^{L_P}(\chi)$. We can first induce $\chi$ from $B\cap L_P$ to $L_P$, then twist by $\xi$ and induce from $P$ to $G$. Induction in stages identifies this with inducing directly from $B$ to $G$:
\begin{equation}\label{eq:parabolic-induction-in-stages}
 \Ind_P^G\bigl(I_{L_P}(\chi),\xi\bigr)
       \simeq I\bigl(\chi\cdot(\xi|_T)\bigr).
\end{equation}
Here $\chi\cdot(\xi|_T)$ is the pointwise product $t\mapsto\chi(t)\xi(t)$. The two normalizations combine because
\[
 \delta_B|_T=(\delta_{B\cap L_P}|_T)(\delta_P|_T).
\]

For the family \eqref{eq:parabolic-modulus-family}, suppose $\tau=I_{L_P}(\chi_{\ulm})$ for a smooth character $\chi_{\ulm}$ of $T$. Apply \eqref{eq:parabolic-induction-in-stages} with $\chi=\chi_{\ulm}$ and $\xi=\delta_P^{-s/(2s_0)}$. Then
\[
 \begin{aligned}
 I(\tau, s)
   &=\Ind_P^G\bigl(I_{L_P}(\chi_{\ulm}),\delta_P^{-s/(2s_0)}\bigr)
       \simeq I(\chi_s),\\
 \chi_s(t)&=\chi_{\ulm}(t)\delta_P(t)^{-s/(2s_0)}
                   \qquad(t\in T).
 \end{aligned}
\]
Thus $\chi_s$ is obtained by multiplying $\chi_{\ulm}$ by the restriction of $\delta_P^{-s/(2s_0)}$ to $T$. By \eqref{eq:parabolic-modulus-differential}, the differential of this twist is $(s\gamma_P,s\gamma_P)$, giving
\begin{equation}\label{eq:parabolic-twist-differential}
 d\chi_s=d\chi_{\ulm}+(s\gamma_P,s\gamma_P).
\end{equation}
The defining transformation law for functions in $I(\chi_s)$ uses the character $\delta_B^{1/2}\chi_s$, whose differential is $d\chi_s+\underline\rho_G$.

\begin{exam}[Classical parabolics and half-sums]
\label{exam:maximal}
In type $A_{N-1}$, take $\alpha_i=e_i-e_{i+1}$ for $i<N$. Use the same chain in types $B_N,C_N,D_N$, with $\alpha_N=e_N,\ 2e_N,\ e_{N-1}+e_N$, respectively. As in \Cref{cor:finite_dim_repn}, the coordinates lie in the derived root space and are extended by zero in the central directions. For the maximal parabolic $P_m$ obtained by omitting $\alpha_m$, the table lists its Levi subgroup $L_{P_m}$ and the half-sum \eqref{eq:parabolic-half-sum}. All groups are evaluated at $\C$, with $N\geq4$ in type $D_N$.
\[
\begingroup
\setlength{\arraycolsep}{4pt}
\renewcommand{\arraystretch}{1.35}
\begin{array}{c|c|c|l}
 G & m & L_{P_m} & \rho_{P_m}\\ \hline
 \GL_N & 1\leq m<N & \GL_m\times\GL_{N-m}
   & \frac{N-m}{2}\sum_{i=1}^{m}e_i
       -\frac m2\sum_{i=m+1}^{N}e_i\\ \hline
 \GSp_{2N} & 1\leq m<N & \GL_m\times\GSp_{2(N-m)}
   & \frac{2N-m+1}{2}\sum_{i=1}^{m}e_i\\
 & N & \GL_N\times\Gm
   & \frac{N+1}{2}\sum_{i=1}^{N}e_i\\ \hline
 \GSpin_{2N+1} & 1\leq m<N & \GL_m\times\GSpin_{2(N-m)+1}
   & \frac{2N-m}{2}\sum_{i=1}^{m}e_i\\
 & N & \GL_N\times\Gm
   & \frac N2\sum_{i=1}^{N}e_i\\ \hline
 \GSpin_{2N} & 1\leq m\leq N-2 & \GL_m\times\GSpin_{2(N-m)}
   & \frac{2N-m-1}{2}\sum_{i=1}^{m}e_i\\
 & N & \GL_N\times\Gm
   & \frac{N-1}{2}\sum_{i=1}^{N}e_i\\
 & N-1 & \GL_N\times\Gm
   & \frac{N-1}{2}\bigl(\sum_{i=1}^{N-1}e_i-e_N\bigr)
\end{array}
\endgroup
\]
Each complex root occurring in the unipotent radical of $P_m$ is counted once. The two terminal $D_N$ cases are interchanged by $e_N\mapsto-e_N$. For descriptions of these Levi subgroups in the general-spin groups, see \cite{asgari-cogdell-shahidi}.

\Cref{def:point_of_evaluation} gives the following points of evaluation:
\[
\begin{array}{c|c|c}
 G & \text{omitted root} & s_0\\ \hline
 \GL_N & \alpha_m,\ 1\leq m<N & -N/2\\
 \GSp_{2N} & \alpha_m,\ 1\leq m<N & -(2N-m+1)/2\\
 \GSp_{2N} & \alpha_N & -(N+1)/2\\
 \GSpin_{2N+1} & \alpha_m,\ 1\leq m<N & -(2N-m)/2\\
 \GSpin_{2N+1} & \alpha_N & -N\\
 \GSpin_{2N} & \alpha_m,\ 1\leq m\leq N-2 & -(2N-m-1)/2\\
 \GSpin_{2N} & \alpha_{N-1}\ \text{or}\ \alpha_N & -(N-1).
\end{array}
\]
Indeed, the coroots at the last node are $2e_N$ in type $B_N$, $e_N$ in type $C_N$, and $e_{N-1}+e_N$ in type $D_N$. Pairing them with the half-sums above gives the displayed values.
\end{exam}

\section{\bf Cohomological representations}
\label{sec:temp_coh_repn}
An irreducible admissible smooth representation $\pi$ of $G$ is \textit{cohomological} with respect to an irreducible finite-dimensional algebraic representation $V_{\underline\lambda}$ if
\[
 H^\bullet(\mathfrak g,K;\pi\otimes V_{\underline\lambda})\ne0,
 \qquad K=CZ(G).
\]
It is \textit{essentially tempered} if $\pi|_{G^1}$ is tempered. We use the same notation $\pi$ for a smooth representation and its Harish--Chandra module. For the definition and basic properties of relative Lie algebra cohomology, see \cite[Chapter~I]{borel-wallach}. The highest weight of a coefficient representation supporting tempered cohomology satisfies a \textit{purity condition} relating its two components on the derived group. We use this relation and the inverse central character to realize a cohomological representation of $L_P$ as a principal series and to describe its parabolic induction to $G$.

For a Levi subgroup $L$ (including $G$), put $K_L=(C\cap L)Z(L)$.

\begin{lemma}[\bf Purity condition]
\label{lem:purity}
Let $\tau$ be an irreducible essentially tempered representation of $L$, cohomological with respect to $V_{\underline\mu}$, where
\[
 \underline\mu=(\mu+\mathsf w,\mu'+\mathsf w'),
\]
with $\mu,\mu'$ in the derived directions and $\mathsf w,\mathsf w'$ central, as in \Cref{cor:finite_dim_repn}. Let $w_L$ be the longest element of $W_{\mathrm{res}}(L)$. Then
\[
 \mu'=-w_L(\mu).
\]
\end{lemma}
\begin{proof}
Nonzero relative cohomology forces the central actions on $\tau$ and $V_{\underline\mu}$ to cancel. Since $K_L^\circ$ contains $Z(L)^\circ$, the relative cochain complex then reduces to that for $L^1$ and its maximal compact subgroup. The representation $\tau|_{L^1}$ is tempered and hence unitarizable. The vanishing criterion of \cite[Chapter~II, Proposition~6.12]{borel-wallach} implies that the coefficient representation restricted to $L^1$ is invariant under the Cartan involution. By the formula for the action on the two weight components in \eqref{eq:theta-paired-weights}, this condition is $\mu'=-w_L(\mu)$.
\end{proof}

For the cohomological degree in the next proposition, we use Borel--Wallach's definition \cite[Chapter~III, \S4.3]{borel-wallach}:
\begin{equation}\label{eqn:q-0-L}
 q_0(L)=\frac12\bigl(\dim_{\R}L-\dim_{\R}(C\cap L)
                 -\operatorname{rk}L+\operatorname{rk}(C\cap L)\bigr).
\end{equation}
Here $C\cap L$ is a maximal compact subgroup of $L$, and $\operatorname{rk}$ denotes the rank of the complexified real Lie algebra. For $L=\mathrm L_0(\C)$, the dimension difference is $\dim_{\C}\mathrm L_0$, while the two ranks are $2\operatorname{rk}_{\C}\mathrm L_0$ and $\operatorname{rk}_{\C}\mathrm L_0$, respectively. Hence
\[
 q_0(L)=\frac{\dim_{\C}\mathrm L_0
                    -\operatorname{rk}_{\C}\mathrm L_0}{2}
       =\#\Phi_0^+(L)=q_0(L^1).
\]
Removing the center subtracts equally from the complex dimension and rank. Thus passing to the derived group as in \Cref{lem:purity} preserves the lowest nonzero cohomological degree. For $P=L_PN_P$, root counting gives
\begin{equation}\label{eq:complex-bottom-degree}
 q_0(G)=q_0(L_P)+\#\Phi(P)
       =q_0(L_P)+\tfrac12\dim_{\R}N_P.
\end{equation}

\begin{rem}[Bottom cohomological degree]
For $\tau$ and $V_{\underline\mu}$ as in \Cref{lem:purity}, the relative cohomology $H^\bullet(\mathfrak l,K_L^\circ;\tau\otimes V_{\underline\mu})$ vanishes below $q_0(L)$ and is one-dimensional in degree $q_0(L)$. This is the bottom degree one can read from the tempered cohomology formula in Borel--Wallach's book~\cite[Chapter~III, Theorem~5.1 and Corollary~5.2]{borel-wallach}, after the central reduction above. In the principal-series realization of \Cref{prop:principal-series}, this degree is the length of the contributing Kostant element $(w_L,1)$, while the remaining torus factor contributes in degree zero.

For the parabolic inductions in Theorem~\ref{thm:mainthm}, the formula \cite[III.3.3(ii), (iii)]{borel-wallach} adds the degree in which the matching Levi coefficient occurs in nilradical cohomology to the bottom Levi degree. \Cref{lem:combinatorial_lemma} shows that this added degree is exactly $\#\Phi(P)$ in \eqref{eq:comb-conclusion}. Thus \eqref{eq:complex-bottom-degree} identifies the ambient bottom degree as $q_0(G)$.
\end{rem}

\begin{prop}[\bf Principal-series realizations]
\label{prop:principal-series}
Let $P=L_PN_P$ be a standard parabolic, and let $\tau$ and $V_{\underline\mu}$ satisfy the hypotheses of \Cref{lem:purity} for the Levi subgroup $L_P$. Then $\tau$ is the irreducible normalized principal series of $L_P$
\[
 \tau\simeq I_{L_P}(\chi_{\ulm}),
\]
where
\begin{equation}\label{eq:complex-tempered-parameter}
 d\chi_{\ulm}=-(w_{L_P},1)(\underline\mu+\underline\rho_{L_P})=\bigl(-w_{L_P}(\mu)+\rho_{L_P}-\mathsf w,
             \ w_{L_P}(\mu)-\rho_{L_P}-\mathsf w'\bigr).
\end{equation}
Its central character is the inverse of the coefficient's central character. Its bottom nonzero cohomology is one-dimensional:
\[
 \dim H^{q_0(L_P)}(\mathfrak l_P,K_{L_P}^\circ;
                    \tau\otimes V_{\underline\mu})=1.
\]
For every smooth character $\xi$ of $A_P^\circ$, extended trivially to $M_P$, the normalized parabolic induction of $\tau\otimes\xi$ is a principal series of $G$:
\[
 \Ind_P^G(\tau,\xi)
   \simeq I\bigl(\chi_{\ulm}\cdot(\xi|_T)\bigr).
\]
In particular, the unnormalized induction is
\[
 {}^a\Ind_P^G(\tau)
   \simeq I\bigl(\chi_{\ulm}\cdot(\delta_P^{-1/2}|_T)\bigr),
\]
whose inducing character for normalized induction from $B$ has differential $d\chi_{\ulm}-\underline\rho_P$.
\end{prop}
\begin{proof}
After passing to $L_P^1$ as in \Cref{lem:purity}, \cite[Chapter~III, Theorem~5.1 and Corollary~5.2]{borel-wallach} shows that $\tau|_{L_P^1}$ occurs in a unitary principal series with zero split parameter. For a complex group the inducing discrete-series datum is a character of a compact torus, and \eqref{eq:pre-resA} shows that the two derived exponents are opposite. The infinitesimal-character condition in \cite[Chapter~III, Theorem~3.3]{borel-wallach}, together with \Cref{lem:purity} and regularity of $\mu+\rho_{L_P}$, identifies this pair, up to the action of the same Weyl group element on both components, with
\[
 (-w_{L_P}(\mu)+\rho_{L_P},\;w_{L_P}(\mu)-\rho_{L_P}).
\]
Restoring the inverse central character places $\tau$ in a principal series representation $I_{L_P}(\chi_{\ulm})$ of $L_P$ with parameter \eqref{eq:complex-tempered-parameter}. The difference of the two exponents is
\[
 -w_{L_P}(\mu+\mathsf w)+(\mu'+\mathsf w')+2\rho_{L_P}\in\Lambda,
\]
so the single-valuedness criterion in \cref{ex:torus-characters} implies a smooth character of the full torus, even when the separated weights are not integral.

On every positive coroot of $L_P$, the two exponents are nonzero integers of opposite signs. The complex-root reducibility criterion \cite[Theorem~6.19]{speh-vogan} therefore shows that the containing principal series is irreducible, proving the realization of $\tau$.

In Borel--Wallach's cohomology formula \cite[III.3.3(ii)]{borel-wallach}, \eqref{eq:complex-tempered-parameter} selects the unique Kostant representative $(w_{L_P},1)$:
\[
 d\chi_{\ulm}+(w_{L_P},1)(\underline\mu+\underline\rho_{L_P})=0.
\]
Its length is $q_0(L_P^1)=q_0(L_P)$, and the degree-zero torus cohomology is one-dimensional. Applying \cite[Chapter~III, Theorem~3.3]{borel-wallach} on $L_P^1$ and restoring the center gives the stated bottom cohomology.

Finally, apply \eqref{eq:parabolic-induction-in-stages} with $\chi=\chi_{\ulm}$. Taking $\xi=\delta_P^{-1/2}$ realizes the unnormalized induction as the normalized principal series above, whose inducing differential follows from \eqref{eq:parabolic-modulus-differential}.
\end{proof}
\Cref{sec:irreducibility-principal-series} proves irreducibility of the induced representations of $G$ at the evaluation point under the criticality hypotheses introduced below.

\begin{exam}[Classical cohomological parameters]
\label{ex:classical-cohomological-parameters}
Let $\pi$ be an irreducible, essentially tempered representation that is cohomological with respect to $V_{\underline\lambda}$. Retain the derived and central decomposition
\[
 \underline\lambda=(\lambda+\mathsf w,\lambda'+\mathsf w'),
 \qquad \underline{\mathsf{w}}=(\mathsf w,\mathsf w').
\]
\Cref{lem:purity} requires $\lambda'=-w_G(\lambda)$. Applying \eqref{eq:complex-tempered-parameter} with $P=G$, we obtain $\pi\simeq I(\chi)$ with
\[
 d\chi=
 \bigl(-w_G(\lambda)+\rho_G-\mathsf w,\;
        w_G(\lambda)-\rho_G-\mathsf w'\bigr).
\]
On the derived torus $T^1$, the central terms vanish, so
\[
 d(\chi|_{T^1})
   =\bigl(-w_G(\lambda)+\rho_G,\;w_G(\lambda)-\rho_G\bigr).
\]
For the simply connected groups below, $-w_G(\lambda)+\rho_G$ is an integral weight of $T^1$. In the classical root coordinates of \cref{exam:maximal}, write $\lambda=\sum_i k_i e_i$. The table gives $\lambda'$ and the first component of the differential of the inducing character:
\[
\begin{array}{c|c|c}
 G^1 & \lambda' & (-w_G(\lambda)+\rho_G)_i\\ \hline
 \Spin_{2n+1} & \lambda & k_i+n-i+\tfrac12\\
 \Sp_{2n} & \lambda & k_i+n-i+1\\
 \Spin_{2n},\ n\text{ even} & \lambda & k_i+n-i\\
 \Spin_{2n},\ n\text{ odd}
   & (k_1,\ldots,k_{n-1},-k_n)
   & \begin{cases}
       k_i+n-i,&i<n,\\
       -k_n,&i=n.
     \end{cases}
\end{array}
\]
Half-integral spin coordinates are interpreted in the spin weight lattice.

For $\mathrm G_0=\GL_n$, write $\lambda=(\lambda_1,\ldots,\lambda_n)$ in the derived root space, so $\sum_i\lambda_i=0$, and write the rational extensions of the algebraic central weights to the full torus as $\mathsf w=(c,\ldots, c)$, $\mathsf w'=(c',\ldots,c')$, for $c, c'\in \frac{1}{n}\Z$. Thus the full coefficient weight is
\[
 \underline\lambda=
 \bigl((\lambda_i+ c)_{i=1}^n;\;
       (-\lambda_{n+1-i}+c')_{i=1}^n\bigr).
\]
Both lists of highest-weight coordinates are integral, although their summands need not be. With $\rho_{G,i}=(n+1-2i)/2$, the inducing character is
\[
 \chi(\operatorname{diag}(z_1,\ldots,z_n))
  =\prod_{i=1}^{n}
       z_i^{-\lambda_{n+1-i}+\rho_{G,i}-c}
       \bar z_i^{\,\lambda_{n+1-i}-\rho_{G,i}-c'}.
\]
The exponent difference is integral, so each factor is single-valued as in \cref{ex:torus-characters}. On scalar matrices,
\[
 \chi(zI_n)
  =z^{-nc}\bar z^{-nc'}
  =\underline{\mathsf w}(zI_n)^{-1}.
\]
The highest-weight coordinates in \cite[Section~2.5.1]{raghuram-imaginary} are $\lambda_i+c$ and $-\lambda_{n+1-i}+c'$ in our notation. His purity weight is the common value  $c+c'$.

The table follows from the longest Weyl element, which is $-1$ in types $B_n,C_n$ and in type $D_n$ with $n$ even. In type $D_n$ with $n$ odd, it fixes the last coordinate and negates the others. For $\GL_n$, $w_G$ reverses the coordinates and fixes both central weights, giving the displayed inducing character.
\end{exam}

\section{\bf \texorpdfstring{$L$}{L}-functions and criticality}
\label{sec:L-functions-criticality}

\subsection{Langlands-Shahidi \texorpdfstring{$L$}{L}-functions}
\label{subsec:LS_L-functions}
Let $\widehat G$ be the connected complex reductive group whose based root datum is dual to that of $\mathrm G_\C$. The Galois form of the $L$-group is
\[
 {}^LG=\widehat G\rtimes\operatorname{Gal}(\C/\R),
\]
with the pinned action induced by the Galois action on the based root datum \cite[Section~2.1]{borel}. In the Weil form, replace $\operatorname{Gal}(\C/\R)$ by $W_\R$, acting through the same quotient. In our restriction-of-scalars setting,
\[
 \widehat G=\widehat{\mathrm G}_0\times\widehat{\mathrm G}_0,\qquad
 {}^LG=(\widehat{\mathrm G}_0\times\widehat{\mathrm G}_0)
                  \rtimes\operatorname{Gal}(\C/\R),
\]
and the nontrivial Galois element interchanges the two factors. For example, when $\mathrm G_0=\mathrm{GL}_n$,
\[
 {}^LG=(\mathrm{GL}_n(\C)\times\mathrm{GL}_n(\C))
                  \rtimes\operatorname{Gal}(\C/\R).
\]

For the complex place itself, $W_\C=\C^\times$ and the Weil $L$-group of $\mathrm G_0$ is $\widehat{\mathrm G}_0\times W_\C$. An \textit{$L$-parameter} is therefore specified by a continuous homomorphism with semisimple image
\begin{equation}\label{eq:complex-L-parameter}
 \varphi:W_\C=\C^\times\longrightarrow\widehat{\mathrm G}_0,
\end{equation}
up to $\widehat{\mathrm G}_0$-conjugacy. The corresponding homomorphism into $\widehat{\mathrm G}_0\times W_\C$ is $z\mapsto(\varphi(z),z)$. The restriction-of-scalars correspondence identifies these with parameters for $G$: restrict a parameter $W_\R\to{}^LG$ to $W_\C$ and project to the first dual factor \cite[Section~8.2 and Proposition~8.4]{borel}. We use the complex-place description \eqref{eq:complex-L-parameter} for the local factors below.
\vspace{1em}

\noindent\textbf{\large L-parabolics and adjoint representations}
A standard $L$-parabolic ${}^LP\subset{}^LG$ projects onto $\mathrm{Gal}(\C/\R)$ and has a standard parabolic identity component in $\widehat G$. It is the normalizer of this identity component in ${}^LG$ \cite[Section 3.3]{borel}. Its identity component is stable under interchange of the two factors, so every standard $L$-parabolic here has the form
\[
 {}^LP=(\widehat P_0\times\widehat P_0)
           \rtimes\mathrm{Gal}(\C/\R),
\]
where $\widehat P_0$ is a standard parabolic of $\widehat{\mathrm G}_0$.

The duality $\Delta_{abs}(G)^\vee=\Delta(\widehat G)$ describes these parabolics through the absolute simple roots. Our positive systems give
\[
 \Delta_{abs}(G)
 =\{\alpha^{(1)}:\alpha\in\Delta_{res}(G)\}
   \sqcup
  \{\alpha^{(2)}:\alpha\in\Delta_{res}(G)\}.
\]
The root-restriction map in \eqref{eq:Res} has the two-element fiber
\[
 \operatorname{res}_{\mathrm{rt}}^{-1}(\alpha)
 =\{\alpha^{(1)},\alpha^{(2)}\}
 \qquad(\alpha\in\Delta_{res}(G)).
\]
Both lifts are complex roots and are exchanged, up to sign, by the Cartan involution.

Such a standard $L$-parabolic is \textit{relevant} if the simple roots of its Levi subgroup correspond to $\operatorname{res}_{\mathrm{rt}}^{-1}(I)$ for some $I\subseteq\Delta_{\rm res}(G)$. Every standard $L$-parabolic in the present setting is relevant, since the Levi subgroups of its two factors have the same simple roots.
\vspace{1em}

\noindent\textbf{\large The adjoint $L$-functions}
Let $\varphi$ be a complex $L$-parameter and let $r$ be a finite-dimensional algebraic representation of its target dual group. We apply this to the representations $r_j$ occurring in the adjoint action of $\widehat{\mathrm L}_{P,0}$ on $\widehat{\mathfrak n}_{P,0}$, defined below. The Weil representation $r\circ\varphi$ is a sum of characters of $W_\C$, each of the form
\[
 \psi(z)=\left(\frac{z}{|z|}\right)^m|z|_\C^t,\qquad
 m\in\Z,\quad t\in\C,\quad |z|_\C=z\bar z.
\]
Define
\begin{equation}\label{eq:parameter-character-L-factor}
 \Gamma_\C(u)=2(2\pi)^{-u}\Gamma(u),\qquad
 L_\C(u,\psi)=\Gamma_\C\left(u+t+\frac{|m|}{2}\right).
\end{equation}
If $r\circ\varphi=\bigoplus_k\psi_k$, counted with multiplicity, put
\begin{equation}\label{eq:parameter-L-factor}
 L_\C(u,\varphi,r):=L_\C(u,r\circ\varphi)
       =\prod_k L_\C(u,\psi_k).
\end{equation}
This is the complex Artin local factor \cite[Section~8.2, pp.~135--136]{shahidi-book}. It depends only on $\varphi$ and $r$, with $L$-function variable $u$. Compatibility with induction from $W_\C$ to $W_\R$ identifies it with the induced real Weil representation's factor. Thus the two dual factors in $\hat G$ contribute one complex-place factor.

Fix a standard maximal parabolic $P=L_PN_P$, corresponding to $I=\Delta_{\mathrm{res}}(G)\setminus\{\alpha_P\}$. Let $\widehat P_0=\widehat{\mathrm L}_{P,0}\widehat{\mathrm N}_{P,0}$ be the corresponding standard parabolic of $\widehat{\mathrm G}_0$. The adjoint action of $\widehat{L}_{P,0}$ on the dual unipotent radical gives
\[
 r:\widehat{\mathrm L}_{P,0}\longrightarrow
       \mathrm{GL}(\widehat{\mathfrak n}_{P,0}),\qquad
 \widehat{\mathfrak n}_{P,0}=\operatorname{Lie}(\widehat{\mathrm N}_{P,0}).
\]
The grading of its root-space decomposition uses the same integers as the character identity \eqref{eq:parabolic-root-twist}:
\[
 \widehat{\mathfrak n}_{P,0}
   =\bigoplus_{\alpha\in\Phi(P)}
                 \widehat{\mathfrak n}_{P,0,\alpha^\vee}.
\]
For $j=1,\cdots ,m$, put $\Phi_j(P)=\{\alpha\in\Phi(P): \langle\gamma_P,\alpha^\vee\rangle=j\}$, where $\gamma_P$ is the same as in \eqref{eq:parabolic-parameter-normalization}.  Let $r_j$ denote the representation on the sum of these root spaces. The grade $j$ is the coefficient of $\alpha_P^\vee$ in the simple-coroot expansion of $\alpha^\vee$. Adding a root of $\widehat{\mathrm L}_{P,0}$ preserves this coefficient, so these sums are $\widehat{\mathrm L}_{P,0}$-stable. Put
\[
 m=\max_{\alpha\in\Phi(P)}\langle\gamma_P,\alpha^\vee\rangle.
\]
Then
\[
 r=\bigoplus_{j=1}^m r_j.
\]
\begin{exam}
\label{ex:classical-adjoint-gradings}
The examples in \cite[Sections~5.1--5.3]{raghuram-cjm} give the following values of $m$. The grading depends only on the based root datum and the chosen parabolic, so the same calculations apply here. All groups in the table are evaluated at $\C$, assuming $a,b\geq1$ and $n\geq2$.
\[
\begingroup
\renewcommand{\arraystretch}{1.25}
\begin{array}{c|c|c}
 G & L_P & m\\ \hline
 \GL_{a+b} & \GL_a\times\GL_b & 1\\
 \mathrm{SO}_{2n+2} & \GL_1\times\mathrm{SO}_{2n} & 1\\
 \mathrm{Sp}_{2n} & \GL_n & 2
\end{array}
\endgroup
\]
Here $P$ is, respectively, the standard $(a,b)$-block parabolic, the stabilizer of an isotropic line, or the Siegel parabolic. Write $\operatorname{Std}_d$ for the standard $d$-dimensional representation.

In the general linear case, $\GL_a(\C)\times\GL_b(\C)$ acts on $\operatorname{Hom}(\C^b,\C^a)$ by $(g,h)\cdot X=gXh^{-1}$, so $r=r_1=\operatorname{Std}_a\boxtimes\operatorname{Std}_b^\vee$. In the orthogonal case, $r=r_1$ is the action $(t,h)\cdot v=t\,hv$ on $\C^{2n}$. In both cases every root space in the Lie algebra of the unipotent radical of the dual parabolic has grade $1$. For the orthogonal example we use the connected special orthogonal group. Passing to the identity component in \cite[Section~5.2]{raghuram-cjm} leaves the nilradical and its grading unchanged.

For the Siegel parabolic of $\mathrm{Sp}_{2n}$, the dual group is $\mathrm{SO}_{2n+1}(\C)$, and
\[
 r=r_1\oplus r_2,\qquad
 r_1=\operatorname{Std}_n,\qquad
 r_2=\bigwedge\nolimits^2\operatorname{Std}_n.
\]
The central element $tI_n\in\GL_n(\C)$ acts on these two summands by $t$ and $t^2$, respectively, giving grades $1$ and $2$. When $n=1$, the exterior square vanishes and $m=1$.
\end{exam}

\subsection{Local factors for the induced family and criticality}
\label{subsec:L-function}
Retain $P=L_PN_P$ and the family $I(\tau, s)$ of \eqref{eq:parabolic-central-deformation}, now with $\tau$ an irreducible admissible representation of $L_P$ with its specified central character. Write
\[
 \varphi_\tau:W_\C\longrightarrow\widehat{\mathrm L}_{P,0}
          \subset\widehat{\mathrm G}_0
\]
for its complex Langlands parameter. For each adjoint summand $r_j$ of \Cref{subsec:LS_L-functions}, put $V_j=r_j\circ\varphi_\tau$ and apply \eqref{eq:parameter-L-factor}.

When $\tau=I_{L_P}(\chi_{\ulm})$ is an irreducible normalized principal series, as in \Cref{prop:principal-series}, write $d\chi_{\ulm}=(\nu,\nu')$ and choose the representative of $\varphi_\tau$ whose image lies in the dual torus and which corresponds to $\chi_{\ulm}$ \cite[Section~11.4]{borel}. The characters obtained by composing $\chi_{\ulm}$ with the coroots are
\[
 \chi_{\ulm}\circ\alpha^\vee(z)
   =z^{p_\alpha}\bar z^{q_\alpha},\qquad
 p_\alpha=\langle\nu,\alpha^\vee\rangle,\quad
 q_\alpha=\langle\nu',\alpha^\vee\rangle.
\]
The product \eqref{eq:parameter-L-factor} becomes
\begin{equation}\label{eq:L_j}
 L_\C(u,\tau,r_j):=L_\C(u,\varphi_\tau,r_j)=L_\C(u,V_j)=\prod_{\alpha\in\Phi_j(P)} L_\C(u,\chi_{\ulm}\circ\alpha^\vee).
\end{equation}
Combining \eqref{eq:parabolic-root-twist} with \eqref{eq:parameter-character-L-factor} gives
\begin{equation}\label{eq:induced-L-factor-shift}
 L_\C\bigl(u,\tau\otimes\delta_P^{-s/(2s_0)},r_j\bigr)
       =L_\C(u+js,\tau,r_j).
\end{equation}
This explains the arguments $js$ and $1+js$ in the normalizing ratios for the induced family.

We now assume that the root exponents $p_\alpha,q_\alpha$ belong to $\tfrac12\Z$. Single-valuedness requires $p_\alpha-q_\alpha\in\Z$ by \cref{ex:torus-characters}. For such exponents, \eqref{eq:parameter-character-L-factor} becomes
\begin{equation}\label{eq:local_L_factor}
 L_\C(u,z^p\bar z^q)
   =\Gamma_\C\left(u+\frac{p+q}{2}+\frac{|p-q|}{2}\right)
   =\Gamma_\C(u+\max\{p,q\}).
\end{equation}
Set
\[
 c_\alpha=\max\{p_\alpha,q_\alpha\},\qquad
 c_\alpha^\vee=\max\{-p_\alpha,-q_\alpha\}.
\]
The dual $V_j^\vee$ has the inverse root characters $z^{-p_\alpha}\bar z^{-q_\alpha}$. Thus \eqref{eq:local_L_factor} implies
\[
 L_\C(u,V_j^\vee)=\prod_{\alpha\in\Phi_j(P)}\Gamma_\C(u+c_\alpha^\vee).
\]
Each original factor has poles at $u=-c_\alpha-\mathbb N$, where $\mathbb N=\{0,1,2,\ldots\}$.

\begin{defn}[\bf Critical set]\label{def:critical_set}
For each $1\leq j\leq m$, define
\[
 \mathrm{Crit}_j=\left\{u\in\R:\begin{array}{l}
 u-c_\alpha\in\Z\quad\text{for all }\alpha\in\Phi_j(P),\\
 L_\C(u,V_j)\text{ and }L_\C(1-u,V_j^\vee) \text{ are finite}
 \end{array}\right\}.
\]
We call $s$ a \textit{hyper-critical point} of the induced family if both $js$ and $1+js$ lie in $\mathrm{Crit}_j$ for every grade $j$.
\end{defn}

\section{\bf Intertwining operators and the rank-one computation}
\label{sec:inter_op}

\subsection{Standard intertwining operators for principal series}
\label{sec:rank-one-int-op}
The representations used in the main theorem are normalized principal series induced from $B$ by \Cref{prop:principal-series} and \eqref{eq:parabolic-induction-in-stages}. We introduce the intertwining operators in this realization.

The restricted Weyl group $W_{\mathrm{res}}(G)$ of \eqref{eq:res_weyl} acts on smooth characters of $T$ by $({}^{w}\chi)(t)=\chi(w^{-1}tw)$.  In this formula and in the integrals below, we use the representatives fixed in \Cref{subsec:root-homomorphisms}. On weights and differentials we retain the notation $w(\lambda)$; thus, if $d\chi=(\nu,\nu')$, then $d({}^{w}\chi)=(w(\nu),w(\nu'))$. Write $l(w)$ for length with respect to the simple restricted reflections. Thus, \(l(w)=\#\{\alpha\in\Phi^+_{\mathrm{res}}(G):w^{-1}(\alpha)<0\}\).

Write $B=TN_B$, and let $\bar N_B$ be the unipotent radical of the opposite Borel subgroup. For $w\in W_{\mathrm{res}}(G)$, put
\[
 N_w=N_B\cap w\bar N_Bw^{-1}.
\]
Its roots are the positive roots $\alpha$ with $w^{-1}(\alpha)<0$. For a smooth character $\chi$ of $T$ and $w\in W_{res}(G)$, the standard intertwining operator between the normalized principal series
\[
 \operatorname{T}_w(\chi):I(\chi)\longrightarrow I({}^{w}\chi),
\]
in its domain of absolute convergence, is defined by
\begin{equation}
\label{eq:inter_op}
 (\operatorname{T}_w(\chi)f)(g)
    =\int_{N_w}f(w^{-1}ug)\,du.
\end{equation}
The measure $du$ is the product of the measures $2\,dx\,dy$ in the fixed complex root coordinates, as in \eqref{eq:rk1-measure}.  On each root subgroup this is the self-dual additive measure for the character $z\mapsto\exp\left(2\pi i(\operatorname{Tr}_{\C/\R}z)\right)$, where $\operatorname{Tr}_{\C/\R}z=z+\bar z=2\operatorname{Re}z$.

Write $d\chi=(\nu,\nu')$. The integral \eqref{eq:inter_op} converges absolutely on every smooth section whenever
\[
 \operatorname{Re}\langle\nu+\nu',\beta^\vee\rangle>0
 \qquad
 \text{for every }\beta\in\Phi_{\mathrm{res}}^+(G)
 \text{ with }w(\beta)<0.
\]
It admits meromorphic continuation as the split parameters vary with $\chi|_{T\cap C}$ fixed; see \cite[Theorem~4.2]{knapp-stein-II}. These inequalities specify a region of absolute convergence. 

For the factorization below, see \cite[Section 4.1]{shahidi-book}.
\begin{lemma}\label{lem:standard-factorization}
With compatible Weyl representatives and product measures, if $l(w_1w_2)=l(w_1)+l(w_2)$, then
\[
 \operatorname{T}_{w_1w_2}(\chi)
   =\operatorname{T}_{w_1}({}^{w_2}\chi)
       \circ\operatorname{T}_{w_2}(\chi).
\]
\end{lemma} 

For a simple reflection, write $\operatorname{T}_\alpha^G(\chi)=\operatorname{T}_{s_\alpha}(\chi)$. The same construction applies to principal series of a standard Levi subgroup. We suppress the character when it is fixed by the context.

\subsection{The rank-one model}
\label{sec:rk1-model}
\Cref{lem:standard-factorization} reduces the computation to the operators attached to simple reflections. Put $H=\mathrm{SL}_2(\mathbb C)$, with upper triangular Borel subgroup $B_{H}$ and maximal compact subgroup $C_{H}=\mathrm{SU}(2)$. The root homomorphism in \eqref{eq:rk1-root-map} relates this model to $G$. At a given step of the factorization, let $\chi$ be the normalized inducing character, including its parameter shifts and Weyl conjugations. Its pullback along $\alpha^\vee$ is
\begin{equation}\label{eq:rk1-root-character}
 \psi_\alpha=\chi\circ\alpha^\vee,\qquad
 \psi_\alpha(z)=z^{p_\alpha}\bar z^{q_\alpha},\qquad
 p_\alpha=\langle\nu,\alpha^\vee\rangle,\quad
 q_\alpha=\langle\nu',\alpha^\vee\rangle
 \quad(d\chi=(\nu,\nu')).
\end{equation}
The scalars $p_\alpha,q_\alpha$ are the weights evaluated on the coroot.  To check that the normalized inductions agree, restrict the half-modulus as well.  Since $\alpha$ is simple and $\rho_G$ is the half-sum of the positive roots of $\mathrm G_0$,
\begin{equation}\label{eq:rk1-modulus}
 \delta_B^{1/2}(\alpha^\vee(z))
   =|z|_{\mathbb C}^{\langle\rho_G,\alpha^\vee\rangle}
   =|z|_{\mathbb C}
   =\delta_{B_{H}}^{1/2}(a(z)),\qquad |z|_{\mathbb C}=z\bar z.
\end{equation}
Thus pullback along $\varphi_\alpha$ gives the normalized inducing character $\psi_\alpha$ on $H$, without a further shift.

For a smooth character $\psi(z)=z^p\bar z^q$, with $p-q\in\mathbb Z$, write
\[
 z^p\bar z^q=|z|^{p+q}(z/|z|)^{p-q},
\]
where $|z|$ is the ordinary modulus.  Our normalized principal series is
\begin{equation}\label{eq:rk1-induction}
 I_{H}(\psi)=\left\{f\in C^\infty(H):
 f\left(\begin{pmatrix}z&v\\0&z^{-1}\end{pmatrix}h\right)
    =|z|_{\mathbb C}\psi(z)f(h)\right\},
\end{equation}
with $H$ acting by right translations.  The standard operator $\operatorname{T}_{H}(\psi):I_{H}(\psi)\longrightarrow I_{H}(\psi^{-1})$ is:
\begin{equation}\label{eq:rk1-integral}
 \bigl(\operatorname{T}_{H}(\psi)f\bigr)(h)
    =\int_{\mathbb C} f(w_2^{-1}n(u)h)\,d\mu_{\mathbb C}(u).
\end{equation}
It converges in the usual half-plane $\Re(p+q)>0$ and is otherwise understood by standard meromorphic continuation.  Pullback along $\varphi_\alpha$ gives
\[
 (\operatorname{T}_\alpha^G f)\circ\varphi_\alpha
    =\operatorname{T}_{H}(\psi_\alpha)(f\circ\varphi_\alpha).
\]
This identity compares the pullbacks of the principal-series functions along $\varphi_\alpha$.  Its compatibility with the prescribed cohomological generators is proved separately in \Cref{subsec:main-first-transitivity}.
\vspace{1em} 

\noindent\textbf{\large Test vectors and the analytic scalar}
To compute the action on the lowest compact type, put
\begin{equation}\label{eq:rk1-parameters}
 k=p-q\in\mathbb Z,\qquad D=|k|,\qquad
 A(\psi)=\frac{p+q+D}{2}.
\end{equation}
The compact inducing character is $a(e^{i\vartheta})\mapsto e^{ik\vartheta}$.  The $C_{H}$-type $V_D=\Sym^D(\mathbb C^2)$ occurs with multiplicity one in $I_{H}(\psi)$ and $I_{H}(\psi^{-1})$.  Indeed, its weights are $D,D-2,\ldots,-D$, each with multiplicity one.  The assertion then follows from compact Frobenius reciprocity.  It is the lowest occurring type, and the remaining compact types have highest weights $D+2j$, $j\geq 1$.

For $c=(c_{ij})\in C_{H}$, define
\begin{equation}\label{eq:rk1-test-vectors}
 \begin{array}{c|cc}
     & f_\psi(c) & \widetilde f_\psi(c)\\ \hline
 k\geq0 & c_{11}^{D} & c_{21}^{D}\\
 k<0 & c_{22}^{D} & (-c_{12})^{D}.
 \end{array}
\end{equation}
Extend the two functions by \eqref{eq:rk1-induction} for $\psi$ and $\psi^{-1}$, respectively.  Under the right action, they span extremal weight lines of the same weight for the diagonal torus in $C_H$ and are normalized by
\begin{equation}\label{eq:rk1-point-normalization}
 f_\psi(1)=1,\qquad \widetilde f_\psi(w_2)=1.
\end{equation}
When $k=0$, both functions restrict to the constant function $1$ on $C_H$.

\begin{prop}\label{prop:rk1-scalar}
With the representative and measure in \eqref{eq:rk1-root-map}--\eqref{eq:rk1-measure}, one has the meromorphic identity
\begin{equation}\label{eq:rk1-scalar}
 \operatorname{T}_{H}(\psi)f_\psi=c(\psi)\widetilde f_\psi,
 \qquad
 c(\psi)=\frac{L_{\mathbb C}(0,\psi)}{L_{\mathbb C}(1,\psi)}
       =\frac{\Gamma_{\mathbb C}(A(\psi))}
                    {\Gamma_{\mathbb C}(A(\psi)+1)}
       =\frac{2\pi}{A(\psi)},
\end{equation}
where $\Gamma_{\mathbb C}(v)=2(2\pi)^{-v}\Gamma(v)$ and $L_{\mathbb C}(v,\psi)=\Gamma_{\mathbb C}(v+A(\psi))$.
\end{prop}
\begin{proof}
First suppose $\Re(p+q)>0$.  $C_H$-equivariance and multiplicity one reduce the operator on the test vector to a scalar. By \eqref{eq:rk1-point-normalization}, evaluation at $w_2$ determines this scalar.  Write $\Delta_u=(1+|u|^2)^{1/2}$.  Then
\begin{equation}
 w_2^{-1}n(u)w_2=\begin{pmatrix}1&0\\-u&1\end{pmatrix}=\begin{pmatrix}\Delta_u^{-1}&-\bar u\Delta_u^{-1}\\
                   0&\Delta_u\end{pmatrix}
   \begin{pmatrix}\Delta_u^{-1}&\bar u\Delta_u^{-1}\\
                   -u\Delta_u^{-1}&\Delta_u^{-1}\end{pmatrix}.
   \end{equation}
The second matrix is in $C_{H}$.  In both rows of \eqref{eq:rk1-test-vectors}, $f_\psi$ evaluates to $\Delta_u^{-D}$ on this matrix, while the induction law \eqref{eq:rk1-induction} gives the factor $\Delta_u^{-(p+q+2)}$ from the first matrix. Consequently
\[
 f_\psi(w_2^{-1}n(u)w_2)
      =(1+|u|^2)^{-A(\psi)-1},
\]
and polar coordinates with the measure \eqref{eq:rk1-measure} give
\begin{equation}
 (\operatorname{T}_{H}(\psi)f_\psi)(w_2)=4\pi\int_0^\infty r(1+r^2)^{-A(\psi)-1}\dd r=2\pi\int_0^\infty(1+v)^{-A(\psi)-1}\dd v
   =\frac{2\pi}{A(\psi)}.
\end{equation}
The Gamma quotient in \eqref{eq:rk1-scalar} follows from $\Gamma(v+1)=v\Gamma(v)$. To extend the calculation beyond the region of convergence, vary $\psi_\zeta=\psi|\cdot|_{\mathbb C}^{\zeta}$. Then both $p$ and $q$ increase by $\zeta$, the integer $k$ and the restrictions of the test vectors to $C_H$ stay fixed, and $A(\psi_\zeta)=A(\psi)+\zeta$. The identity therefore continues meromorphically from a right half-plane.
\end{proof}

\begin{rem}\label{rem:rk1-scope}
Formula \eqref{eq:rk1-scalar} determines the intertwining operator on the lowest compact type $V_D$. Even when this scalar has no pole, the formula alone does not exclude poles on other compact types. Before specializing the operator and passing to cohomology, we must therefore establish holomorphy on the whole principal series. \Cref{lem:rk1-regularity} proves this at the cohomological parameters by combining the scalar formula with irreducibility.
\end{rem}
\section{\bf Combinatorial lemma and irreducibility}
\label{sec:prepmainthm}

We first compare Weyl representatives for associate parabolics. Criticality will then determine the highest weight of a coefficient representation with nonzero cohomology and exclude the reducibility obstruction.

\subsection{Weyl groups and associate parabolics}
Let $G$ be as in \Cref{sec:nota}, and let $P=L_PN_P$ be a standard maximal parabolic. Thus $I=\Delta_{res}(G)\setminus\{\alpha_P\}$ is the set of simple roots of $L_P$. Denote the corresponding complex algebraic unipotent radical by $\mathrm N_{P,0}$. Put
\[
 d_P=\#\Phi(P),\qquad
 \Phi(P)=\Phi_0^+(G)\setminus\Phi_0^+(L_P).
\]

\begin{prop}
\label{prop:dimension}
One has
\[
 \dim_{\R}N_P
 =2\dim_{\C}\mathrm N_{P,0}
 =2d_P=2\bigl(q_0(G)-q_0(L_P)\bigr).
\]
\end{prop}
\begin{proof}
Each root space in $\operatorname{Lie}_{\C}\mathrm N_{P,0}$ is one-dimensional over $\C$, so $\dim_{\C}\mathrm N_{P,0}=d_P$. 

The underlying real unipotent group $N_P$ has dimension $2d_P$. Finally, \eqref{eq:complex-bottom-degree} implies that $q_0(G)-q_0(L_P)=d_P$.
\end{proof}

Lengths in the restricted and absolute Weyl groups use $\Delta_{res}(G)$ and $\Delta_{abs}(G)$, respectively. The absolute Weyl group $W_{abs}(L_P)$ is generated by $\{s_{\ual}:\ual\in\Delta_{abs}(L_P)\}$, where
\[
 \Delta_{abs}(L_P)=\Delta_{abs}(G)\setminus
                    \operatorname{res}_{\mathrm{rt}}^{-1}(\alpha_P).
\]
Write $W_{abs}^P$ for the set of Kostant representatives which are the minimal-length representatives for the cosets $W_{abs}(L_P)\backslash W_{abs}(G)$. Define $W_{res}^P$ similarly for $W_{res}(L_P)\backslash W_{res}(G)$. Their defining inequalities are
\begin{equation}
\label{eq:weyl_P}
\begin{aligned}
 W_{abs}^P
   &=\left\{w\in W_{abs}(G)\mid
       w^{-1}(\ual)>0,\quad\forall\ual\in\Delta_{abs}(L_P)\right\},\\
 W_{res}^P
   &=\left\{w\in W_{res}(G)\mid
       w^{-1}(\alpha)>0,\quad\forall\alpha\in\Delta_{res}(L_P)\right\}.
\end{aligned}
\end{equation}
Under the factorwise identification $W_{abs}(G)\cong W_{res}(G)\times W_{res}(G)$, one has $W_{abs}^P\cong W_{res}^P\times W_{res}^P$; in particular, if $s=(w_1,w_2)\in W_{abs}^P$, then $w_1,w_2\in W_{res}^P$ and $l(s)=l(w_1)+l(w_2)$.

The following involution pairs complementary lengths.
\begin{lemma}
    \label{lem:dual_weyl}
   Let $w_G$ and $w_{L_P}$ be the longest elements in $W_{res}(G)$ and $W_{res}(L_P)$, respectively.  The map
    \[
        w \mapsto w^{\vee} := w_{L_P} w w_{G}
    \]
    is an involution of $W_{res}^{P}$.  For every $w\in W_{res}^{P}$,
    \[
        l(w) + l(w^{\vee}) = d_P.
    \]
\end{lemma}

\begin{proof}
For $\alpha\in\Delta_{res}(L_P)$, the root $w_{L_P}^{-1}(\alpha)$ is a negative Levi simple root.  Minimality of $w$ therefore requires $w^{-1}(w_{L_P}^{-1}(\alpha))<0$.  Since $w_G^{-1}$ sends negative roots to positive roots,
\[
 (w^\vee)^{-1}(\alpha)=w_G^{-1}(w^{-1}(w_{L_P}^{-1}(\alpha)))>0.
\]
Thus $w^\vee\in W_{res}^P$.  Both longest elements are involutions, so $(w^\vee)^\vee=w$.

Minimality makes length additive under multiplication on the left by elements of $W_{res}(L_P)$, so $l(w_{L_P}w)=l(w_{L_P})+l(w)$.  Applying the longest-element identity to $w_{L_P}w$, we obtain
\[
 l(w^\vee)=l(w_{L_P}ww_G)
 =l(w_G)-l(w_{L_P})-l(w)=d_P-l(w),
\]
because $l(w_G)-l(w_{L_P})=\#\Phi(P)=d_P$.
\end{proof}

\begin{cor}
    \label{cor:longest_Wp}
    Let $w^{P}$ be the longest element in $W_{res}^P$. Then
    \[
       w^{P}=w_{L_P}w_{G}.
    \]
\end{cor}
\begin{proof}
    Apply \Cref{lem:dual_weyl} to the identity element $\mathrm{Id}\in W_{res}^P$, which has length zero:
    \[
       \mathrm{Id}^{\vee}=w_{L_P}w_G\in W_{res}^P,\qquad
       l(w_{L_P}w_G)=d_P.
    \]
    For every $w\in W_{res}^P$, the lemma shows that $l(w)\leq d_P$, with equality precisely when $w^\vee=1$. Therefore one has $w=w_{L_P}w_G$.
\end{proof}

We shall distinguish this Kostant representative from its inverse:
\[
    w^{P}:=w_{L_P}w_G,\qquad
    w_P:=(w^{P})^{-1}=w_Gw_{L_P}.
\]
Thus $w^{P}$ is the longest minimal representative for $W_{res}(L_P)\backslash W_{res}(G)$, and $w_P$ is the element used in the standard intertwining operator from induction along $P$ to induction along the associate parabolic defined below. In general $w^{P}\neq w_P$.

We call two standard parabolic subgroups associate if their Levi subgroups are Weyl-conjugate.  The associate $Q$ used below is the standard parabolic determined by
\[
    \Delta_{res}(L_Q)=-w_G\bigl(\Delta_{res}(L_P)\bigr).
\]

\begin{cor}
    \label{cor:associate_parabolic}
    One has
    \[
       w_P\bigl(\Delta_{res}(L_P)\bigr)=\Delta_{res}(L_Q),
       \qquad
       w^{P}\bigl(\Delta_{res}(L_Q)\bigr)=\Delta_{res}(L_P).
    \]
\end{cor}
\begin{proof}
    Since $w_{L_P}(\Delta_{res}(L_P))=-\Delta_{res}(L_P)$, one has $$w_P(\Delta_{res}(L_P))=w_G(w_{L_P}(\Delta_{res}(L_P)))=-w_G(\Delta_{res}(L_P))=\Delta_{res}(L_Q).$$
    The second identity follows by applying $w_P^{-1}=w^{P}$.
\end{proof}

For a representation $\tau$ of $L_P$, its transport to a representation of $L_Q$ is
\[
 {}^{w_P}\tau(l_Q)=\tau(w_P^{-1}l_Qw_P).
\]
For a character $\xi$ of $L_P$, we write ${}^{w_P}\xi$ for its transport, defined by the same formula. This notation applies in particular to $\delta_P$.

\begin{lemma}
    \label{lem:associate_weyl}
    The map
    \[
        w\longmapsto w':=w_Pw
    \]
    defines a bijection from $W_{res}^{P}$ to $W_{res}^{Q}$. Moreover, for every $w\in W_{res}^{P}$,
    \[
        l(w)+l(w')=d_P.
    \]
\end{lemma}
\begin{proof}
    By \Cref{cor:associate_parabolic},
    \[
       (w')^{-1}\bigl(\Delta_{res}(L_Q)\bigr)
       =w^{-1}\bigl(w_P^{-1}(\Delta_{res}(L_Q))\bigr)
       =w^{-1}\bigl(\Delta_{res}(L_P)\bigr)>0.
    \]
    Hence $w'\in W_{res}^{Q}$, and the same argument with $w_P^{-1}$ gives the inverse bijection.

    Since $w$ is a minimal representative for $W_{res}(L_P)\backslash W_{res}(G)$,
    \[
       l(w_{L_P}w)=l(w_{L_P})+l(w).
    \]
    Applying the longest-element identity to $w_{L_P}w$, we obtain
    \[
       l(w_Pw)
       =l(w_Gw_{L_P}w)
       =l(w_G)-l(w_{L_P})-l(w)
       =d_P-l(w).
    \]
\end{proof}

For $\ulw=(w_1,w_2)\in W_{abs}^P$, define $\ulw'=(w'_1, w'_2)$, i.e.,
\[
   \ulw':=(w_Pw_1,w_Pw_2)=\jmath(w_P)\ulw.
\]
\Cref{lem:associate_weyl} and \Cref{prop:dimension} give
\[
   l(\ulw')=2d_P-l(\ulw)=\dim_{\R}(N_P)-l(\ulw).
\]
In particular, if $l(\ulw)=d_P$, then $l(\ulw')=d_P=l(\ulw)$.

\subsection{The combinatorial lemma}
\label{subsec:combinatorial_lemma}
Write $\mathrm P_0=\mathrm L_{P,0}\mathrm N_{P,0}$ for the complex algebraic parabolic corresponding to $P=L_PN_P$.  We identify the complex root systems $\Phi_0(G)$ and $\Phi_0(L_P)$ with the corresponding restricted systems. \Cref{prop:dimension} indicates that
\[
 d_P=\#\Phi(P)=\dim_{\mathbb C}\mathrm{N}_{P,0}
                =\tfrac12\dim_{\mathbb R}N_P.
\]
The half-sums $\rho_G$ and $\rho_{L_P}$ count each complex root once, and $\underline\rho_G=(\rho_G,\rho_G)$. Write $W_0=W_{\rm res}(G)$, and retain $w_G,w_{L_P}$ and the minimal \emph{left}-coset representatives from \eqref{eq:weyl_P}:
\[
 W_{\rm res}^P
 =\{w\in W_0:w^{-1}(\Phi_0^+(L_P))\subset\Phi_0^+(G)\},
 \qquad W_{\rm abs}^P=W_{\rm res}^P\times W_{\rm res}^P.
\]
\vspace{1em}

\noindent\textbf{\large The coefficient representation and the inducing character.}
Let $\underline\mu$ be the dominant integral highest weight of an algebraic coefficient representation of $L_P$ satisfying the purity condition of \Cref{lem:purity}. Using the rational derived and central decomposition of \Cref{cor:finite_dim_repn}, write
\begin{equation}
\label{eq:comb-full-coefficient}
 \underline\mu=
   (\mu+\mathsf w,\,-w_{L_P}(\mu)+\mathsf w')
       \in X^*(\mathrm{T}_{\mathbb C})=\Lambda\oplus\Lambda.
\end{equation}
Here $\mu$ is dominant for $L_P$, and $\mathsf w,\mathsf w'$ vanish on the coroots of $L_P$. The full weight in \eqref{eq:comb-full-coefficient} is integral and includes the central character. Its derived and central summands need not lie in $\Lambda$.

Let $\tau$ be the essentially tempered representation of $L_P$, cohomological with respect to $V_{\ulm}$, from \Cref{prop:principal-series}. Its normalized realization $\tau=I_{L_P}(\chi_{\ulm})$ has parameter \eqref{eq:complex-tempered-parameter}, which we write as
\begin{equation}
\label{eq:comb-unshifted-parameter}
 \mathrm{d}\chi_{\ulm}=\underline\nu=(-w_{L_P}(\mu)+\rho_{L_P}-\mathsf w,\ w_{L_P}(\mu)-\rho_{L_P}-\mathsf w').
\end{equation}
Pairing these weights with $\alpha^\vee$, we get the scalar exponents for a root $\alpha$. Recall \eqref{eq:parabolic-parameter-normalization} and \Cref{def:point_of_evaluation}:
\[
 s_0=-\langle\rho_P,\alpha_P^\vee\rangle,
 \qquad \rho_P=-s_0\gamma_P.
\]
For the induced family \eqref{eq:parabolic-modulus-family} with a twisting character, set $\chi_{\ulm,s}=\chi_{\ulm}\cdot(\delta_P^{-s/(2s_0)}|_T)$. Equations~\eqref{eq:parabolic-induction-in-stages} and \eqref{eq:parabolic-twist-differential} give
\[
 I(\tau, s)=\Ind_P^G(\tau,\delta_P^{-s/(2s_0)})\simeq I(\chi_{\ulm,s}), d\chi_{\ulm,s}=(\nu+s\gamma_P,\nu'+s\gamma_P).
\]
At the evaluation point, $I(\tau, s_0)={}^a\Ind_P^G(\tau)$ by \eqref{eq:parabolic-family-law}. Since $s_0\gamma_P=-\rho_P$, the parameter formula at $s=s_0$ becomes
\begin{equation}
\label{eq:comb-evaluation-parameter}
 \dd\chi_{\ulm, s_0}=(\nu-\rho_P,\ \nu'-\rho_P)=(-w_{L_P}(\mu)-\mathsf w+\rho_{L_P}-\rho_P, w_{L_P}(\mu)-\rho_G-\mathsf w').
\end{equation}
\vspace{0.5 em}
To simplify the notation in what follows, set
\begin{equation}
\label{eq:comb-xi}
 \xi=-w_{L_P}(\mu)-\mathsf w+\rho_{L_P}-\rho_P,\quad \xi'  =w_{L_P}(\mu)-\rho_G-\mathsf w'.
\end{equation}
Then $\dd\chi_{\ulm, s_0}=(\xi, \xi'  )$.
\vspace{1em}

\noindent\textbf{\large Criticality and the lattice condition.}
For $\alpha\in\Phi(P)$ put $j=\langle\gamma_P,\alpha^\vee\rangle$ and write $\alpha\in\Phi_j(P)$.  Thus
\begin{equation}
\label{eq:s_0}
 -js_0=\langle\rho_P,\alpha^\vee\rangle.
\end{equation}
Let $V_j=r_j\circ\varphi_\tau$ be the Weil representation as in \eqref{eq:L_j}, and recall:
\begin{equation}
\label{eq:c_alpha}
 c_\alpha =\max\{\langle\nu,\alpha^\vee\rangle,
                         \langle\nu',\alpha^\vee\rangle\},\quad  c_\alpha^\vee =\max\{-\langle\nu,\alpha^\vee\rangle,
                         -\langle\nu',\alpha^\vee\rangle\}.
\end{equation}
Then \eqref{eq:L_j} and \eqref{eq:local_L_factor} give
\[
 L_{\mathbb C}(u,V_j)=\prod_{\alpha\in\Phi_j(P)}
                  \Gamma_{\mathbb C}(u+c_\alpha),\qquad
 L_{\mathbb C}(u,V_j^\vee)=\prod_{\alpha\in\Phi_j(P)}
                  \Gamma_{\mathbb C}(u+c_\alpha^\vee).
\]
The superscript $\vee$ here denotes the dual Weil representation, not passage to the representation induced from $Q$.

Rewriting \eqref{eq:comb-xi} makes the lattice condition explicit:
\begin{equation}
\label{eq:comb-xi-lattice}
 \xi=-w_{L_P}(\mu+\mathsf w+\rho_G),\qquad
 -\xi' =-w_{L_P}(\mu)+\mathsf w'+\rho_G.
\end{equation}
Indeed, $w_{L_P}(\rho_{L_P})=-\rho_{L_P}$, $w_{L_P}(\rho_P)=\rho_P$, and $w_{L_P}(\mathsf w)=\mathsf w$. Since \eqref{eq:comb-full-coefficient} is an algebraic weight and $\rho_G$ is integral on coroots, both $\xi$ and $\xi'$ take integer values on every coroot of $G_0$. Thus \eqref{eq:s_0}, \eqref{eq:c_alpha}, and \eqref{eq:comb-xi} give, for $\alpha\in\Phi_j(P)$,
\begin{equation}
\label{eq:comb-integral-gamma-arguments}
 \begin{aligned}
 js_0+c_\alpha
 &=\max\{\langle\xi,\alpha^\vee\rangle,
                  \langle\xi' ,\alpha^\vee\rangle\}\in\mathbb Z,\\
 -js_0+c_\alpha^\vee
 &=\max\{-\langle\xi,\alpha^\vee\rangle,
                   \langle-\xi' ,\alpha^\vee\rangle\}\in\mathbb Z.
 \end{aligned}
\end{equation}

\begin{lemma}[\bf Combinatorial lemma]
\label{lem:combinatorial_lemma}
Suppose that $s_0$ is hyper-critical in the sense of \Cref{def:critical_set}. In other words, $js_0$ and $1+js_0$ are both critical for $V_j$ for every $j$ occurring in the graded adjoint representation, ensuring that $L_{\mathbb C}(u,V_j)$ and $L_{\mathbb C}(1-u,V_j^\vee)$ are finite at both $u=js_0$ and $u=1+js_0$. Then
\begin{equation}
\label{eq:critical_condition}
 \langle\xi,\alpha^\vee\rangle
 \langle-\xi' ,\alpha^\vee\rangle>0
 \qquad(\alpha\in\Phi_0(G)).
\end{equation}
There is a unique $w_\xi\in W_{\rm res}^P$ such that both $w_\xi^{-1}(\xi)$ and $-w_\xi^{-1}(\xi')$ are strictly dominant. Put
\begin{equation}
\label{eq:comb-kostant-representative}
 w_\xi^\vee=w_{L_P}w_\xi w_G,\qquad
 \ulw_{\xi}=(w_\xi^\vee,w_\xi).
\end{equation}
Then $\ulw_{\xi}\in W_{\rm abs}^P$ and
\begin{equation}
\label{eq:comb-conclusion}
 \begin{aligned}
 &\ell_{\rm abs}(\ulw_{\xi})=d_P=\tfrac12\dim_{\mathbb R}N_P,\\
 &\underline\lambda:=\ulw_{\xi}^{-1}
        (\underline\mu+\underline\rho_G)-\underline\rho_G
 \quad\text{is dominant and algebraic.}
 \end{aligned}
\end{equation}
\end{lemma}

\begin{proof}
We first use hyper-criticality to show that $\xi$ and $-\xi'$ have the same nonzero sign on every coroot. This will place the two weights in a common Weyl chamber and determine $w_\xi$; we can then use the complementary-length involution to construct $\ulw_\xi$ and check the dominance and algebraicity of its coefficient weight.

Fix a root $\alpha\in\Phi_j(P)$. Of the finiteness conditions supplied by hyper-criticality, we use $L_{\mathbb C}(js_0,V_j)$ at the first critical point and $L_{\mathbb C}(-js_0,V_j^\vee)$ at the second: the latter is the dual value $L_{\mathbb C}(1-u,V_j^\vee)$ with $u=1+js_0$. Since $\Gamma_{\mathbb C}$ has no zeros, a pole in an individual root factor cannot cancel in either product. Thus the factors $\Gamma_{\mathbb C}(js_0+c_\alpha)$ and $\Gamma_{\mathbb C}(-js_0+c_\alpha^\vee)$ are both finite. Their arguments are integers by \eqref{eq:comb-integral-gamma-arguments}, and $\Gamma_{\mathbb C}$ has a pole at every nonpositive integer, so these arguments must be strictly positive:
\begin{equation}
\label{eq:comb-critical-bounds}
 \begin{array}{rcll}
 c_\alpha&>&-js_0
    &\text{from criticality at }js_0,\\
 c_\alpha^\vee&>&js_0
    &\text{from dual finiteness at }1+js_0.
 \end{array}
\end{equation}
Together, the two inequalities place $-js_0$ in the open interval
\begin{equation}
\label{eq:comb-double-bound}
 c_\alpha>-js_0>-c_\alpha^\vee.
\end{equation}

By \eqref{eq:c_alpha}, the endpoints of this interval are the larger and smaller of $\langle\nu,\alpha^\vee\rangle$ and $\langle\nu',\alpha^\vee\rangle$. Passing to the parameter at $s_0$ adds $js_0$ to both evaluations, because \eqref{eq:s_0} and \eqref{eq:comb-evaluation-parameter} identify them as
\[
 \langle\xi,\alpha^\vee\rangle
    =\langle\nu,\alpha^\vee\rangle+js_0,\qquad
 \langle\xi',\alpha^\vee\rangle
    =\langle\nu',\alpha^\vee\rangle+js_0.
\]
Consequently, zero lies strictly between $\langle\xi,\alpha^\vee\rangle$ and $\langle\xi',\alpha^\vee\rangle$, so the two possible orders are
\begin{equation}
\label{eq:comb-case-positive}
 \langle\xi,\alpha^\vee\rangle
 >0>\langle\xi',\alpha^\vee\rangle
\end{equation}
or
\begin{equation}
\label{eq:comb-case-negative}
 \langle\xi',\alpha^\vee\rangle
 >0>\langle\xi,\alpha^\vee\rangle.
\end{equation}
In either case, $\xi$ and $-\xi'$ have nonzero evaluations of the same sign on $\alpha^\vee$, which proves \eqref{eq:critical_condition} for every $\alpha\in\Phi(P)$.

For $\alpha\in\Phi_0^+(L_P)$, the sign is already determined by the coefficient weight. The central weights $\mathsf w,\mathsf w'$ and the parabolic half-sum $\rho_P$ vanish on $\alpha^\vee$, while $-w_{L_P}(\mu)$ is dominant for $L_P$ and $\rho_{L_P}$ is strictly positive on every positive coroot of $L_P$. Thus \eqref{eq:comb-xi} reads
\begin{equation}
\label{eq:comb-Levi-sign}
 \langle\xi,\alpha^\vee\rangle
 =\langle-\xi',\alpha^\vee\rangle
 =\langle-w_{L_P}(\mu)+\rho_{L_P},\alpha^\vee\rangle>0.
\end{equation}
We have therefore proved \eqref{eq:critical_condition} for every positive root, and the assertion for negative roots follows by negating both evaluations.

These signs show that $\xi$ and $-\xi'$ avoid every root hyperplane and lie in the same open Weyl chamber. The Weyl group acts simply transitively on the chambers of the derived root space and fixes the central directions \cite[Chapter~II, Section~6]{knapp3}, so there is a unique $w_\xi\in W_0$ taking the dominant chamber to this common chamber. By construction, both $w_\xi^{-1}(\xi)$ and $w_\xi^{-1}(-\xi')$ are strictly dominant. To verify that $w_\xi$ is also a minimal left-coset representative for the Weyl group of the Levi subgroup, take a positive root $\alpha$ of $L_P$ and use \eqref{eq:comb-Levi-sign}:
\[
 0<\langle\xi,\alpha^\vee\rangle
   =\langle w_\xi^{-1}(\xi),
                 (w_\xi^{-1}(\alpha))^\vee\rangle.
\]
A strictly dominant weight has negative evaluation on every negative coroot, so the positivity on the right forces $w_\xi^{-1}(\alpha)>0$. This is precisely the condition in \eqref{eq:weyl_P} for $w_\xi\in W_{\rm res}^P$.

We can now apply \Cref{lem:dual_weyl}, which associates to $w_\xi$ the element $w_\xi^\vee=w_{L_P}w_\xi w_G\in W_{\rm res}^P$ of complementary length. Since $W_{\rm abs}^P$ is the product of two copies of $W_{\rm res}^P$, the pair $\ulw_\xi=(w_\xi^\vee,w_\xi)$ belongs to $W_{\rm abs}^P$, and its absolute length is the sum of the two restricted lengths:
\[
 \ell_{\rm abs}(\ulw_\xi)
   =\ell_{\rm res}(w_\xi^\vee)+\ell_{\rm res}(w_\xi)
   =d_P=\tfrac12\dim_{\mathbb R}N_P.
\]

It remains to check that the coefficient weight in \eqref{eq:comb-conclusion} is dominant and algebraic. To express it in terms of the strictly dominant weights just constructed, first use \eqref{eq:comb-xi-lattice} to rewrite $\underline\mu+\underline\rho_G$ as
\[
 \underline\mu+\underline\rho_G
    =\bigl(-w_{L_P}(\xi),-\xi'\bigr).
\]
Since $(w_\xi^\vee)^{-1}=w_Gw_\xi^{-1}w_{L_P}$, applying $\ulw_\xi^{-1}$ to this pair and subtracting $\underline\rho_G$ expresses the coefficient as
\begin{equation}
\label{eq:dominant_wt}
 \underline\lambda=
 \bigl(-w_G(w_\xi^{-1}(\xi))-\rho_G,
             w_\xi^{-1}(-\xi')-\rho_G\bigr).
\end{equation}
The weights $w_\xi^{-1}(\xi)$ and $w_\xi^{-1}(-\xi')$ are integral on coroots by \eqref{eq:comb-xi-lattice}, so strict dominance means that each of their simple-coroot evaluations is an integer at least $1$. The map $-w_G$ permutes the simple coroots, and therefore the first component before subtracting $\rho_G$ has the same property. Since $\langle\rho_G,\alpha^\vee\rangle=1$ for every simple root, subtracting $\rho_G$ leaves nonnegative simple-coroot evaluations in both components of \eqref{eq:dominant_wt}. Hence $\underline\lambda$ is dominant.

For algebraicity, we must also verify that both components belong to the character lattice $\Lambda$ for the given isogeny form of $G_0$, or equivalently that $\underline\lambda\in\Lambda\oplus\Lambda$. Using the definition of the coefficient, write
\[
 \underline\lambda
   =\ulw_\xi^{-1}(\underline\mu)
      +\bigl(\ulw_\xi^{-1}(\underline\rho_G)
                         -\underline\rho_G\bigr).
\]
The first summand belongs to $\Lambda\oplus\Lambda$, because $\underline\mu$ is the full algebraic coefficient weight and the Weyl group preserves the character lattice. For the second summand, each simple reflection changes $\rho_G$ by a root, so $w(\rho_G)-\rho_G$ belongs to the root lattice for every $w\in W_0$. The second summand therefore lies in the absolute root lattice, which is contained in $\Lambda\oplus\Lambda$, even when $\rho_G$ itself is not a character. This proves that $\underline\lambda$ is algebraic.

Finally, both components of $\underline\mu+\underline\rho_G$ are regular by \eqref{eq:comb-xi-lattice} and \eqref{eq:critical_condition}. Any absolute Weyl element producing a dominant coefficient must, under its inverse, send both of these components to the strictly dominant chamber, since adding $\underline\rho_G$ to a dominant weight makes it strictly dominant on coroots. There is a unique such chamber representative for each regular component, so this property also characterizes $\ulw_\xi$ uniquely.
\end{proof}

We now identify the Kostant representative contributing to the bottom degree cohomology of the same representation $I(\tau,s_0)\simeq I(\chi_{\ulm, s_0})$, now viewed as a principal series.

The Kostant representative $\ulw_{\xi}$ in \eqref{eq:comb-conclusion} indexes the coefficient representation $V_{\underline\mu}^{L_P}$ in degree $\ell_{\rm abs}(\ulw_{\xi})=\#\Phi(P)$ of the cohomology of the complexified Lie algebra of $N_P$. The remaining cohomology for $L_P$ begins in degree $q_0(L_P)$. By \Cref{prop:principal-series}, realizing $\tau$ as a principal series induced from $B\cap L_P$ selects the Kostant representative $(w_{L_P},1)$. Consequently, the Kostant representative contributing to the cohomology of $I(\tau,s_0)$ when it is realized as a principal series induced from $B$ is
\begin{equation}
\label{eq:comb-Borel-element}
 \underline w_{\chi_{\ulm, s_0}}=(w_{L_P},1)\ulw_{\xi}
                  =(w_\xi w_G,w_\xi)\in W_{\rm abs}(G).
\end{equation}
Minimality of $\ulw_{\xi}\in W_{\rm abs}^P$ makes this product length-additive. Using \eqref{eq:comb-conclusion} and \eqref{eq:complex-bottom-degree}, we obtain
\[
 \ell_{\rm abs}(\underline w_{\chi_{\ulm, s_0}})
 =q_0(L_P)+\ell_{\rm abs}(\ulw_{\xi})
 =q_0(L_P)+\#\Phi(P)=q_0(G).
\]

For $\underline w\in W_{\rm abs}(G)$, we write the twisted Weyl action as
\[
 \underline w\cdot\underline\lambda
   :=\underline w(\underline\lambda+\underline\rho_G)-\underline\rho_G.
\]
Equations~\eqref{eq:comb-evaluation-parameter} and \eqref{eq:dominant_wt} give the first identity below. This definition and the formula for $d(\delta_B^{1/2})$ in \eqref{eq:parabolic-modulus-differential} give the equivalent second identity:
\[d\chi_{\ulm, s_0}
   =-\underline w_{\chi_{\ulm, s_0}}
                       (\underline\lambda+\underline\rho_G), \quad d(\delta_B^{1/2}\chi_{\ulm, s_0})
   +\underline w_{\chi_{\ulm, s_0}}\cdot\underline\lambda=0.
\]
In Borel--Wallach's cohomology formula \cite[III.3.3(ii)]{borel-wallach}, the inducing character $\delta_B^{1/2}\chi_{\ulm, s_0}$ is tensored with the one-dimensional summand of nilradical cohomology having weight $\underline w_{\chi_{\ulm, s_0}}\cdot\underline\lambda$. The last identity says that these torus characters cancel. Regularity of $\underline\lambda+\underline\rho_G$ makes this the unique summand with that property. Their tensor product contributes in degree zero to the cohomology of the torus, giving a class in degree $q_0(G)$ for $G$. This is the summand used in \Cref{subsec:main-first-Delorme} to normalize the generator in bottom-degree cohomology.

\begin{prop}
\label{prop:associated}
Let $Q$ be the standard associate parabolic of $P$, let $w_P=w_Gw_{L_P}$, and set $\ulw_{\xi}'=\jmath(w_P)\ulw_{\xi}$, where $\jmath(w)=(w,w)$. Then $\ulw_{\xi}'\in W_{\rm abs}^Q$ and $\ell_{\rm abs}(\ulw_{\xi}')=d_P$.  The weight
\begin{equation}
\label{eq:comb-associate-coefficient}
 \tilde{\ulm}:=\ulw_{\xi}'\cdot\underline\lambda=\jmath(w_P)(\underline\mu+\underline\rho_G)
                        -\underline\rho_G=\jmath(w_P)(\underline\mu)-2\underline\rho_Q
\end{equation}
is a dominant integral highest weight for $L_Q$ satisfying the purity condition.
\end{prop}
\begin{proof}
By \Cref{cor:associate_parabolic}, $w_P$ carries the simple roots in $L_P$ to those of $L_Q$. \Cref{lem:associate_weyl} then implies that $w\mapsto w_Pw$ from $W_{\rm res}^P$ to $W_{\rm res}^Q$ with
\[
 \ell(w_Pw)=\ell(w_Gw_{L_P}w)=d_P-\ell(w)
                  \qquad(w\in W_{\rm res}^P).
\]
Apply this to both components of $\ulw_{\xi}$ and use \eqref{eq:comb-conclusion} to obtain the degree. Since
\[w_P\rho_{L_P}=\rho_{L_Q}, \ w_P\rho_P=-\rho_Q,\]
 one has $w_P\rho_G-\rho_G=-2\rho_Q$, which proves \eqref{eq:comb-associate-coefficient}.

The two components of $\tilde{\ulm}$ are $L_Q$-dominant: $w_P$ carries $L_P$-dominance to $L_Q$-dominance, and $\rho_Q$ vanishes on all $L_Q$-coroots. Their algebraicity follows from the formula for the twisted action. For purity put $w_{L_Q}=w_Pw_{L_P}w_P^{-1}$.  The sum of the first component and the image of the second under $w_{L_Q}$ equals
\[
 w_P(\mathsf w+\mathsf w')-4\rho_Q,
\]
which vanishes on all coroots in the Levi $L_Q$. Thus the derived part satisfies the purity relation of \Cref{lem:purity}, and the remainder is central.
\end{proof}

\subsection{Irreducibility of the induced representations}
\label{sec:irreducibility-principal-series}
Retain the hypotheses of \Cref{lem:combinatorial_lemma} and the normalized inducing parameter from \eqref{eq:comb-evaluation-parameter}:
\begin{equation}
\label{eq:irr-parameter}
 \dd\chi_{\ulm, s_0}=\underline\nu_{s_0}=(\xi, \xi')
                                  \in\mathfrak t^*.
\end{equation}
To identify the associate target, recall that the definition of $Q$ shows that $\alpha_Q=-w_G(\alpha_P)$ and $\rho_Q=-w_G(\rho_P)$. Weyl invariance of the root--coroot pairing then indicates that 
\[
 \langle\rho_Q,\alpha_Q^\vee\rangle
      =\langle\rho_P,\alpha_P^\vee\rangle=-s_0.
\]
Together with $w_P(\rho_P)=-\rho_Q$, we have the corresponding identities for the normalized weights and modulus characters of \eqref{eq:parabolic-parameter-normalization} and \eqref{eq:parabolic-modulus}:
\[
 w_P(\gamma_P)=-\gamma_Q,\qquad
 {}^{w_P}\delta_P=\delta_Q^{-1}.
\]
Here $\gamma_Q$ is defined by \eqref{eq:parabolic-parameter-normalization} with $P$ and $L_P$ replaced by $Q$ and $L_Q$, respectively. The twist in \eqref{eq:parabolic-modulus-family} is therefore transported to $\delta_Q^{s/(2s_0)}$, whose value at $s=s_0$ is $\delta_Q^{1/2}$. We denote the fixed target by $\widetilde\Pi$. The source $\Pi$ from \Cref{subsec:parabolic_induction} and this target have realizations as normalized principal series induced from $B$
\begin{equation}
\label{eq:irr-endpoints}
 \begin{aligned}
 \Pi
   &=\Ind_P^G(\tau,\delta_P^{-1/2})\simeq I(\chi_{\ulm, s_0}),\\
 \widetilde\Pi
   &:=\Ind_Q^G({}^{w_P}\tau,\delta_Q^{1/2})
       \simeq I({}^{w_P}\chi_{\ulm, s_0}).
 \end{aligned}
\end{equation}

\begin{prop}[Irreducibility of the principal series]
  \label{prop:irreducibility_principal_series}
Under the hypotheses of \Cref{lem:combinatorial_lemma}, $I({}^{w}\chi_{\ulm, s_0})$ is irreducible for every $w\in W_{\rm res}(G)$. In particular, both representations $\Pi=I(\chi_{\ulm,s_0})$  and $\widetilde{\Pi}=I({}^{w_P}\chi_{\ulm, s_0})$ are irreducible.
\end{prop}
\begin{proof}
To prove irreducibility, we apply the necessary reducibility condition in \cite[Theorem~6.19, p.~292]{speh-vogan} directly to these normalized principal series. Our connected complex reductive group $G$ satisfies the group hypotheses of \cite[Definition~2.1]{speh-vogan}. For the source $I(\chi_{\ulm,s_0})$, put
\[
 \uln_c=\operatorname{pr}_c\uln_{s_0},\qquad
 \uln_A=\operatorname{pr}_A\uln_{s_0},
\]
where $\uln_{s_0}=(\xi,\xi')$ is the parameter in \eqref{eq:irr-parameter}, and the compact and split projections are defined in \eqref{eq:Cartan-zero-extended-projections}.

Apply Speh--Vogan with their parabolic $P=MAN$ equal to our $B=T_cA^\circ N_B$. Their Cartan subgroup $H$ and maximal compact subgroup $K$ are our $T$ and $C$. Since $M=T_c$ is a compact torus, their discrete-series datum is the character $\chi_{\ulm,s_0}|_{T_c}$. Their character of $A$ is $\chi_{\ulm,s_0}|_{A^\circ}$. At the level of differentials, extended by zero to $\mathfrak t$, the parameter correspondence is
\[
 \begin{array}{c|c}
  \text{Speh--Vogan}&\text{our notation}\\ \hline
  \lambda&\uln_c\\
  \nu&\uln_A\\
  \gamma=(\lambda,\nu)&\uln_c+\uln_A=\uln_{s_0}.
 \end{array}
\]
A character of the compact torus $T_c$ is unitary and square-integrable. To identify the compact and split parameters, consider the principal series $I(\chi)$ and put
\[
 \dd\chi=\ule=(\eta,\eta')\in\mathfrak t^*.
\]
Write $\ule_c$ and $\ule_A$ for the complex-linear extensions of $\ule|_{\mathfrak t_{c,0}}$ and $\ule|_{\mathfrak a_0}$ by zero on the other summand. Formula~\eqref{eq:Cartan-zero-extended-projections} implies that
\begin{equation}\label{eq:rk1-cartan-projections}
 \begin{aligned}
 \ule_c&=\frac{\ule+\theta\ule}{2}
       =\left(\frac{\eta-\eta'}2,-\frac{\eta-\eta'}2\right),\\
 \ule_A&=\frac{\ule-\theta\ule}{2}
       =\left(\frac{\eta+\eta'}2,\frac{\eta+\eta'}2\right),
       \qquad \ule=\ule_c+\ule_A.
 \end{aligned}
\end{equation}
Here $\theta(\eta,\eta')=(-\eta',-\eta)$ by \eqref{eq:theta-paired-weights}. If $\mathcal J$ denotes the original complex structure on the real Lie algebra of the complex torus, \eqref{eq:paired-character-differential} reads
\[
 \ule(H+\mathcal JH')
   =(\eta+\eta')(H)+i(\eta-\eta')(H'),\qquad H,H'\in\mathfrak a_0.
\]
Specializing to $\ule=\uln_{s_0}=(\xi,\xi')$ from \eqref{eq:irr-parameter}, we have
\begin{equation}\label{eq:rk1-cartan-evaluation}
 \uln_c=\left(\frac{\xi-\xi' }{2},-\frac{\xi-\xi' }{2}\right),
 \qquad
 \uln_A=\left(\frac{\xi+\xi'  }{2},\frac{\xi+\xi'  }{2}\right).
\end{equation}

The factors $1/2$ in \eqref{eq:rk1-cartan-evaluation} arise when the compact and split restrictions are represented in doubled coordinates by extending them by zero on the complementary summand. These extensions preserve the values of the original restrictions.

Just to avoid confusion, one can harmlessly identify our symmetric pairing $(\, ,\,)_{\rm abs}$ from \eqref{eq:absolute-symmetric-form} with the pairing in Speh--Vogan \cite[p.~231]{speh-vogan}, where they used the notation $\langle \, , \, \rangle$. The relation between our notations $(\, ,\,)_{abs}$ and $\langle \, , \, \rangle$ is explained in formula~\eqref{eq:root-coroot-pairing-comparison}.

Now, let's first take $w=\mathrm{Id}\in W_{res}(G)$. By \eqref{eq:absolute_roots}, every absolute root is $\alpha^{(1)}=(\alpha,0)$ or $\alpha^{(2)}=(0,\alpha)$, and \eqref{eq:theta-paired-weights} implies that
\[
 \theta\alpha^{(1)}=-\alpha^{(2)},\qquad
 \theta\alpha^{(2)}=-\alpha^{(1)}.
\]
Thus all absolute roots are complex. By \cite[Theorem~6.19]{speh-vogan}, reducibility would therefore require an absolute root $\beta$ with
\begin{equation}
\label{eq:irr-SV-signs}
 \langle\uln_{s_0},\beta^\vee\rangle>0,\qquad
 \langle\uln_{s_0},(\theta\beta)^\vee\rangle<0.
\end{equation}
Here \eqref{eq:root-coroot-pairing-comparison} converts their symmetric-pairing inequalities to coroot pairings.

For either $\beta=\alpha^{(1)}$ or $\beta=\alpha^{(2)}$, the parameter \eqref{eq:irr-parameter} satisfy
\[
 \langle\uln_{s_0},\beta^\vee\rangle
 \langle\uln_{s_0},(\theta\beta)^\vee\rangle
 =\langle\xi,\alpha^\vee\rangle
  \langle-\xi',\alpha^\vee\rangle>0
\]
by \eqref{eq:critical_condition}. This excludes \eqref{eq:irr-SV-signs}, so $I(\chi_{\ulm,s_0})$ is irreducible.

For arbitrary $w\in W_{\rm res}(G)$, the parameter of $I({}^{w}\chi_{\ulm,s_0})$ is $(w(\xi),w(\xi'))$ by \eqref{eq:absolute-weyl-diagonal}. Weyl invariance implies that
\[
 \langle w(\xi),\alpha^\vee\rangle
 \langle-w(\xi'),\alpha^\vee\rangle
 =\langle\xi,(w^{-1}(\alpha))^\vee\rangle
  \langle-\xi',(w^{-1}(\alpha))^\vee\rangle>0.
\]
The same criterion therefore proves irreducibility for every $w$. When $w=w_P$, the proof gives the irreducibility of $\Pi$ and $\widetilde{\Pi}$. 
\end{proof}

\begin{rem}
The two adjacent critical arguments enter the strict bounds in different ways, as shown in \eqref{eq:comb-critical-bounds}. Irreducibility follows from the resulting signs in \eqref{eq:critical_condition}. Regularity and nonvanishing at $s_0$ of a particular meromorphic standard intertwiner remain separate analytic assertions.  Once these are established, its nonzero value is a map between irreducible modules and hence an isomorphism.
\end{rem}
\section{\bf Main theorem}
\label{sec:mainthm}

\subsection{The arithmetic theorem}
\label{subsec:mt3-main}
We now determine the standard intertwining operator on prescribed rational lines in bottom-degree cohomology. Retain the hypotheses and inducing data of \Cref{sec:irreducibility-principal-series}. We use $q_0(G)$, as in \eqref{eq:complex-bottom-degree}, for the bottom cohomological degree for $G$. Recall the family of characters $\chi_{\ulm,s}$ of $T$ defined in \Cref{subsec:parabolic_induction}, with differential \eqref{eq:parabolic-twist-differential}. These characters induce the normalized principal series $I(\chi_{\ulm,s})$. At the point $s_0$ of \Cref{def:point_of_evaluation}, we use the fixed characters
\[
 \chi=\chi_{\ulm}\cdot(\delta_P^{-1/2}|_T),\qquad
 \widetilde\chi={}^{w_P}\chi.
\]
Thus $\chi$ differs from $\chi_{\ulm}$ by the twist $\delta_P^{-1/2}|_T$ required at $s_0$. Write $\chi_s=\chi_{\ulm,s}$ for the varying family, so $\chi=\chi_{s_0}$. The source and target representations in \eqref{eq:irr-endpoints} are
\begin{equation}
\label{eq:mt3-endpoints}
 \begin{aligned}
 \Pi
   &=\Ind_P^G(\tau,\delta_P^{-1/2})
     ={}^a\Ind_P^G(\tau)
       \simeq I(\chi),\\
 \widetilde\Pi
   &=\Ind_Q^G({}^{w_P}\tau,\delta_Q^{1/2})={}^a\Ind_Q^G({}^{w_P}\tau\otimes\delta_Q)
       \simeq I(\widetilde\chi).
\end{aligned}
\end{equation}
In the representation $\widetilde\Pi$, the inducing twist $\delta_Q^{1/2}$ combines with the factor $\delta_Q^{1/2}$ in normalized induction \eqref{eq:normalized-parabolic-induction} to give $\delta_Q$ in unnormalized induction. We use the root coordinates, Weyl representatives, and Haar measures of \eqref{eq:rk1-root-map}--\eqref{eq:rk1-measure}. For the Weyl element $w_P=w_Gw_{L_P}$ defined after \Cref{cor:longest_Wp}, define
\[
 \operatorname{T}_{\mathrm{st}}(s)
   :=\operatorname{T}_{w_P}(\chi_s):
       I(\chi_s)\longrightarrow I({}^{w_P}\chi_s).
\]
Under \eqref{eq:parabolic-induction-in-stages}, these are the parabolic representations
\[
 \Ind_P^G(\tau,\delta_P^{-s/(2s_0)})
 \quad\text{and}\quad
 \Ind_Q^G({}^{w_P}\tau,\delta_Q^{s/(2s_0)}).
\]
Here the transported Levi realization is \({}^{w_P}\tau\simeq I_{L_Q}({}^{w_P}\chi_{\ulm})\). Since $w_P$ carries the positive Levi roots of $P$ to those of $Q$ and $\Phi(P)$ to $-\Phi(Q)$, its integration subgroup is $N_{w_P}=N_Q$. For $F\in \Pi=I_{P}^G(\tau, \delta_P^{-1/2})$, the function $f(g)=F(g)(1)\in I(\chi)$, the integral \eqref{eq:inter_op} therefore becomes
\begin{equation}
\label{eqn:standard_inter_op}
 (\operatorname{T}_{\mathrm{st}}(s)f)(g)
   =\left(\int_{N_Q}F(w_P^{-1}ug)\,du\right)(1).
\end{equation}
Under the evaluation maps of \eqref{eq:parabolic-induction-in-stages}, this is the standard intertwining integral between the representations induced from $P$ and $Q$. The equality holds initially in the region of convergence and then by meromorphic continuation.

For each principal series $I\left(\chi^{(a)}\right)$ appearing in the factorization of $T_{st}(s_0)$, including $\Pi$ and $\widetilde{\Pi}$, we choose the generator $[J_{\chi^{(a)}}]$ of the one-dimensional space $H^{q_0(G)}(\g,K;I\left(\chi^{(a)}\right)\otimes V_{\ulb})$ by \eqref{eq:main-first-inverse}, independently of the intertwining operators. The construction is given in \Cref{subsec:main-first-Delorme}. We can now state the main theorem.

\begin{thm}
\label{thm:mainthm}
Assume that $s_0$ is hyper-critical as in \Cref{lem:combinatorial_lemma}. Then the family $\operatorname{T}_{\mathrm{st}}(s)$ is holomorphic at $s=s_0$, 
and its value
\[
 \operatorname{T}_{\mathrm{st}}(s_0):\Pi\xrightarrow{\ \sim\ }\widetilde\Pi
\]
is an isomorphism of smooth irreducible representations. It induces an isomorphism on relative Lie algebra cohomology with coefficients in 
$V_{\ulb}$ in every degree.
The two cohomology groups in bottom degree $q_0(G)$ are one-dimensional. On the generators $[J_{\chi}]$ and $[J_{\tilde{\chi}}]$ defined 
by \eqref{eq:main-first-inverse}, one has
\begin{equation}
\label{eq:main-arithmetic-scalar}
 \operatorname{T}_{\mathrm{st}}^\bullet(s_0)([J_{\chi}])
   =a\left(\prod_j
       \frac{L_{\mathbb C}(js_0,\tau,r_j)}
            {L_{\mathbb C}(1+js_0,\tau,r_j)}\right)
       [J_{\tilde{\chi}}]
\end{equation}
for some $a\in\mathbb Q^\times$. 
\end{thm}

\subsection{Strategy of the proof}
We first fix the cohomology generators independently of the operator. Evaluating functions in the induced representation at the identity and then taking the summand in nilradical cohomology specified by Kostant's theorem gives the isomorphism \eqref{eq:main_kostant} from bottom-degree cohomology to a tensor product of degree-zero relative cohomology groups for the torus factors. The vectors $1_{\chi^j}\otimes1_{(\chi^j)^{-1}}$ determine a generator for each intermediate principal series.

\Cref{subsec:main-first-rank-one} constructs the rational classes for $H=\mathrm{SL}_2(\mathbb C)$ and represents them by cochains built from the test vectors of \eqref{eq:rk1-test-vectors}. \Cref{subsec:main-first-regularity} uses the scalar calculation \eqref{eq:rk1-scalar} and irreducibility to prove holomorphy, first for the operator on $H$ and then for each operator on $G$ associated with a simple reflection.

In \Cref{subsec:main-first-transitivity}, we use these cochains to compute the intertwining operator on the rational classes for $H$, and compare those classes with the generators $[J_{\chi^{(a)}}]$ for $(\mathfrak g,K)$-cohomology. Computing cohomology through the parabolic subgroup $P_\alpha$ identifies the coefficient representation and degree that occur in the calculation for $L_\alpha$. The chosen root vectors and highest-weight vectors make this comparison rational. Compatibility with the intertwining integral gives the commutative diagram, whose identities yield \eqref{eq:main-arithmetic-scalar}. This follows the rank-one method of \cite[Sections~4.1.8 and 4.2.3]{raghuram-imaginary}.

\subsubsection{Intermediate characters and the contributing Kostant elements}
\label{subsec:main-first-intermediates}
We first check that the same coefficient system and bottom degree apply throughout the factorization. Computing the cohomology of $\Pi$ and $\widetilde\Pi$ through the parabolic subgroups $P$ and $Q$ uses the Kostant elements $\ulw_{\xi}$ and $\ulw_{\xi}'$ from \eqref{eq:comb-kostant-representative} and \Cref{prop:associated}. Put $\mathfrak n_P=\operatorname{Lie}_{\R}(N_P)\otimes_{\R}\C$. Kostant's decomposition places the coefficient representation $V^{L_P}_{\ulw_{\xi}\cdot\underline\lambda}$ of $L_P$ in $H^{\ell_{\rm abs}(\ulw_{\xi})}(\mathfrak n_P,V_{\ulb})$. The equal lengths of $\ulw_{\xi}$ and $\ulw_{\xi}'$, given by \eqref{eq:comb-conclusion} and \Cref{prop:associated}, combine with the bottom-degree formula \eqref{eq:complex-bottom-degree} and $q_0(L_P)=q_0(L_Q)$ to give
\begin{equation}\label{eq:main-parabolic-bottom-degrees}
 q_0(L_P)+\ell_{\rm abs}(\ulw_{\xi})=q_0(G)=q_0(L_Q)+\ell_{\rm abs}(\ulw_{\xi}').
\end{equation}
For $\widetilde\Pi={}^a\Ind_Q^G({}^{w_P}\tau\otimes\delta_Q)$ in \eqref{eq:mt3-endpoints}, the representation of $L_Q$ from which we induce is \({}^{w_P}\tau\otimes\delta_Q\). Its cohomology uses the coefficient representation $V^{L_Q}_{\tilde{\ulm}}$, where $\tilde{\ulm}=\ulw_{\xi}'\cdot\underline\lambda$ as in \Cref{prop:associated}.

Using the identification $\Pi\simeq I(\chi)$ in \eqref{eq:mt3-endpoints}, we can instead compute the source cohomology by induction from the Borel subgroup $B$. Here $\tau\simeq I_{L_P}(\chi_{\ulm})$ and $\chi=\chi_{\ulm}\cdot(\delta_P^{-1/2}|_T)$. The element $\ulw_\xi$ indexes the coefficient representation in the cohomology of $\mathfrak n_P$, while $(w_{L_P},1)$ indexes the subsequent summand for the Borel subgroup $B\cap L_P$ of $L_P$. Thus \eqref{eq:comb-Borel-element} implies that $\underline w_{\chi}=(w_{L_P},1)\ulw_{\xi}=(w_\xi w_G,w_\xi)$. Minimality of $\ulw_{\xi}$ makes this factorization length-additive:
\[
 \ell_{\rm abs}(\underline w_{\chi})
   =\ell_{\rm res}(w_{L_P})+\ell_{\rm abs}(\ulw_{\xi})
   =q_0(L_P)+\ell_{\rm abs}(\ulw_{\xi})=q_0(G).
\]

To follow the operator one simple reflection at a time, choose a reduced expression and its partial products
\[
 w_P=s_{\alpha_{d_P}}\cdots s_{\alpha_1},\qquad
 w_{(a)}=s_{\alpha_a}\cdots s_{\alpha_1},\quad w_{(0)}=1,
 \qquad \chi^{(a)}={}^{w_{(a)}}\chi.
\]
Thus $\chi^{(0)}=\chi$ and $\chi^{(d_P)}=\widetilde\chi$. Applying the Weyl action to the differential in \eqref{eq:comb-evaluation-parameter} gives, for $0\leq a\leq d_P$,
\[
 d\chi^{(a)}
   =\bigl(w_{(a)}(\xi),w_{(a)}(\xi')\bigr).
\]
Every $I\left(\chi^{(a)}\right)$ is irreducible by \Cref{prop:irreducibility_principal_series}. For an intermediate character $\chi^{(a)}$, put
\begin{equation}
\label{eq:main-first-absolute-elements}
 \underline w_{\chi^{(a)}}
   =\jmath(w_{(a)})\underline w_\chi
   =(w_{(a)}w_\xi w_G,w_{(a)}w_\xi).
\end{equation}
Applying $\jmath(w_{(a)})$ to the identity following \eqref{eq:comb-Borel-element} gives $d\chi^{(a)}=-\underline w_{\chi^{(a)}}(\underline\lambda+\underline\rho_G)$. The length identity for multiplication by the longest element, applied to $w_{(a)}w_\xi$, shows that the lengths of the two components of $\underline w_{\chi^{(a)}}$ sum to $\ell(w_G)=q_0(G)$. Together these facts give
\begin{equation}
\label{eq:main-first-matching}
 \ell_{\rm abs}(\underline w_{\chi^{(a)}})=q_0(G),\qquad
 d(\delta_B^{1/2}\chi^{(a)})+\underline w_{\chi^{(a)}}\cdot\underline\lambda=0.
\end{equation}
Thus the same coefficient system $V_{\ulb}$ and bottom degree occur at every step.

For the rank-one calculation, fix $a$ and a simple root $\alpha$, and put
\begin{equation}
\label{eq:main-first-root-exponents}
 r_\alpha=|\langle w_{(a)}(\xi),\alpha^\vee\rangle|,
 \qquad
 t_\alpha=|\langle w_{(a)}(-\xi'),\alpha^\vee\rangle|.
\end{equation}
The lattice calculation \eqref{eq:comb-xi-lattice} and the sign condition \eqref{eq:critical_condition} make these positive integers. Consequently
\[
 \chi^{(a)}\circ\alpha^\vee=z^{r_\alpha}\bar z^{-t_\alpha}
 \quad\text{or}\quad z^{-r_\alpha}\bar z^{t_\alpha}.
\]
These are the two inducing characters treated in \eqref{eq:rk1-cohom-data}.

\subsubsection{Rational structures in bottom-degree cohomology}
\label{subsec:main-first-Delorme}
We now choose root vectors and a highest-weight vector in $V_{\ulb}$, and use evaluation of functions in $I(\chi)$ at the identity to define a generator of $H^{q_0(G)}(\mathfrak g,K;I(\chi)\otimes V_{\ulb})$ and its rational structure.

Let $\mathfrak n$ be the complexification of the real Lie algebra of the unipotent radical of the fixed Borel subgroup $B$. Thus
\[
 \mathfrak n=\bigoplus_{\beta\in\Phi_{\mathrm{abs}}^+(G)}
                         \mathfrak g_\beta.
\]
Each absolute root space $\mathfrak g_\beta$ is a complex line.

For $\alpha\in\Phi_0$ and the complex root parametrization $x_\alpha$ fixed in \Cref{subsec:root-homomorphisms}, we have
\[
 X_\alpha=(dx_\alpha)_0(1)
      \in\operatorname{Lie}_{\C}(\mathrm G_0),
\]
where $1$ is the standard tangent vector at $0$ in $\mathbb G_a(\C)=\C$. After restriction of scalars and base change to $\C$, this map is, in the matrix coordinates of \eqref{eq:embedding},
\[
 (u_1,u_2)\longmapsto
   \bigl(x_\alpha(u_1),\overline{x_\alpha(\bar u_2)}\bigr).
\]
Its complex differential at $(0,0)$ in the two coordinate directions then gives us the absolute root vectors
\[
 X_{\alpha^{(1)}}=(X_\alpha,0),\qquad
 X_{\alpha^{(2)}}=(0,\overline{X_\alpha}).
\]
Their Chevalley bracket constants are integers \cite[Theorem~23.71]{milne}. Fix a total order on $\Phi_{\mathrm{abs}}^+(G)$ and write $\left\{X_\beta^*\right\}$ for the dual basis of $\mathfrak n^*$. We use the induced order on each root set for all exterior products below, thereby fixing their signs.

To define a rational form of the coefficient representation, choose a highest-weight vector $v_{\underline\lambda}\in V_{\ulb}$ and put
\[
 \mathfrak n_{\mathbb Q}
   =\bigoplus_{\beta\in\Phi_{\mathrm{abs}}^+(G)}\mathbb Q X_\beta.
\]
Let $\mathcal{V}_{\ulb,\mathbb Q}$ be the rational span of
\[
 \left(\prod_{\beta\in\Phi_{\mathrm{abs}}^+(G)}
          X_{-\beta}^{k_\beta}\right)v_{\underline\lambda},
 \qquad k_\beta\in\mathbb Z_{\geq0},
\]
where the product follows the fixed root order and the powers mean repeated action through the coefficient representation. The action preserves $\mathcal{V}_{\ulb,\mathbb Q}$, so both the bracket on $\mathfrak n_{\mathbb Q}$ and its coefficient action are defined over $\mathbb Q$. It follows that the cochain complex
\begin{equation}\label{eq:main-rational-nilradical-complex}
 C^q(\mathfrak n_{\mathbb Q},\mathcal{V}_{\ulb,\mathbb Q})
   =\operatorname{Hom}_{\mathbb Q}
      \bigl(\bigwedge^q\mathfrak n_{\mathbb Q},\mathcal{V}_{\ulb,\mathbb Q}\bigr)
\end{equation}
has a differential over $\mathbb Q$, and its complexification is $C^\bullet(\mathfrak n,V_{\ulb})$. The pinned Weyl representatives act by rational linear combinations of the chosen root vectors and preserve $\mathcal{V}_{\ulb,\mathbb Q}$. Indeed, they are products of root subgroup elements with parameters $\pm1$, acting by finite exponential sums of rational nilpotent operators. These choices also determine the vectors used in the nilradical cohomology summands and in the subsequent calculations for Levi subgroups.

Thus the rational structures below are defined using the finite-dimensional complex in \eqref{eq:main-rational-nilradical-complex}. No rational form of the smooth principal series is needed.

For the fixed initial character $\chi$ in \Cref{thm:mainthm}, put
\[
 \mathscr H_G(\chi)=
 H^{q_0(G)}(\mathfrak g,K;I(\chi)\otimes V_{\ulb}),\qquad
 K_B=K\cap B.
\]
Equation \eqref{eq:main-first-matching} shows that the central characters of $I(\chi)$ and $V_{\ulb}$ are inverse. Hence the center acts trivially on their tensor product, and we may quotient out this center, which is contained in $K$, when applying the cohomology formula. Since $\mathfrak b/(\mathfrak b\cap\mathfrak k)\simeq\mathfrak g/\mathfrak k$, evaluation at the identity takes the following form on cochains:
\begin{equation}
\label{eq:main-first-evaluation}
 \begin{split}
 r_B:C^j(\mathfrak g,K;I(\chi)\otimes V_{\ulb})
  &\xrightarrow{\ \sim\ }
 C^j(\mathfrak b,K_B;
       \mathbb C_{\delta_B^{1/2}\chi}\otimes V_{\ulb}),\\
 (r_BJ)(Y_1,\ldots,Y_j)
  &=(\operatorname{ev}_1\otimes1)
                  J(\overline Y_1,\ldots,\overline Y_j).
 \end{split}
\end{equation}
Here $\overline Y_i$ denotes the image of $Y_i\in\mathfrak b$ in $\mathfrak g/\mathfrak k$. This is the Frobenius--Shapiro construction with the tensor identity, as described in \cite[III.2.5]{borel-wallach} and \cite[Section~4.1.8]{raghuram-imaginary}. Evaluation is $\mathfrak b$-equivariant and commutes with the relative cochain differentials, so $r_B$ is an isomorphism of cochain complexes.

For the summand of $H^{q_0(G)}(\mathfrak n,V_{\ulb})$ indexed by $\underline w_\chi$, put $\Phi_{\underline w}=\{\beta\in\Phi_{\rm abs}^+(G): \underline w^{-1}(\beta)<0\}$ and choose the Kostant cocycle
\begin{equation}
\label{eq:main-first-Kostant}
 h_\chi^{\rm coch}
 =\left(\bigwedge_{\beta\in\Phi_{\underline w_\chi}}X_\beta^*\right)
       \otimes\dot{\underline w}_\chi(v_{\underline\lambda}).
\end{equation}
By Kostant's theorem \cite[III.3.1]{borel-wallach}, $\mathbb C[h_\chi^{\rm coch}]$ is the $(\delta_B^{1/2}\chi)^{-1}$-weight line in $H^{q_0(G)}(\mathfrak n,V_{\ulb})$, with the extremal-vector normalization of \cite[Section~2.7, equation~(2.28)]{raghuram-imaginary}. The cocycle \eqref{eq:main-first-Kostant} lies in the rational complex \eqref{eq:main-rational-nilradical-complex}, and its class fixes the rational form $\mathbb Q[h_\chi^{\rm coch}]$.

Borel--Wallach's cohomology formula \cite[III.3.3(ii)]{borel-wallach}, together with the cancellation of the two torus characters in \eqref{eq:main-first-matching}, gives an isomorphism
\[
 \mathscr H_G(\chi)\xrightarrow{\ \sim\ }
 H^0(\mathfrak t,K_B;
       \mathbb C_{\delta_B^{1/2}\chi}\otimes\mathbb C[h_\chi^{\rm coch}]).
\]
It is induced by $r_B$ in \eqref{eq:main-first-evaluation}, followed by restriction to $\bigwedge^{q_0(G)}\mathfrak n$ and passage to nilradical cohomology \cite[III.3.4]{borel-wallach}. We fix its normalization by identifying
\[
 \mathbb C_{\delta_B^{1/2}}\otimes\mathbb C[h_\chi^{\rm coch}]
      \xrightarrow{\ \sim\ }\mathbb C_{\chi^{-1}},
 \qquad 1\otimes [h_\chi^{\rm coch}]\longmapsto1_{\chi^{-1}}.
\]

Write $\chi=\boxtimes_j\chi^j$ in the fixed coordinates $T\simeq(\mathbb C^\times)^n$, and let $\mathfrak t_j$ be the complexification of the real Lie algebra of the $j$-th factor. For each factor, the degree-zero relative cohomology group
\[
 \mathscr H^0(\chi^j)=
 H^0(\mathfrak t_j,U(1);
       \mathbb C_{\chi^j}\otimes\mathbb C_{(\chi^j)^{-1}})
\]
has the rational basis $1_{\chi^j}\otimes1_{(\chi^j)^{-1}}$. Their tensor product is
\[
 \Omega_\chi=\bigotimes_{j=1}^n
       (1_{\chi^j}\otimes1_{(\chi^j)^{-1}}).
\]
These identifications give
\begin{equation}
\label{eq:main_kostant}
 \mathscr{K}_{\chi}^G:\mathscr H_G(\chi)\xrightarrow{\ \sim\ }
       \bigotimes_j\mathscr H^0(\chi^j),\qquad
 \mathscr{K}_{\chi}^G([J])=c\ \Omega_\chi,
\end{equation}
where, for a closed cochain $J$, the scalar $c\in\mathbb C$ is determined by
\begin{equation} 
  \label{eq:HG_to_H0}
 \left[(r_BJ)|_{\bigwedge^{q_0(G)}\mathfrak n}\right]
    =c\bigl(1_{\delta_B^{1/2}\chi}\otimes [h_\chi^{\rm coch}]\bigr).
\end{equation}

The normalized class specified by \eqref{eq:HG_to_H0} with $c=1$ has the cocycle representative
\begin{equation}\label{eq:main-first-inverse}
 J_\chi=r_B^{-1}
       \bigl(1_{\delta_B^{1/2}\chi}\otimes h_\chi^{\rm coch}\bigr),
\end{equation}
where the cochain on $\mathfrak b$ is extended by zero whenever an argument lies in $\mathfrak t$. Thus $\mathscr{K}_{\chi}^G([J_\chi])=\Omega_\chi$, and we set
\begin{equation}\label{eq:main-prescribed-rational-line}
 \mathscr H_G(\chi)_{\mathbb Q}
   :=(\mathscr{K}_{\chi}^G)^{-1}(\mathbb Q\Omega_\chi)
    =\mathbb Q[J_\chi].
\end{equation}

The same construction applies to any inducing character satisfying the character identity and length equality in \eqref{eq:main-first-matching}. In particular, the root vectors, their order, and the highest-weight vector chosen above determine $\mathscr{K}_{\chi^{(a)}}^G$ and $[J_{\chi^{(a)}}]$ for every $0\leq a\leq d_P$. This includes the characters $\chi^{(0)}=\chi$ and $\chi^{(d_P)}=\widetilde\chi$ that induce $\Pi$ and $\widetilde\Pi$, respectively.

\subsection{Rational cohomology classes for
\texorpdfstring{$\mathrm{SL}_2(\mathbb C)$}{SL2(C)}}
\label{subsec:main-first-rank-one}
Let $H=\mathrm{SL}_2(\mathbb C)$, with upper triangular Borel subgroup $B_H$ and maximal compact subgroup $C_H=\mathrm{SU}(2)$, as in \Cref{sec:rk1-model}. Write $\mathfrak h$ and $\mathfrak c_H$ for their complexified real Lie algebras. For $r,t\in\mathbb Z_{>0}$, put
\begin{equation}\label{eq:rk1-cohom-data}
 \psi_+(z)=z^r\bar z^{-t},\qquad
 \psi_-(z)=z^{-r}\bar z^t=\psi_+^{-1}.
\end{equation}
We use the normalized principal series $I_\pm=I_H(\psi_\pm)$ of \eqref{eq:rk1-induction} and the coefficient representation
\begin{equation}\label{eq:rk1-coefficient}
 \mathcal{E}_{r,t}=\Sym^{r-1}(\mathbb C^2)
                 \boxtimes\Sym^{t-1}(\mathbb C^2).
\end{equation}
We will fix rational generators of $H^1(\mathfrak h,C_H;I_\pm\otimes \mathcal{E}_{r,t})$ and construct representatives using the test vectors of \eqref{eq:rk1-test-vectors}.
\vspace{1em}

\noindent\textbf{\large The rational generators}
The complexified nilradical of $B_H$ is $\mathfrak n_H=\mathbb C e_1\oplus\mathbb C e_2$, where
\[
 e_1=(E_{12},0),\qquad e_2=(0,E_{12})
 \quad\text{in}\quad
 \mathfrak h=\mathfrak{sl}_2(\mathbb C)\oplus
             \mathfrak{sl}_2(\mathbb C).
\]
The vectors $e_1$ and $e_2$ lie in different direct summands, so $[e_1,e_2]=0$ and $\mathfrak n_H$ is abelian. Alongside them, fix the negative-root vectors
\begin{equation}\label{eq:rk1-lowering-operators}
 f_1=(E_{21},0),\qquad f_2=(0,E_{21}),\qquad
 E_{21}=\begin{pmatrix}0&0\\1&0\end{pmatrix}.
\end{equation}
Use the monomial bases
\[
 A_j=x^{r-1-j}y^j\quad(0\leq j\leq r-1),\qquad
 B_k=x^{t-1-k}y^k\quad(0\leq k\leq t-1)
\]
of the two coefficient factors. On either polynomial factor, $E_{12}$ acts as $x\,\partial/\partial y$ and $E_{21}$ as $y\,\partial/\partial x$, so the negative-root actions are
\begin{equation}\label{eq:rk1-lowering-action}
 \begin{aligned}
 f_1(A_j\otimes B_k)&=(r-1-j)A_{j+1}\otimes B_k
                         &&(j<r-1),\\
 f_2(A_j\otimes B_k)&=(t-1-k)A_j\otimes B_{k+1}
                         &&(k<t-1).
 \end{aligned}
\end{equation}
The first action is zero when $j=r-1$, and the second is zero when $k=t-1$. These are the \emph{lowering operators}, since each lowers the weight in its $\mathfrak{sl}_2$ factor by two. Repeated application expresses every monomial in terms of the highest-weight vector:
\begin{equation}\label{eq:rk1-lowering-monomials}
 A_j\otimes B_k
  =\frac{(r-1-j)!\,(t-1-k)!}{(r-1)!\,(t-1)!}
     f_1^j f_2^k(A_0\otimes B_0).
\end{equation}
Here the powers mean ordinary repeated action, with the factorials accounting for the monomial normalization.

The positive-root action $E_{12}=x\,\partial/\partial y$ has the following kernels and cokernels:
\[
 \begin{aligned}
 \ker\!\left(E_{12}|_{\Sym^{r-1}(\mathbb C^2)}\right)
   &=\mathbb C A_0,&
 \Sym^{r-1}(\mathbb C^2)/E_{12}\Sym^{r-1}(\mathbb C^2)
   &=\mathbb C[A_{r-1}],\\
 \ker\!\left(E_{12}|_{\Sym^{t-1}(\mathbb C^2)}\right)
   &=\mathbb C B_0,&
 \Sym^{t-1}(\mathbb C^2)/E_{12}\Sym^{t-1}(\mathbb C^2)
   &=\mathbb C[B_{t-1}].
 \end{aligned}
\]
Identify the cochains $f:\mathfrak n_H\to \mathcal{E}_{r,t}$ with elements in $(\n_H)^{\vee}\otimes \mathcal{E}_{r,t}$. Taking the cokernel in one factor of $\mathcal{E}_{r,t}$ and the kernel in the other gives the two cochains:
\begin{equation}\label{eq:rk1-kostant}
 h_+=e_1^{\vee}\otimes A_{r-1}\otimes B_0,\qquad
 h_-=e_2^{\vee}\otimes A_0\otimes B_{t-1}.
\end{equation}
Recall that the Chevalley--Eilenberg differential is
\[
 (df)(X,Y)=X\cdot f(Y)-Y\cdot f(X)-f([X,Y]),
 \qquad X,Y\in\mathfrak n_H;
\]
see, for example, \cite[Corollary~7.7.3]{weibel}. Writing $f=e_1^{\vee}\otimes v_1+e_2^{\vee}\otimes v_2$, with $v_i\in \mathcal{E}_{r,t}$, we have $f(e_i)=v_i$. Since $[e_1,e_2]=0$, this becomes
\[
 df=\left(e_1^{\vee}\wedge e_2^{\vee}\right)\otimes(e_1v_2-e_2v_1).
\]
Thus the bracket term vanishes because $\n_H$ is abelian. Here $e_1$ acts on the first coefficient factor and $e_2$ on the second. The kernel identities above give
\[
 \begin{aligned}
 e_2(A_{r-1}\otimes B_0)
   &=A_{r-1}\otimes E_{12}B_0=0,\\
 e_1(A_0\otimes B_{t-1})
   &=E_{12}A_0\otimes B_{t-1}=0.
 \end{aligned}
\]
It follows that $dh_+=dh_-=0$, so both cochains are cocycles. The cochain complex is the tensor product of the cochain complexes for the one-dimensional Lie algebras $\mathbb C e_1$ and $\mathbb C e_2$ with their respective coefficient factors. The preceding kernel and cokernel calculation therefore shows that
\[
 H^1(\mathfrak n_H,\mathcal{E}_{r,t})
    =\mathbb C[h_+]\oplus\mathbb C[h_-].
\]

The diagonal torus $\{\operatorname{diag}(z,z^{-1}):z\in\mathbb C^\times\}$ acts on these classes with weights
\[
 \operatorname{wt}([h_+])=(-r-1,t-1),\qquad
 \operatorname{wt}([h_-])=(r-1,-t-1).
\]
By \eqref{eq:rk1-induction}, the inducing factors $\delta_{B_H}^{1/2}\psi_+$ and $\delta_{B_H}^{1/2}\psi_-$ have weights $(r+1,1-t)$ and $(1-r,1+t)$. Thus $1\otimes[h_\pm]$ spans the torus-invariant line in $\mathbb C_{\delta_{B_H}^{1/2}\psi_\pm} \otimes H^1(\mathfrak n_H,\mathcal{E}_{r,t})$, respectively.

The other degree-one class has nontrivial compact-circle character. The degree-zero class, of weight $(r-1,t-1)$ before twisting, also has no compact invariants after either twist. Consequently the degree-one case of the cohomology formula \cite[III.3.3(ii)]{borel-wallach} shows that
\begin{equation}\label{eq:rk1-edge}
 H^1(\mathfrak h,C_H;I_\pm\otimes \mathcal{E}_{r,t})
 \xrightarrow{\ \sim\ }
 \mathbb C\,(1\otimes[h_\pm]).
\end{equation}
As in \eqref{eq:main-first-evaluation}, the map is evaluation at the identity followed by restriction to $\mathfrak n_H$ and passage to its cohomology. Define $[J_\pm]$ to be the classes whose images are $1\otimes[h_\pm]$. This fixes the rational structure on $H^1(\mathfrak h,C_H;I_\pm\otimes \mathcal{E}_{r,t})$ as $\mathbb Q[J_\pm]$.
\vspace{1em}

\noindent\textbf{\large Cochains constructed from test vectors in the minimal $C_H$-type}
We next compute the scalars of the maps induced by $\operatorname T_H$ on cohomology with respect to the generators $[J_+]$ and $[J_-]$. By \Cref{sec:rk1-model}, the lowest $C_H$-type in both $I_+$ and $I_-$ is
\[
 V_D=\Sym^D(\mathbb C^2),\qquad D=r+t,
\]
and it occurs with multiplicity one. We first identify its two copies explicitly. Let $\ell_x,\ell_y$ extract the coefficients of $x^D,y^D$. For $v\in V_D$, set
\begin{equation}\label{eq:rk1-compact-embeddings}
 j_+(v)(c)=\ell_x(cv),\qquad
 j_-(v)(c)=\ell_y(cv),\qquad c\in C_H,
\end{equation}
and extend to $H$ by the induction law \eqref{eq:rk1-induction} for $\psi_+$ and $\psi_-$, respectively. Left multiplication by $a(e^{i\vartheta})$ acts by the characters $e^{iD\vartheta}$ and $e^{-iD\vartheta}$, so these extensions belong to $I_+$ and $I_-$. Moreover,
\[
 j_\pm(c_0v)(c)=j_\pm(v)(cc_0).
\]
Thus $j_\pm$ are nonzero $C_H$-equivariant maps from the irreducible module $V_D$, and hence are embeddings.

For the standard action $cx=c_{11}x+c_{21}y$ and $cy=c_{12}x+c_{22}y$, comparison of the restrictions to $C_H$ with \eqref{eq:rk1-test-vectors} gives
\[
 \begin{aligned}
 f_{\psi_+}&=j_+(x^D),&
 \widetilde f_{\psi_+}&=j_-(x^D),\\
 f_{\psi_-}&=j_-(y^D),&
 \widetilde f_{\psi_-}&=(-1)^D j_+(y^D).
 \end{aligned}
\]
The last sign comes from $\widetilde f_{\psi_-}(c)=(-c_{12})^D$.

To turn these functions into cochains with coefficient $\mathcal{E}_{r,t}$, we need a common $C_H$-equivariant map to $V_D$. A relative degree-one cochain $J:\mathfrak h/\mathfrak c_H\to I_\pm\otimes \mathcal{E}_{r,t}$ can equivalently be written as
\[
 \widehat J:(\mathfrak h/\mathfrak c_H)\otimes \mathcal{E}_{r,t}^\vee
       \longrightarrow I_\pm,\qquad
 \widehat J(q\otimes\eta)
       =(\operatorname{id}_{I_\pm}\otimes\eta)J(q).
\]
Here $\eta$ evaluates the coefficient factor. We construct the desired map on this domain by polynomial multiplication.

The quotient $\mathfrak h/\mathfrak c_H$ is the adjoint module under $[(X,Y)]\mapsto X+{}^tY$. Put $e_0=(\operatorname{diag}(1,-1),0)$. Its identification with $\Sym^2(\mathbb C^2)$ is
\[
 \beta([e_1])=x^2,\qquad
 \beta([e_0])=-2xy,\qquad
 \beta([e_2])=-y^2.
\]
We next identify the dual coefficient factors with polynomial spaces so that they can be multiplied with $\beta(q)$. The ordinary dual basis vectors are coefficient-extraction functionals:
\[
 A_j^\vee\!\left(\sum_i a_iA_i\right)=a_j,\qquad
 B_k^\vee\!\left(\sum_i b_iB_i\right)=b_k.
\]
To represent these functionals by polynomials, we must specify a pairing on each factor. For $c\in C_H$, the two coefficient factors carry the actions of $c$ and $\bar c$, respectively. Both target polynomial spaces below carry the usual action of $c$.

For the first factor, a homogeneous polynomial $p$ of degree $r-1$ represents the functional
\begin{equation}\label{eq:rk1-determinant-duality}
 f\longmapsto
 \frac{1}{(r-1)!}\,p(\partial_y,-\partial_x)f,
 \qquad f\in\Sym^{r-1}(\mathbb C^2).
\end{equation}
Here $p(\partial_y,-\partial_x)$ means replacing $x$ by $\partial_y$ and $y$ by $-\partial_x$. The result is a scalar because $p$ and $f$ have the same degree. This is the normalized symmetric power of the determinant pairing on $\mathbb C^2$, so it is invariant under $H$.

For the second factor, $c^{\mathsf T}\bar c=1$ gives an invariant bilinear pairing between the usual and conjugate standard representations. Its normalized symmetric power associates a polynomial $p$ of degree $t-1$ with the functional
\begin{equation}\label{eq:rk1-conjugate-duality}
 f\longmapsto
 \frac{1}{(t-1)!}\,p(\partial_x,\partial_y)f,
 \qquad f\in\Sym^{t-1}(\mathbb C^2).
\end{equation}
Here $p$ carries the usual action and $f$ the conjugate action. The pairing is complex bilinear. The conjugation occurs in the group action, not in the coefficients of $p$ or $f$.

The pairings \eqref{eq:rk1-determinant-duality} and \eqref{eq:rk1-conjugate-duality} identify each polynomial space with the dual of its coefficient factor. Their inverses are
\[
 A_j^\vee\longmapsto
       (-1)^{r-1-j}\binom{r-1}{j}x^jy^{r-1-j}, \quad
 B_k^\vee\longmapsto
       \binom{t-1}{k}x^{t-1-k}y^k.
\]
Indeed, substituting the first polynomial into \eqref{eq:rk1-determinant-duality}, for $f=\sum_i a_iA_i$, we have
\[
 \frac{(-1)^{r-1-j}\binom{r-1}{j}}{(r-1)!}
       \partial_y^j(-\partial_x)^{r-1-j}f
 =\frac{\partial_x^{r-1-j}\partial_y^j f}{(r-1-j)!\,j!}
 =a_j.
\]
For $f=\sum_i b_iB_i$, the second polynomial similarly indicates that
\[
 \frac{\binom{t-1}{k}}{(t-1)!}
       \partial_x^{t-1-k}\partial_y^k f
 =\frac{\partial_x^{t-1-k}\partial_y^k f}{(t-1-k)!\,k!}
 =b_k.
\]
Thus the two polynomials represent exactly $A_j^\vee$ and $B_k^\vee$ under the specified pairings. This also explains the signs and binomial coefficients. Both inverse maps are $C_H$-equivariant and have rational matrices in the chosen bases.

Applying $\beta$ and these two identifications, we get the three polynomials of degrees $2,r-1,t-1$. Their product has degree $D=r+t$ and hence lies in $V_D$. This multiplication is the Cartan product used to define
\begin{equation}\label{eq:rk1-cartan-product}
 \begin{split}
 P_{r,t}:(\mathfrak h/\mathfrak c_H)\otimes \mathcal{E}_{r,t}^\vee
       &\longrightarrow V_D,\\
 P_{r,t}(q\otimes A_j^\vee\otimes B_k^\vee)
       &=\beta(q)(-1)^{r-1-j}\binom{r-1}{j}\binom{t-1}{k}
          x^{j+t-1-k}y^{r-1-j+k}.
 \end{split}
\end{equation}
In particular,
\[
 P_{r,t}([e_1]\otimes A_{r-1}^\vee\otimes B_0^\vee)=x^D.
\]
This fixes the normalization of the map.

Now define the two test cochains by
\[
 \widehat J_+^{\mathrm{test}}=j_+\circ P_{r,t},\qquad
 \widehat J_-^{\mathrm{test}}=j_-\circ P_{r,t}.
\]
In the original cochain notation, this means
\[
 J_\pm^{\mathrm{test}}(q)
   =\sum_{j,k}
       j_\pm\!\left(P_{r,t}(q\otimes A_j^\vee\otimes B_k^\vee)\right)
          \otimes A_j\otimes B_k.
\]
At the identity, $j_\pm$ reduce to coefficient extraction. Using \eqref{eq:rk1-cartan-product}, we obtain
\begin{equation}\label{eq:rk1-cochain-evaluation}
 \begin{array}{c|ccc}
 &[e_1]&[e_0]&[e_2]\\ \hline
 (\operatorname{ev}_1\otimes1)J_+^{\mathrm{test}}
    &1\otimes A_{r-1}\otimes B_0&0&0\\
 (\operatorname{ev}_1\otimes1)J_-^{\mathrm{test}}
    &0&0&(-1)^r(1\otimes A_0\otimes B_{t-1}).
 \end{array}
\end{equation}
These are the cochains $1\otimes h_+$ and $(-1)^r(1\otimes h_-)$, extended by zero on the torus Lie algebra. They are closed because $h_\pm$ are cocycles for $\mathfrak n_H$ and their weights are inverse to the inducing characters $\delta_{B_H}^{1/2}\psi_\pm$. The evaluation isomorphism of cochain complexes \eqref{eq:main-first-evaluation}, applied to $H$, therefore proves that $J_\pm^{\mathrm{test}}$ are closed. Their images under \eqref{eq:rk1-edge} give
\begin{equation}\label{eq:rk1-test-to-kostant}
 [J_+^{\mathrm{test}}]=[J_+],\qquad
 [J_-^{\mathrm{test}}]=(-1)^r[J_-].
\end{equation}
This comparison will let us read the cohomological scalar from the operator on the two copies $j_\pm(V_D)$.

\subsection{Regularity of the intertwining operators}
\label{subsec:main-first-regularity}
Before specializing the intertwining operator and passing to cohomology, we need holomorphy on the whole principal series, since \eqref{eq:rk1-scalar} computes only its restriction to the lowest $C_H$-type. We first combine that calculation with irreducibility to prove holomorphy and invertibility for the operator on $H$. We then obtain the same conclusions for the operator associated with a simple root $\alpha$ on $L_\alpha$, and use induction in stages to pass to $G$.
\vspace{1em}

\noindent\textbf{\large The operator on \texorpdfstring{$H$}{H}}
For either $\psi=\psi_+$ or $\psi=\psi_-$ from \Cref{subsec:main-first-rank-one}, vary the radial part of the inducing character by setting
\[
 \psi_s=\psi|\cdot|_{\mathbb C}^{s},\qquad
 \operatorname{T}_H(\psi_s):
 I_H(\psi_s)\longrightarrow I_H(\psi_s^{-1}),\qquad
 \psi_s^{-1}=\psi^{-1}|\cdot|_{\mathbb C}^{-s}.
\]
The parameter $s=0$ recovers the fixed character $\psi$, while the inverse character in the target varies with radial parameter $-s$. Since these twists are trivial on the compact torus, restriction to $C_H$ realizes the source and target families in fixed compact pictures, in which the $H$-actions depend holomorphically on $s$. We use the same compact functions to define the embeddings $j_\pm$ from \eqref{eq:rk1-compact-embeddings} throughout these families. In these realizations, the continuation theorem ensures that $\operatorname{T}_H(\psi_s)$ is a meromorphic family of continuous operators whose poles have finite order \cite[Theorem~3.1]{knapp-stein-II}.

\begin{lemma}\label{lem:rk1-regularity}
For either $\psi=\psi_+$ or $\psi=\psi_-$, the family $\operatorname{T}_H(\psi_s)$ is holomorphic at $s=0$, and its value is an isomorphism $I_H(\psi)\to I_H(\psi^{-1})$ of smooth principal series and of their Harish--Chandra modules.
\end{lemma}
\begin{proof}
Both $I_+$ and $I_-$ are irreducible by the criterion used in \Cref{sec:irreducibility-principal-series}: the Cartan-conjugate coroot evaluations have product $rt>0$, and there are no real absolute roots, so the Speh--Vogan obstruction \eqref{eq:irr-SV-signs} cannot occur.

The identities following \eqref{eq:rk1-compact-embeddings} identify an extremal vector in each copy of $V_D$, on which \Cref{prop:rk1-scalar} computes the intertwining operator explicitly. Because these vectors generate their respective copies of $V_D$ under $C_H$, $C_H$-equivariance extends the calculation to the identities
\[
 \operatorname{T}_H(\psi_+|\cdot|_{\mathbb C}^{s})\circ j_+
   =\frac{2\pi}{r+s}j_-,\quad
 \operatorname{T}_H(\psi_-|\cdot|_{\mathbb C}^{s})\circ j_-
   =(-1)^D\frac{2\pi}{t+s}j_+.
\]
Since $r,t>0$, both restrictions are holomorphic and nonzero at $s=0$.

To prove holomorphy of $\operatorname T_H(\psi_s)$, suppose that it has a pole of order $m>0$ at $s=0$, and let
\[
 A=\lim_{s\to0}s^m\operatorname{T}_H(\psi_s)
\]
be its nonzero leading Laurent coefficient. Multiplying the equivariance identity by $s^m$ and taking the limit as $s\to0$, using the holomorphic dependence of the $H$-actions in the fixed compact pictures, shows that $A$ is a continuous $H$-intertwiner from $I_H(\psi)$ to $I_H(\psi^{-1})$. The formulas for the restrictions to $j_+(V_D)$ and $j_-(V_D)$ imply that $A$ vanishes on the corresponding nonzero subspace of its source, because the restriction of the family to that subspace has no pole. The kernel of $A$ on the source Harish--Chandra module is therefore a nonzero submodule, which must be the whole module by irreducibility. Continuity then implies that $A$ vanishes on the smooth principal series as well, since the $C_H$-finite vectors are dense.

This contradicts the choice of $A$ as a nonzero leading coefficient and proves that the family is holomorphic at $s=0$. Its value is nonzero on the corresponding copy $j_+(V_D)$ or $j_-(V_D)$, so it is an isomorphism between the irreducible Harish--Chandra modules.

To construct a continuous inverse on the smooth principal series, we specialize the formulas for the restrictions to $j_\pm(V_D)$ at $s=0$:
\begin{equation}\label{eq:rk1-embedding-scalars}
 \operatorname{T}_H(\psi_+)\circ j_+=\frac{2\pi}{r}j_-,
 \qquad
 \operatorname{T}_H(\psi_-)\circ j_-=(-1)^D\frac{2\pi}{t}j_+.
\end{equation}
Composing the two operators in either order acts on the corresponding copy of $V_D$ by the nonzero scalar $(-1)^D(2\pi)^2/(rt)$. Since this compact type generates its irreducible Harish--Chandra module, equivariance extends each scalar identity to every $C_H$-finite vector, and continuity then extends it to the entire smooth principal series. Thus the reverse operator multiplied by $(-1)^D rt/(2\pi)^2$ is a continuous inverse.
\end{proof}
\vspace{1em}

\noindent\textbf{\large The operators on $G$}
To apply the rank-one result to a step $0\leq a<d_P$ of the reduced expression in \Cref{subsec:main-first-intermediates}, put $\alpha=\alpha_{a+1}$ and let $P_\alpha$ be the standard parabolic whose Levi subgroup $L_\alpha$ has derived root system $\{\pm\alpha\}$. The pinned homomorphism $\varphi_\alpha:H\to L_\alpha^1$ from \eqref{eq:rk1-root-map} sends $a(z)$ to $\alpha^\vee(z)$, so the inducing characters that enter the comparison with $H$ are
\[
 \psi_\alpha=\chi^{(a)}\circ\alpha^\vee,\qquad
 \chi_s^{(a)}={}^{w_{(a)}}\chi_s,\qquad
 \psi_{\alpha,s}=\chi_s^{(a)}\circ\alpha^\vee.
\]
Equation \eqref{eq:main-first-root-exponents} identifies $\psi_\alpha$ with $\psi_+$ or $\psi_-$ for $r=r_\alpha$ and $t=t_\alpha$, while the reflected inducing character satisfies $\chi^{(a+1)}\circ\alpha^\vee=\psi_\alpha^{-1}$. If $w_{(a)}^{-1}(\alpha)\in\Phi_j(P)$, \eqref{eq:parabolic-root-twist} gives
\[
 \psi_{\alpha,s}
   =\psi_\alpha|\cdot|_{\mathbb C}^{j(s-s_0)}.
\]
Thus $\psi_{\alpha,s}$ is the family used in \Cref{lem:rk1-regularity}, with the parameter there replaced by $j(s-s_0)$. Its value at $s=s_0$ is $\psi_\alpha$.

We first compare $\operatorname T_H(\psi_{\alpha,s})$ with the standard intertwining operator between the principal series representations of $L_\alpha$,
\[
 \operatorname{T}^{L_\alpha}_\alpha(\chi_s^{(a)}):
 I_{L_\alpha}(\chi_s^{(a)})
 \longrightarrow I_{L_\alpha}({}^{s_\alpha}\chi_s^{(a)}).
\]
Pulling functions back along $\varphi_\alpha$ identifies the source and target, as $H$-representations, with $I_H(\psi_{\alpha,s})$ and $I_H(\psi_{\alpha,s}^{-1})$, respectively, because \eqref{eq:rk1-modulus} shows that pullback respects the modulus factors in normalized induction. A function in either principal series of $L_\alpha$ is determined by its pullback and its prescribed central character: this follows from $L_\alpha=Z(L_\alpha)^\circ\varphi_\alpha(H)$, with compatibility on the finite central intersection supplied by the inducing character. With the root coordinates and measure fixed in \eqref{eq:rk1-root-map}--\eqref{eq:rk1-measure}, the two intertwining integrals agree under this identification:
\[
 \bigl(\operatorname{T}^{L_\alpha}_\alpha(\chi_s^{(a)})f\bigr)
       \circ\varphi_\alpha
 =\operatorname{T}_H(\psi_{\alpha,s})(f\circ\varphi_\alpha).
\]
The identity holds initially in the region of convergence and extends to the meromorphic families by continuation. Hence \Cref{lem:rk1-regularity}, with the parameter substitution above, proves that $\operatorname{T}^{L_\alpha}_\alpha(\chi_s^{(a)})$ is holomorphic at $s_0$ and has an invertible value there.

We now pass from $L_\alpha$ to $G$ by using induction in stages \eqref{eq:parabolic-induction-in-stages} to realize
\[
 I(\chi_s^{(a)})\simeq
 \operatorname{Ind}_{P_\alpha}^G
       \bigl(I_{L_\alpha}(\chi_s^{(a)})\bigr).
\]
For $f\in I(\chi_s^{(a)})$, the corresponding function $F_f$ on $G$ has values in $I_{L_\alpha}(\chi_s^{(a)})$ and is given by
\begin{equation}\label{eq:main-first-stages-function}
 F_f(g)(l)=\delta_{P_\alpha}(l)^{-1/2}f(lg),
 \qquad g\in G,\quad l\in L_\alpha.
\end{equation}
The inverse identification evaluates the function $F_f(g)$ on $L_\alpha$ at the identity, giving $f(g)=F_f(g)(1)$. Under this identification, $\operatorname T_\alpha^G$ acts by applying $\operatorname T_\alpha^{L_\alpha}$ to each value $F_f(g)$, as the following calculation shows:
\begin{equation}\label{eq:main-stage-integral-identity}
 \begin{aligned}
 F_{\operatorname{T}_\alpha^G f}(g)(l)
 &=\delta_{P_\alpha}(l)^{-1/2}
   \int_{\mathbb C}
      f(\dot s_\alpha^{-1}x_\alpha(z)lg)\,d\mu_{\mathbb C}(z)\\
 &=\bigl(\operatorname{T}^{L_\alpha}_\alpha F_f(g)\bigr)(l).
 \end{aligned}
\end{equation}
Here both operators have inducing character $\chi_s^{(a)}$, and the second equality uses the fact that $\delta_{P_\alpha}$ is trivial on the root subgroup and on $\dot s_\alpha$. As with the preceding integral identity, the calculation in the region of convergence extends meromorphically in $s$.

In fixed compact pictures, applying the holomorphic family $\operatorname T_\alpha^{L_\alpha}(\chi_s^{(a)})$ to each value of an induced function gives a holomorphic family of operators on $G$. At $s=s_0$, applying its inverse in the same way gives the inverse of $\operatorname T_\alpha^G(\chi^{(a)})$. Therefore every factor associated with a simple reflection is holomorphic at $s_0$ with invertible value, and the finite factorization in \Cref{lem:standard-factorization} gives the same conclusions for $\operatorname{T}_{\mathrm{st}}(s)$.

\subsection{Comparison of cohomology for \texorpdfstring{$G$ and $\mathrm{SL}_2(\mathbb C)$}{G and SL2(C)}}
\label{subsec:main-first-transitivity}
We now compute the action of the intertwining operators on the cohomology classes constructed in \Cref{subsec:main-first-rank-one}, using the holomorphy proved in \Cref{lem:rk1-regularity} to specialize the families at $\psi_\pm$. Since the test cochains are obtained by composing $j_\pm$ with the same map $P_{r,t}$ from \eqref{eq:rk1-cartan-product}, the identities \eqref{eq:rk1-embedding-scalars} for the maps $j_\pm$ also determine the operator on these cochains. Expressing the result in the prescribed generators by means of $[J_+^{\mathrm{test}}]=[J_+]$ and $[J_-^{\mathrm{test}}]=(-1)^r[J_-]$ from \eqref{eq:rk1-test-to-kostant}, we obtain
\begin{equation}\label{eq:rk1-cohom-scalars}
 \operatorname{T}_H(\psi_+)^\bullet[J_+]
       =(-1)^r\frac{2\pi}{r}[J_-],\quad
 \operatorname{T}_H(\psi_-)^\bullet[J_-]
       =(-1)^t\frac{2\pi}{t}[J_+].
\end{equation}
The positive scalar factors are $c(\psi_+)=2\pi/r$ and $c(\psi_-)=2\pi/t$, as in \eqref{eq:rk1-scalar}. Division by $c(\psi_\pm)$ sends each chosen rational generator to the other with the displayed sign.

To transfer this calculation to $G$, let's fix a step $\alpha=\alpha_{a+1}$, with $0\leq a<d$, and write $P_\alpha=L_\alpha N_\alpha$ for the parabolic introduced in \Cref{subsec:main-first-regularity}. Retain $\psi_\alpha=\chi^{(a)}\circ\alpha^\vee$, and take $r=r_\alpha$ and $t=t_\alpha$ as in \eqref{eq:main-first-root-exponents}. For $b=a,a+1$, we denote the following degree-one cohomology groups of $H$ by
\[
 \mathscr{H}_\alpha(\chi^{(b)})
   =H^1(\mathfrak h,C_H;
       I_H(\chi^{(b)}\circ\alpha^\vee)\otimes \mathcal{E}_{r,t}).
\]
These cohomology groups are defined for $H=\mathrm{SL}_2(\mathbb C)$ and the inducing characters $\chi^{(b)}\circ\alpha^\vee$. For each $\mathscr{H}_{\alpha}(\chi^{(b)})$, we choose the generator $[J_{\alpha,\chi^{(b)}}]$ to be $[J_\pm]$ as in \eqref{eq:rk1-edge}. Here we take $\chi^{(b)}\circ\alpha^\vee=\psi_\pm$, where $\pm$ is determined as in \eqref{eq:rk1-cohom-data}.

Since $\chi^{(a+1)}\circ\alpha^\vee=\psi_\alpha^{-1}$, the two cases of \eqref{eq:rk1-cohom-scalars} can be written together as
\begin{equation}\label{eq:main-first-root-line-formula}
 \operatorname{T}_H(\psi_\alpha)^\bullet[J_{\alpha,\chi^{(a)}}]
  =\epsilon_{\alpha,\chi^{(a)}}
    \frac{L_{\mathbb C}(0,\psi_\alpha)}
         {L_{\mathbb C}(1,\psi_\alpha)}
       [J_{\alpha,\chi^{(a+1)}}],
 \qquad
 \epsilon_{\alpha,\chi^{(a)}}=
 \begin{cases}
 (-1)^r,&\psi_\alpha=\psi_+,\\
 (-1)^t,&\psi_\alpha=\psi_-.
 \end{cases}
\end{equation}

We now compare these generators for the cohomology of $H$ with the generators $[J_{\chi^{(a)}}]$ and $[J_{\chi^{(a+1)}}]$ for the cohomology of $G$, fixed independently in \Cref{subsec:main-first-Delorme}. The following lemma constructs comparison maps $\mathcal S_{\alpha,\chi^{(b)}}$ and proves two properties:
\begin{enumerate}
\item they carry the generators for the cohomology of $G$ to nonzero rational multiples of the generators for the cohomology of $H$;
\item and they commute with the respective intertwining operators.
\end{enumerate}
The first property compares the chosen rational structures, while the second determines the scalar of $(\operatorname T_\alpha^G)^\bullet$ from the scalar of $\operatorname T_H(\psi_\alpha)^\bullet$. We use the isomorphisms $\mathscr K_{\chi^{(b)}}^G$ in \eqref{eq:main_kostant}, writing $\chi^{(b)}=\boxtimes_j(\chi^{(b)})^j$.

\begin{lemma}[Comparison of cohomology classes for $G$ and $H$]
\label{lem:main-first-Kostant-transitivity}
The pinning and the rational forms specified in \Cref{subsec:main-first-Delorme} determine isomorphisms
\[
 \mathcal S_{\alpha,\chi^{(b)}}:
 \mathscr H_G(\chi^{(b)})\xrightarrow{\ \sim\ }
 \mathscr{H}_\alpha(\chi^{(b)}),\qquad b=a,a+1,
\]
which preserve the chosen rational structures in the sense that
\begin{equation}\label{eq:main-first-staged-rational}
 \mathcal S_{\alpha,\chi^{(b)}}[J_{\chi^{(b)}}]
   =d_{\alpha,\chi^{(b)}}[J_{\alpha,\chi^{(b)}}],
 \qquad d_{\alpha,\chi^{(b)}}\in\mathbb Q^\times,
 \quad b=a,a+1,
\end{equation}
and are compatible with the intertwining operators through
\begin{equation}\label{eq:main-first-staged-naturality}
 \mathcal S_{\alpha,\chi^{(a+1)}}(\operatorname{T}_\alpha^G)^\bullet
   =\operatorname{T}_H(\psi_\alpha)^\bullet
       \mathcal S_{\alpha,\chi^{(a)}}.
\end{equation}
Using the isomorphisms $\mathscr K_{\chi^{(b)}}^G$ from \eqref{eq:main_kostant}, these properties are expressed by the commutative diagram
\begin{equation}\label{eq:main-actual-stage-square}
\begin{tikzcd}[column sep=5em,row sep=large]
 \mathscr H_G(\chi^{(a)})
    \arrow[r,"(\operatorname{T}_\alpha^G)^\bullet"]
    \arrow[d,"\mathscr{K}_{\chi^{(a)}}^G"']
 &\mathscr H_G(\chi^{(a+1)})
    \arrow[d,"\mathscr{K}_{\chi^{(a+1)}}^G"]\\
 \bigotimes_j\mathscr H^0((\chi^{(a)})^j)
    \arrow[r,"a_{\alpha,\chi^{(a)}}c(\psi_\alpha)"]
    \arrow[d,"d_{\alpha,\chi^{(a)}}"']
 &\bigotimes_j\mathscr H^0((\chi^{(a+1)})^j)
    \arrow[d,"d_{\alpha,\chi^{(a+1)}}"]\\
 \mathscr{H}_\alpha(\chi^{(a)})
    \arrow[r,"\operatorname{T}_H(\psi_\alpha)^\bullet"']
 &\mathscr{H}_\alpha(\chi^{(a+1)}).
\end{tikzcd}
\end{equation}
Here the scalar in the middle horizontal arrow is defined by
\[
 c(\psi_\alpha)=
 \frac{L_{\mathbb C}(0,\psi_\alpha)}{L_{\mathbb C}(1,\psi_\alpha)},
 \qquad
 a_{\alpha,\chi^{(a)}}=
 \epsilon_{\alpha,\chi^{(a)}}
 \frac{d_{\alpha,\chi^{(a)}}}{d_{\alpha,\chi^{(a+1)}}}
 \in\mathbb Q^\times,
\]
and that arrow sends $\Omega_{\chi^{(a)}}$ to $a_{\alpha,\chi^{(a)}}c(\psi_\alpha)\Omega_{\chi^{(a+1)}}$.

For $b=a,a+1$, the lower vertical arrow in the corresponding column sends $\Omega_{\chi^{(b)}}$ to $d_{\alpha,\chi^{(b)}}[J_{\alpha,\chi^{(b)}}]$. Therefore its composite with $\mathscr{K}_{\chi^{(b)}}^G$ is the comparison map $\mathcal S_{\alpha,\chi^{(b)}}$.
\end{lemma}

\begin{proof}
We construct the maps by computing $\mathscr H_G(\chi^{(a)})$ and $\mathscr H_G(\chi^{(a+1)})$ in stages, first taking cohomology along the unipotent radical $N_\alpha$ of $P_\alpha$ and then computing relative Lie algebra cohomology for $L_\alpha$. Factoring the contributing Weyl element into an element of $W_{\rm abs}(L_\alpha)$ and a minimal coset representative identifies the coefficient representation after the first stage and explains why it is the same for both inducing characters.

Let $\mathfrak l_\alpha$ and $\n^{(\alpha)}$ be the complexified real Lie algebras of $L_\alpha$ and $N_\alpha$, respectively, and put $K_{L_\alpha}^{G}:=K\cap L_\alpha$. Writing $B_\alpha=B\cap L_\alpha$ and denoting its complexified nilradical by $\mathfrak n_\alpha$, we have the decomposition
\[
 \mathfrak n=\mathfrak n_{\alpha}\ltimes\n^{(\alpha)}.
\]
For this fixed step, choose $\ulw^{(\alpha)}$ to be the minimal-length representative of $W_{\rm abs}(L_\alpha)\underline w_{\chi^{(a)}}$, and let $\ulw_\alpha\in W_{\rm abs}(L_\alpha)$ be the remaining factor. The factorization \cite[III.1.4]{borel-wallach}, with the convention \eqref{eq:weyl_P}, then takes the form
\begin{equation}\label{eq:main-first-Kostant-factorization}
 \begin{gathered}
 \underline w_{\chi^{(a)}}=\ulw_{\alpha}\ulw^{(\alpha)},\qquad
 \ulw_{\alpha}\in W_{\rm abs}(L_\alpha),\quad
 \ulw^{(\alpha)}\in W_{\rm abs}^{P_\alpha},\\
 \ell(\underline w_{\chi^{(a)}})
    =\ell(\ulw_{\alpha})+\ell(\ulw^{(\alpha)}).
 \end{gathered}
\end{equation}
The signs of the exponents of $\psi_\alpha$ determine whether $\ulw_\alpha=(s_\alpha,1)$ or $(1,s_\alpha)$, according as $\psi_\alpha=\psi_+$ or $\psi_-$. 
In either case its length is one, so the equality $\ell(\underline w_{\chi^{(a)}})=q_0(G)$ in \eqref{eq:main-first-matching} implies
\[
 \ell(\ulw_{\alpha})=1,\qquad
 \ell(\ulw^{(\alpha)})=q_0(G)-1.
\]
By \eqref{eq:main-first-absolute-elements}, the Weyl elements for $\chi^{(a)}$ and $\chi^{(a+1)}$ satisfy
\[
 \underline w_{\chi^{(a+1)}}
 \ = \ \jmath(s_\alpha)\underline w_{\chi^{(a)}} 
 \ = \ \bigl(\jmath(s_\alpha)\ulw_{\alpha}\bigr)\ulw^{(\alpha)}.
\]
Since both $\jmath(s_\alpha)$ and $\ulw_\alpha$ belong to $W_{\rm abs}(L_\alpha)$, the factor in this subgroup becomes $\ulw'_\alpha=\jmath(s_\alpha)\ulw_\alpha$, while $\ulw^{(\alpha)}$ remains the minimal representative of the same coset. Thus the factor indexing the cohomology of $\n^{(\alpha)}$ is unchanged when $\chi^{(a)}$ is replaced by $\chi^{(a+1)}$.

For both $\chi^{(a)}$ and $\chi^{(a+1)}$, the contributing summand is therefore
\[
 F_{\ulw^{(\alpha)}}\subset
 H^{q_0(G)-1}(\n^{(\alpha)},\mathcal V_{\ulb}),
\]
whose highest-weight vector $[h_{\ulw^{(\alpha)}}^{\rm coch}]$ is normalized by \eqref{eq:main-first-Kostant}. Under normalized induction from $P_\alpha$, the coefficient representation in the remaining cohomology of $(\mathfrak l_\alpha,K_{L_\alpha}^{G})$ includes the factor $\delta_{P_\alpha}^{1/2}$:
\begin{equation}\label{eq:main-first-inner-coefficient}
 F_\alpha=\mathbb C_{\delta_{P_\alpha}^{1/2}}
                     \otimes F_{\ulw^{(\alpha)}}.
\end{equation}
Equation~\eqref{eq:main-first-matching} says that the weight $\underline w_{\chi^{(a)}}\cdot\underline\lambda$ cancels $d(\delta_B^{1/2}\chi^{(a)})$. Regularity of $\underline\lambda+\underline\rho_G$ makes this the unique weight in Kostant's decomposition of $H^\bullet(\mathfrak n,V_{\ulb})$ with that property; all other summands have zero relative torus cohomology after tensoring with $\mathbb C_{\delta_B^{1/2}\chi^{(a)}}$. The factorization \eqref{eq:main-first-Kostant-factorization} therefore identifies $F_{\ulw^{(\alpha)}}$ as the contributing summand when cohomology is computed in stages. Applying \cite[III.3.3(ii)]{borel-wallach} to this summand gives
\begin{equation}\label{eq:main-stage-map-composition}
 \mathscr H_G(\chi^{(a)})\xrightarrow{\ \sim\ }
 H^1\!\left(\mathfrak l_\alpha,K_{L_\alpha}^{G};
     I_{L_\alpha}(\chi^{(a)})\otimes F_\alpha\right).
\end{equation}
Here the remaining degree is one because the nilradical cohomology contributing $F_{\ulw^{(\alpha)}}$ already accounts for $q_0(G)-1$ degrees.

To identify the pullback of $F_\alpha$ along $\varphi_\alpha$ with $\mathcal E_{r,t}$, restrict \eqref{eq:main-first-matching} along $\alpha^\vee$. The two highest weights of the pullback of $F_\alpha$ are $r-1$ and $t-1$, since $\delta_{P_\alpha}^{1/2}$ is trivial on $H$, so the pullback is isomorphic to $\mathcal E_{r,t}$ from \eqref{eq:rk1-coefficient}. We normalize the equivariant identification by requiring that it send $A_0\otimes B_0$ to $1\otimes[h_{\ulw^{(\alpha)}}^{\rm coch}]$. The lowering operators of \eqref{eq:rk1-lowering-operators} map under the complexified differential of $\varphi_\alpha$ to $X_{-\alpha^{(1)}}$ and $X_{-\alpha^{(2)}}$, respectively, so equivariance and \eqref{eq:rk1-lowering-monomials} specify this identification on every basis vector:
\begin{equation}\label{eq:main-first-coefficient-identification}
 A_j\otimes B_k\longmapsto
 \frac{(r-1-j)!\,(t-1-k)!}{(r-1)!\,(t-1)!}
 X_{-\alpha^{(1)}}^jX_{-\alpha^{(2)}}^k
 \bigl(1\otimes[h_{\ulw^{(\alpha)}}^{\rm coch}]\bigr).
\end{equation}
Here the root vectors act through the action of $L_\alpha$ on $H^{q_0(G)-1}(\n^{(\alpha)},\mathcal V_{\ulb})$. Since the coefficient representation $F_\alpha$ is the same for $\chi^{(a)}$ and $\chi^{(a+1)}$, we use this identification in both cases.

Equation~\eqref{eq:main-first-matching} also ensures that the central characters of $I_{L_\alpha}(\chi^{(a)})$ and $F_\alpha$ cancel on their tensor product. Since $I_H(\psi_\alpha)$ is irreducible by \Cref{lem:rk1-regularity} and infinite-dimensional, its tensor product with the finite-dimensional coefficient $\mathcal E_{r,t}$ has no degree-zero cohomology on $H$. Consequently, in the total degree-one cohomology of \eqref{eq:main-stage-map-composition}, the cohomology contributed by the center of $L_\alpha$ must be in degree zero. Using its canonical generator $1$, pullback along $\varphi_\alpha$ identifies the right-hand side of \eqref{eq:main-stage-map-composition} with $\mathscr{H}_\alpha(\chi^{(a)})$. Note that any finite kernel of $\varphi_\alpha$ acts trivially because both the representation $I_{L_\alpha}(\chi^{(a)})$ and the coefficient representation $F_\alpha$ are pulled back from $L_\alpha$. 

Composing these identifications defines $\mathcal S_{\alpha,\chi^{(a)}}$, and $\mathcal S_{\alpha,\chi^{(a+1)}}$ is constructed in the same way using $\ulw'_\alpha$, with the common coefficient module and its identification with $\mathcal E_{r,t}$ held fixed.

We next check how these maps act on the prescribed rational generators by tracing the Kostant cocycle $h_{\chi^{(a)}}^{\rm coch}$ through the construction.

The rational root spaces and coefficient representation in \eqref{eq:main-rational-nilradical-complex} define the nilradical complexes over $\mathbb Q$, and applying the ordered operators for the negative roots of $L_\alpha$ to $[h_{\ulw^{(\alpha)}}^{\rm coch}]$ constructs the rational form of $F_{\ulw^{(\alpha)}}$. The same construction for the other Kostant summands makes their decomposition and projections rational, while the chosen highest-weight normalization ensures that the identification with $\mathcal E_{r,t}$ is also rational.

The length-additive factorization \eqref{eq:main-first-Kostant-factorization} implies the decomposition:
\begin{equation}\label{eq:main-first-inversion-wedge}
 \Phi_{\ulw_{\alpha}\ulw^{(\alpha)}}
   =\Phi_{\ulw_{\alpha}}\sqcup\ulw_{\alpha}(\Phi_{\ulw^{(\alpha)}}).
\end{equation}
Since $\ell(\ulw_\alpha)=1$, the first set is the singleton $\{\alpha^{(i)}\}$, where $i=1$ if $\ulw_\alpha=(s_\alpha,1)$ and $i=2$ if $\ulw_\alpha=(1,s_\alpha)$. Its contribution to the exterior factor is therefore the single covector $X_{\alpha^{(i)}}^*$, while the roots in the second set belong to $\n^{(\alpha)}$. The pinned representative $\dot{\ulw}_\alpha$ in the complexification of $L_\alpha$ preserves $\n^{(\alpha)}$ and acts on its cochains through both the coefficient representation and the dual adjoint action on the exterior factors. This action preserves the rational forms fixed in \eqref{eq:main-rational-nilradical-complex}, so, with the argument in $\n_\alpha$ placed first, the cocycle $h_{\chi^{(a)}}^{\rm coch}$ on $\mathfrak n$ is a nonzero rational multiple of
\begin{equation}\label{eq:main-first-staged-cocycle}
 X_{\alpha^{(i)}}^*
       \wedge\dot{\ulw}_\alpha(h_{\ulw^{(\alpha)}}^{\rm coch}).
\end{equation}
Here the exterior product combines this degree-one covector with the degree-$(q_0(G)-1)$ cocycle on $\n^{(\alpha)}$, so the result is a degree-$q_0(G)$ cochain on $\mathfrak n$. The nilradical part of the comparison map takes the component
\begin{equation}\label{eq:main-stage-cochain-component}
 Y\longmapsto
 \left[
 (Z_1,\ldots,Z_{q_0(G)-1})\longmapsto
 h_{\chi^{(a)}}^{\rm coch}(Y,Z_1,\ldots,Z_{q_0(G)-1})
 \right],
 \quad Y\in\n_{\alpha},\quad Z_i\in\n^{(\alpha)}.
\end{equation}
For each $Y$, the cochain inside the brackets is a scalar multiple of $\dot{\ulw}_\alpha(h_{\ulw^{(\alpha)}}^{\rm coch})$, which is closed because the action of $L_\alpha$ commutes with the $\n^{(\alpha)}$-differential. Taking its cohomology class while retaining $Y$ as the argument therefore produces a degree-one cochain on $\n_\alpha$ with values in $F_{\ulw^{(\alpha)}}$. Under the identification \eqref{eq:main-first-coefficient-identification}, the class $1\otimes\dot{\ulw}_\alpha[h_{\ulw^{(\alpha)}}^{\rm coch}]$ corresponds, up to a nonzero rational scalar, to $A_{r-1}\otimes B_0$ when $\ulw_\alpha=(s_\alpha,1)$, and to $A_0\otimes B_{t-1}$ when $\ulw_\alpha=(1,s_\alpha)$.

The remaining covector pulls back to $e_1^\vee$ or $e_2^\vee$ under the complexified differential of $\varphi_\alpha$. Consequently, the pullback of the resulting cocycle is a nonzero rational multiple of $h_+$ or $h_-$ from \eqref{eq:rk1-kostant}.

To compare this construction with $\mathscr K_{\chi^{(a)}}^G$ in \eqref{eq:main_kostant}, use the induction-in-stages identification \eqref{eq:main-first-stages-function}: evaluating first at $g=1$ and then at $l=1$ gives $F_f(1)(1)=f(1)$. The map on nilradical cohomology is described in \eqref{eq:main-stage-cochain-component}, and the normalization factors satisfy
\[
 \delta_B^{1/2}|_T
   =\bigl(\delta_{B_\alpha}^{1/2}|_T\bigr)
      \bigl(\delta_{P_\alpha}^{1/2}|_T\bigr).
\]
Both computations use the vector $1$ in degree-zero cohomology for the center of $L_\alpha$. Their normalizations therefore differ by the nonzero rational factors identified above. This proves \eqref{eq:main-first-staged-rational} for $\chi^{(a)}$, and the same argument with $\ulw'_\alpha$ proves it for $\chi^{(a+1)}$.

It remains to check that the comparison maps commute with the intertwining operators. Under induction in stages through $P_\alpha$, evaluation at $g=1$ takes values in $I_{L_\alpha}(\chi^{(a)})$, and \eqref{eq:main-stage-integral-identity} gives
\[
 F_{\operatorname T_\alpha^G f}(1)
   =\operatorname T_\alpha^{L_\alpha}\bigl(F_f(1)\bigr).
\]
Thus evaluation at $g=1$ intertwines $\operatorname T_\alpha^G$ with $\operatorname T_\alpha^{L_\alpha}$, and passage to the cohomology of $\n^{(\alpha)}$ has the same compatibility. The two inducing characters give the same summand $F_{\ulw^{(\alpha)}}$ and the same identification of the pullback of $F_\alpha$ along $\varphi_\alpha$ with $\mathcal E_{r,t}$ in \eqref{eq:main-first-coefficient-identification}. After taking degree-zero cohomology for the center of $L_\alpha$ and pulling back along $\varphi_\alpha$, the operator $\operatorname T_\alpha^{L_\alpha}$ becomes $\operatorname T_H(\psi_\alpha)$. These compatibilities establish \eqref{eq:main-first-staged-naturality}.

Applying \eqref{eq:main-first-staged-naturality} to $[J_{\chi^{(a)}}]$ gives the scalar in the diagram: $\mathcal S_{\alpha,\chi^{(a)}}$ contributes $d_{\alpha,\chi^{(a)}}$, the operator $\operatorname T_H(\psi_\alpha)^\bullet$ contributes $\epsilon_{\alpha,\chi^{(a)}}c(\psi_\alpha)$, and applying $\mathcal S_{\alpha,\chi^{(a+1)}}^{-1}$ divides by $d_{\alpha,\chi^{(a+1)}}$. Hence $(\operatorname T_\alpha^G)^\bullet$ sends $[J_{\chi^{(a)}}]$ to $a_{\alpha,\chi^{(a)}}c(\psi_\alpha)[J_{\chi^{(a+1)}}]$, as in the middle horizontal arrow of \eqref{eq:main-actual-stage-square}.
\end{proof}

\subsection{Completion of the proof of \texorpdfstring{\Cref{thm:mainthm}}{the main theorem}}
\label{subsec:main-first-finish}
\begin{proof}[Proof of \Cref{thm:mainthm}]
The diagram \eqref{eq:main-actual-stage-square} implies that, for $\alpha=\alpha_{a+1}$ and $0\leq a<d_P$,
\begin{equation}
\label{eq:main-rank-one-input}
 (\operatorname{T}_\alpha^G)^\bullet[J_{\chi^{(a)}}]
   =a_{\alpha,\chi^{(a)}}
     \frac{L_{\mathbb C}(0,\chi^{(a)}\circ\alpha^\vee)}
          {L_{\mathbb C}(1,\chi^{(a)}\circ\alpha^\vee)}
                 [J_{\chi^{(a+1)}}],
 \qquad
 a_{\alpha,\chi^{(a)}}\in\mathbb Q^\times.
\end{equation}
Compose along the reduced expression in \Cref{subsec:main-first-intermediates}. \Cref{lem:standard-factorization} gives
\[
 \operatorname{T}_{\mathrm{st}}(s_0)
    =\operatorname{T}_{\alpha_{d_P}}^G(\chi^{(d_P-1)})\circ\cdots\circ
                              \operatorname{T}_{\alpha_1}^G(\chi^{(0)}).
\]
\Cref{subsec:main-first-regularity} proves that every factor is holomorphic and invertible at the evaluation point. Hence their composition $\operatorname{T}_{\mathrm{st}}(s_0)$ is an isomorphism. Tensoring with $V_{\ulb}$ gives an isomorphism of the relative cochain complexes and hence of the cohomology groups in every degree.

In bottom degree $q_0(G)$, the source and target are generated by $[J_\chi]$ and $[J_{\widetilde\chi}]$, with the rational structures \eqref{eq:main-prescribed-rational-line} for $\chi$ and $\widetilde\chi$, respectively. The coefficient of $\operatorname{T}_{\mathrm{st}}^\bullet(s_0)$ between these generators is the product of the scalars in \eqref{eq:main-rank-one-input}. Put $\beta_{(a)}=w_{(a-1)}^{-1}(\alpha_a)$ for $1\leq a\leq d_P$. These roots enumerate $\Phi(P)$ once. By the Weyl action of \Cref{sec:inter_op} and the definition $\chi^{(a-1)}={}^{w_{(a-1)}}\chi$,
\[
 \chi^{(a-1)}\circ\alpha_a^\vee
   =({}^{w_{(a-1)}}\chi)\circ\alpha_a^\vee
   =\chi\circ\beta_{(a)}^\vee.
\]
Thus each step uses the fixed character $\chi$ restricted along $\beta_{(a)}^\vee$, and the product becomes
\begin{equation}
\label{eq:main-L-product}
 \prod_{a=1}^{d_P}
 \frac{L_{\mathbb C}(0,\chi^{(a-1)}\circ\alpha_a^\vee)}
      {L_{\mathbb C}(1,\chi^{(a-1)}\circ\alpha_a^\vee)}
 =\prod_{\beta\in\Phi(P)}
   \frac{L_{\mathbb C}(0,\chi\circ\beta^\vee)}
        {L_{\mathbb C}(1,\chi\circ\beta^\vee)}.
\end{equation}

To group \eqref{eq:main-L-product} according to the subsets $\Phi_j(P)$ of $\Phi(P)$, recall $\chi=\chi_{\ulm}\cdot(\delta_P^{-1/2}|_T)$ from \Cref{subsec:mt3-main}. For $\beta\in\Phi_j(P)$, equation~\eqref{eq:parabolic-root-twist} therefore gives
\[
 \chi\circ\beta^\vee
   =(\chi_{\ulm}\circ\beta^\vee)
       |\cdot|_{\mathbb C}^{js_0}.
\]
The twist $|\cdot|_{\mathbb C}^{js_0}$ shifts the $L$-function argument by $js_0$. Using the product formula \eqref{eq:L_j} and the shift of the argument in \eqref{eq:induced-L-factor-shift}, we obtain
\begin{equation}
\label{eq:main-first-factor}
 \prod_{\beta\in\Phi(P)}
       \frac{L_{\mathbb C}(0,\chi\circ\beta^\vee)}
            {L_{\mathbb C}(1,\chi\circ\beta^\vee)}
 =\prod_{j=1}^m\frac{L_{\mathbb C}(js_0,\tau,r_j)}
                   {L_{\mathbb C}(1+js_0,\tau,r_j)}.
\end{equation}
Each $a_{\alpha_a,\chi^{(a-1)}}$ is nonzero and rational by \eqref{eq:main-rank-one-input}. Their product gives the constant $a\in\mathbb Q^\times$ in \eqref{eq:main-arithmetic-scalar}.
\end{proof}

\bigskip

\noindent
{\it Acknowledgements:} A.R. is grateful to Yiannis Sakellaridis and to Johns Hopkins University for their warm hospitality during several visits, 
when this work was initiated. During the course of the collaboration, X.Z. gratefully acknowledges the hospitality of Fordham University, 
and both authors thank the Institute for Advanced Study for providing an excellent working environment in which they met frequently. 
A.R. was partially supported by a Simons Fellowship for 2026--2027, during which this article was finalized.

\bigskip

\noindent
{\it Declaration on use of AI:} ChatGPT was used solely to proofread the manuscript at the final stage. 
No other use of AI was made. The authors are solely responsible for the content, results, and conclusions of this manuscript.

\end{document}